\documentclass{amsart}
\usepackage{amsmath,amssymb}
\usepackage{latexsym}

\newtheorem{theorem}{Theorem}[section]
\newtheorem{lemma}[theorem]{Lemma}
\newtheorem{proposition}[theorem]{Proposition}
\newtheorem{corollary}[theorem]{Corollary}

\theoremstyle{definition}

\theoremstyle{remark}
\newtheorem{remark}[theorem]{Remark}

\numberwithin{equation}{section}

\begin{document}

\title{Two Dimensional NLS Equation With Random Radial Data}

\author{Yu Deng}
\address{Massachusetts Institute of Technology, Cambridge, MA 02139}
\email{yudeng@mit.edu}

\thanks{}

%    General info

\date{}

\dedicatory{}

\begin{abstract}
In this paper we study radial solutions of certain two-dimensional nonlinear Schr\"{o}dinger equation with harmonic potential, which is supercritical with respect to the initial data. By combining the nonlinear smoothing effect of Schr\"{o}dinger equation with $L^{p}$ estimates of Laguerre functions, we are able to prove an almost sure global well-posedness result and the invariance of the Gibbs measure. We also discuss an application to the NLS equation without harmonic potential.
\end{abstract}

\maketitle

%\section*{This is an unnumbered first-level section head}
%This is an example of an unnumbered first-level heading.

%% The correct journal style for \specialsection is all uppercase; a known bug
%% in amsart.cls prevents this, so input must be uppercase until it is fixed.
%\specialsection*{This is a Special Section Head}
%\specialsection*{THIS IS A SPECIAL SECTION HEAD}
%This is an example of a special section head%
%%%%%%%%%%%%%%%%%%%%%%%%%%%%%%%%%%%%%%%%%%%%%%%%%%%%%%%%%%%%%%%%%%%%%%%%
%\footnote{Here is an example of a footnote. Notice that this footnote
%text is running on so that it can stand as an example of how a footnote
%with separate paragraphs should be written.
%\par
%And here is the beginning of the second paragraph.}%
%%%%%%%%%%%%%%%%%%%%%%%%%%%%%%%%%%%%%%%%%%%%%%%%%%%%%%%%%%%%%%%%%%%%%%%%
\setlength{\parskip}{0.8ex}
\section{Introduction}
In Burq-Thomann-Tzvetkov \cite{BTT10}, the authors studied the nonlinear Schr\"{o}dinger (NLS) equation on $\mathbb{R}\times\mathbb{R}^{d}$ with harmonic potential\begin{equation}\label{nls2}
\mathrm{i}\partial_{t}u+(\Delta-|x|^{2})u=\pm|u|^{p-1}u,
\end{equation} where the space dimension was one. The purpose of this paper is to extend their results to two space dimensions. We will prove almost sure global well-posedness with respect to a Gaussian measure supported on $\cap_{\delta>0}\mathcal{H}^{-\delta}$ (see Section \ref{bour} for definition), and construct the Gibbs measure, absolutely continuous with respect to this Gaussian, which we prove to be invariant.

We also study the NLS equation on $\mathbb{R}\times\mathbb{R}^{d}$ without harmonic potential\begin{equation}\label{nls1}
\mathrm{i}\partial_{t}u+\Delta u=\pm |u|^{p-1}u.
\end{equation} In \cite{BTT10}, it was noticed that using an explicit transform (referred to as the lens transform in Tao \cite{Ta09}), we can obtain local and global well-posedness results of equation (\ref{nls1}) from the corresponding results of (\ref{nls2}). This issue is also pursued here.

Like earlier papers on random data theory of NLS equations in two or more dimensions (with the exception of Bourgain \cite{Bo96}), we only consider radial solutions. In the defocusing case in two dimensions, we have almost sure global well-posedness and measure invariance for (\ref{nls2}), and almost sure global well-posedness and scattering for (\ref{nls1}), when $p\geq 3$ is an odd integer; in the focusing case, we have the same results only for (\ref{nls2}), when $1<p<3$.
\subsection{NLS equation and probabilistic methods}

The nonlinear Schr\"{o}dinger equation (\ref{nls1}) and its periodic variant (which is solved on $\mathbb{R}\times\mathbb{T}^{d}$) have been extensively studied over the last several decades. Beginning from Lebowitz-Rose-Speer \cite{LRS89} and Bourgain \cite{Bo94}, \cite{Bo96}, it has been observed that low regularity local and global solutions to (\ref{nls1}) on $\mathbb{R}\times\mathbb{T}^{d}$ can be obtained via randomization of initial data and construction of Gibbs measure. This idea was later developed in a number of papers, for instance Burq-Tzvetkov \cite{BT08}, \cite{BT08ii}, Nahmod-Oh-Bellet-Staffilani \cite{NORS10}, Oh \cite{Oh09}, \cite{Oh09ii}, Thomann-Tzvetkov \cite{TT10}, Tzvetkov \cite{Tz06}, \cite{Tz08}. In Burq-Thomann-Tzvetkov \cite{BTT10}, the above-mentioned method was first used to study the equation (\ref{nls2}).

There are three reasons why (\ref{nls2}) is worth studying. First, the spectrum of the harmonic oscillator $\mathbf{H}=-\Delta+|x|^{2}$ is discrete, so (\ref{nls2}) can be approximated by ODEs, and the current techniques of constructing Gibbs measure apply at least formally. Second, (\ref{nls2}) is solved on $\mathbb{R}\times\mathbb{R}^{d}$ where the space domain is non-compact, while the proceeding works usually involve a compact manifold. Also (\ref{nls2}) is related to (\ref{nls1}) via the lens transform, so results about (\ref{nls2}) may shed some light on the study of (\ref{nls1}), where probabilistic methods have not yet entered. Finally, (\ref{nls2}) also arise naturally from the theory of Bose-Einstein condensates, as noted in \cite{BTT10}.

The major difficulty in the study of (\ref{nls2}) is that the support of the Gaussian part of the Gibbs measure contains functions with very low regularity. With radial assumption the typical element in the support of the Gibbs measure belongs to $\cap_{\delta>0}\mathcal{H}^{-\delta}$ but not $L^{2}$; without radial assumption the typical element does not even belong to $\mathcal{H}^{1-d}$ (the spaces $\mathcal{H}^{\sigma}$, as defined in Section \ref{bour}, are Sobolev spaces associated to $\mathbf{H}$; see Section \ref{gib} for more details). A consequence of this is that we cannot expect even local well-posedness in the deterministic sense for such low-regularity initial data. In fact in \cite{Tho08} local ill-posedness for $\mathcal{H}^{\sigma}$ initial data was shown\footnote[1]{The counterexample constructed in \cite{Tho08} was for (\ref{nls1}), but it could be easily adapted to (\ref{nls2}) as noted in \cite{Tho09}; also one can check the proof there that the initial data could be made radial.}, provided $\sigma<\sigma_{c}:=\frac{d}{2}-\frac{2}{p-1}$. In particular, we have $\sigma_{c}\to 1$ as $p\to\infty$ for the two-dimensional defocusing equation, thus deterministic local well-posedness fails completely for regularity below $L^{2}$.

In \cite{BTT10}, the problem was resolved by a probabilistic improvement of (weighted) Strichartz estimate, and it was shown that $\mathbf{H}^{\delta}e^{-\mathrm{i}t\mathbf{H}}f(\omega)$ almost surely belongs to some weighted Lebesgue space for $\delta<\frac{1}{2}$. Since $\sigma_{c}<\frac{1}{2}$ in one dimension, local well-posedness in this space could be proved. In two dimensions, however, it will be shown in Appendix \ref{count} that the distribution $\mathbf{H}^{\frac{\sigma}{2}}f(\omega)$ is almost surely not a locally integrable function (thus cannot belong to any weighted space) when $\sigma\geq\frac{1}{2}$. Since $\frac{1}{2}$ fails to reach the $\sigma_{c}$ threshold when $p$ is large, we have to use different tools to get local well-posedness. Fortunately, the nonlinear smoothing effect of the NLS equation provides such a tool. To fully exploit this effect, we will work in $\mathcal{X}^{\sigma,b}$ spaces (see Section \ref{bour} for definition) and use multilinear eigenfunction estimates. This requires $p$ to be an odd integer; but we believe that by more delicate discussions we can remove this restriction and allow for all $1<p<\infty$.

When there is no radial assumption, the support of the  Gaussian will have so low regularity that we can not even define the Gibbs measure. It would be possible to use alternative Gaussians to get local results, but then we do not have an invariant measure, so global results still seem out of reach. One possible way is to combine the probabilistic local result with the high-low analysis of Bourgain or the $I$-method of Colliander-Keel-Staffilani-Takaoka-Tao. For a progress in this direction, see Colliander-Oh \cite{CO10}.

Finally, as we mentioned above, the study of (\ref{nls2}) is closely related to the study of (\ref{nls1}). The result we obtain for (\ref{nls1}) (see Theorem \ref{main2} below) is an almost sure global well-posedness and scattering result with supercritical initial data (the critical index of (\ref{nls1}) is $\frac{d}{2}-\frac{2}{p-1}\to 1$ as $p\to\infty$ in two dimensions, while the initial data is below $L^{2}$), but due to the use of the lens transform, our result is unsatisfactory in the sense that (i) the space in which uniqueness holds is rather complex, and (ii) the Gaussian measure in Theorem \ref{main2} does not arise naturally from (\ref{nls1}), and we do not know how to construct the Gibbs measure of (\ref{nls1}). This may be an interesting problem for further study.

\subsection{Notations and priliminaries}\label{bour} From now on we assume the spacial dimension $d=2$, and all the functions we consider are radial. Define the Hermite operator $\mathbf{H}=-\Delta+|x|^{2}$. It has a complete series of real $L_{\mathrm{rad}}^{2}$ eigenfunctions\begin{equation}\label{eigen}
e_{k}(x)=\frac{1}{\sqrt{\pi}}\mathcal{L}_{k}^{0}(|x|^{2}),\,\,\,\,\,\,k\geq 0
\end{equation} with eigenvalue $4k+2$. Here $\mathcal{L}_{k}^{0}$ are Laguerre functions\[\mathcal{L}_{k}^{0}(z)=\frac{e^{\frac{z}{2}}}{k!}\frac{\mathrm{d}^{k}}{\mathrm{d}z^{k}}(z^{k}e^{-z}).\] Concerning these functions we have the basic pointwise estimates\begin{equation}\label{ctl}
\mathcal{L}_{k}^{\alpha}(z)\leq\left\{
\begin{array}{llll}
C & 0\leq z\leq\frac{1}{\nu}\\
C(z\nu)^{-\frac{1}{4}} & \frac{1}{\nu}\leq z\leq \frac{\nu}{2}\\
C\nu^{-\frac{1}{4}}(\nu^{\frac{1}{3}}+|\nu-z|)^{-\frac{1}{4}} & \frac{\nu}{2}\leq z\leq \frac{3\nu}{2}\\
Ce^{-cz} & z\geq\frac{3\nu}{2}
\end{array}
\right.\end{equation} where $\nu=4k+2$, $C$ and $c$ (possibly with subscripts) are positive constants varying from line to line, and will be used in this manner throughout this paper. For an introduction to Laguerre functions, see Szeg\"{o} \cite{Sz75} or Thangavelu \cite{Th93}, Chapter 1. The proof of (\ref{ctl}) is also contained in \cite{Er60} or \cite{AW65}.

For $\sigma\in\mathbb{R}$, $1\leq p\leq \infty$, we define the Sobolev spaces associated to $\mathbf{H}$:\begin{equation}\mathcal{W}_{\mathrm{rad}}^{\sigma,p}=\big\{u\in\mathcal{S}_{\mathrm{rad}}':\|u\|_{\mathcal{W}^{\sigma,p}}=\big\|\mathbf{H}^{\frac{\sigma}{2}}u\big\|_{L^{p}}<\infty\big\}.\end{equation} We also write $\mathcal{W}_{\mathrm{rad}}^{\sigma,2}=\mathcal{H}_{\mathrm{rad}}^{\sigma}$.

We also define a class of spacetime Hilbert spaces associated to $\mathbf{H}$, as\begin{equation}\label{bgn}
\mathcal{X}_{\mathrm{rad}}^{\sigma,b}=\big\{u\in\mathcal{S}_{\mathrm{rad}}'(\mathbb{R}\times\mathbb{R}^{2}):\|u\|_{\mathcal{X}^{\sigma,b}}=\big\|\mathbf{H}^{\frac{\sigma}{2}}\langle\mathrm{i}\partial_{t}-\mathbf{H}\rangle^{b}u\big\|_{L_{t,x}^{2}}<\infty\big\},\end{equation} or use radial Hermite expansion and Fourier transform to write\begin{equation}\nonumber\|u\|_{\mathcal{X}^{\sigma,b}}^{2}=\bigg(\sum_{k=0}^{\infty}(4k+2)^{\sigma}\int_{\mathbb{R}}\big(1+(\tau+4k+2)^{2}\big)^{b}\big|\mathcal{F}_{t}\langle u,e_{k}\rangle(\tau)\big|^{2}\,\mathrm{d}\tau\bigg)^{\frac{1}{2}},\end{equation} where as usual $\langle t\rangle=(|t|^{2}+1)^{\frac{1}{2}}$, $\mathcal{F}_{t}$ denote the Fourier transform $(2\pi)^{-\frac{1}{2}}\int_{\mathbb{R}}e^{-\mathrm{i}\tau t}f(t)\,\mathrm{d}t$ in $t$, and $\langle f,g\rangle$ denotes the $L^{2}(\mathbb{R}^{n})$ inner product of $f$ and $g$. For an interval $I$ we define a localized version of this space by \begin{equation}\label{localbgn}\|u\|_{\mathcal{X}^{\sigma,b,I}}=\inf\big\{\|v\|_{\mathcal{X}^{\sigma,b}}:v(t)=u(t),\,\,t\in I\big\},\end{equation} and denote it by $\mathcal{X}_{\mathrm{rad}}^{\sigma,b,I}$. When $I=[-T,T]$, we simply write $\mathcal{X}_{\mathrm{rad}}^{\sigma,b,T}$. Since all the functions will be radial, the ``rad'' subscript will be dropped from now on. Trivially $\mathcal{X}^{\sigma,b,I}$ is a separable Banach space (simply restrict a countable dense subset of $\mathcal{X}^{\sigma,b}$ to $I$).

We fix a smooth, non-increasing function $\eta$ such that $1=\eta(1)\geq \eta(x)\geq \eta(2)=0$ for all $x$. Using this cutoff, we define Littlewood-Paley projections\begin{equation}\label{litp}\Delta_{N}=\eta\big(\frac{2\mathbf{H}}{N^{2}}\big)-\eta\big(\frac{4\mathbf{H}}{N^{2}}\big)\end{equation} for dyadic $N$. Then $\Delta_{N}=0$ for $N\leq1$, since the first eigenvalue of $\mathbf{H}$ is $2$. Thus whenever we talk about $\Delta_{N}$, we always assume $N\geq 2$.

We shall use $\#M$ to denote the cardinal of a finite set $M$, $|E|$ to denote the Lebesgue measure of a subset set $E$ of a Euclidean space, $A \lesssim B$ to denote $A \leq CB$, and define $\gtrsim$ and $\sim$ similarly. The constants $C_{j}$ and $c_{j}$ will also be used freely, as indicated above. All these constants will ultimately depend on the only parameter $p$ in (\ref{nls2}) and (\ref{nls1}). Finally, we define the finite dimensional subspace $V_{k}$ to be the span of $\{e_{j}\}_{0\leq j\leq k}$. For a function $g$ on $\mathbb{R}^{2}$ or $I\times\mathbb{R}^{2}$, where $I$ is an interval, we define $g_{k}^{\circ}$ and $g_{k}^{\perp}$ to be the projection of $g$ on $V_{k}$ and $V_{k}^{\perp}$.

\subsection{Statement of main results, and plan for this paper} Fix a probability space $(\Omega,\Sigma,\mathbb{P})$ with a sequence of independent normalized complex Gaussians $\{g_{k}\}$ on $\Omega$ (which has density $\frac{1}{\pi}e^{-|z|^{2}}\mathrm{d}x\mathrm{d}y$, thus $g_{k}$ has mean $0$ and variance $1$), so that $\omega\mapsto(g_{k}(\omega))_{k\geq 0}$ is injective, and the series\begin{equation}\label{random}
f(\omega)=\sum_{k=0}^{\infty}\frac{1}{\sqrt{4k+2}}g_{k}(\omega)e_{k}\end{equation} converge\footnote[2]{For example, we may take the usual product space $\mathbb{C}^{\infty}$ equipped with the product of complex Gaussian measures, and coordinate functions $g_{j}$, and choose the (full-measure) subset where $|g_{k}(\omega)|=O(\langle k\rangle ^{10})$ as $\Omega$, this can easily guarantee the convergence of (\ref{random}).} in $\mathcal{S}'(\mathbb{R}^{2})$, for all $\omega\in\Omega$. Then $f=f(\omega)$ is an $\mathcal{S}'(\mathbb{R}^{2})$-valued random variable, and is a bijection between $\Omega$  and its range. Our main results can then be stated as follows.
\begin{theorem}\label{main}Consider the Cauchy problem \begin{equation}\label{nls22}
\left\{
\begin{array}{ll}
\mathrm{i}\partial_{t}u+(\Delta-|x|^{2})u=\pm |u|^{p-1}u\\
u(0)=f(\omega)
\end{array}
\right.
\end{equation} and separate two cases: the sign is $-$ and $1<p<3$, or the sign is $+$ and $p\geq 3$ is an odd integer. In the former let $\sigma=0$, in the latter let $0<\sigma<1$ be sufficiently close to $1$, depending on $p$. In both cases let $1>b>\frac{1}{2}$ be sufficiently close to $\frac{1}{2}$, depending on $\sigma$ and $p$.

Then a.s. in $\mathbb{P}$, we have a unique global (strong) solution $u$ in the affine space\begin{equation}\mathcal{Y}=e^{-\mathrm{i}t\mathbf{H}}f(\omega)+\bigcap_{T>0}\mathcal{X}^{\sigma,b,T},\end{equation} and we have continuous embeddings \[\mathcal{Y}\subset e^{-\mathrm{i}t\mathbf{H}}f(\omega)+\mathcal{C}\big(\mathbb{R},\mathcal{H}^{\sigma}(\mathbb{R}^{2})\big)\subset\mathcal{C}\big(\mathbb{R},\cap_{\delta>0}\mathcal{H}^{-\delta}(\mathbb{R}^{2})\big).\] We also have a Gibbs measure on $\mathcal{S}'(\mathbb{R}^{2})$, which is absolutely continuous with respect to the push forward of $\mathbb{P}$ under $f$, and is invariant under the flow defined by (\ref{nls22}).
\end{theorem}
\begin{theorem}\label{main2}Let $\sigma$ and $b$ be as in Theorem \ref{main}. Consider the (defocusing) Cauchy problem \begin{equation}\label{main000}
\left\{
\begin{array}{ll}
\mathrm{i}\partial_{t}u+\Delta u=|u|^{p-1}u\\
u(0)=f(\omega)
\end{array}
\right.
\end{equation}with $p\geq 3$ an odd integer. Then a.s. in $\mathbb{P}$, we have a global (strong) solution $u$ in the affine space\begin{equation}\label{moresmooth}\mathcal{Z}=e^{\mathrm{i}t\Delta}f(\omega)+\bigcap_{T>0}X^{\sigma,b,T},\end{equation} and we have a continuous embedding \[\mathcal{Z}\subset e^{\mathrm{i}t\Delta}f(\omega)+\mathcal{C}\big(\mathbb{R},H^{\sigma}(\mathbb{R}^{2})\big).\]Here $X^{\sigma,b,T}$ is defined in the same way as in (\ref{bgn}) and (\ref{localbgn}), but with $\mathbf{H}$ replaced by $-\Delta$. We also have an appropriate affine subspace $\mathcal{Z}'$ of $\mathcal{Z}$ containing the solution $u$, in which uniqueness holds. Finally we have a scattering result: there exist functions $g_{\pm}\in H^{\sigma}$ so that\begin{equation}\label{scatter}\lim_{t\to\pm\infty}\|u-e^{\mathrm{i}t\Delta}(f(\omega)+g_{\pm})\|_{H^{\sigma}}=0.\end{equation}
\end{theorem}
The rest of this paper is devoted to the proof of Theorems \ref{main} and \ref{main2}. In Section \ref{fct} we recall the linear Strichartz and $L^{2}$-based estimates with respect to the propagator $e^{-\mathrm{i}t\mathbf{H}}$. We will rely on the functional calculus of $\mathbf{H}$ (thus the results hold for more general Schr\"{o}dinger operators, though we do not discuss this here). Some results in this section are standard and can be found in, say, \cite{CO10} or \cite{Ta05}. In Section \ref{gib}, we prove some large deviation bounds for Gaussian random variables, and use these to construct the Gibbs measure of (\ref{nls2}). In Section \ref{nonlinear}, which is the core of this paper, we use a Littlewood-Paley decomposition and hypercontractivity of Gaussians to prove a multilinear estimate in $\mathcal{X}^{\sigma,b}$ spaces, which shows the nonlinear smoothing effect. In Section \ref{lwp}, we put these estimates together to develop a local Cauchy theory. Then in Section \ref{gwp} we extend this to a global well-posedness result by exploiting the invariance of truncated Gibbs measure under the flow of approximating ODEs. In Section \ref{lens0} we introduce the lens transform and convert the result on (\ref{nls2}) to one on (\ref{nls1}), proving Theorem \ref{main2}. In Section \ref{inv}, we show the invariance of the Gibbs measure, completing the proof of Theorem \ref{main}. Finally in Appendix \ref{count}, we discuss the typical regularity (in terms of $\mathbf{H}$) on the support of the Gibbs measure.
\subsection{Acknowledgements}
I would like to thank Gigliola Staffilani, who introduced me to this problem, for her guidance and constant encouragements; I would also like to thank Nikolay Tzvetkov, for his helpful comments on the first draft of this paper.

\section{Functional calculus and Strichartz estimates}\label{fct}

We begin with the following kernel estimate about the harmonic oscillator $\mathbf{H}$.
\begin{proposition}\label{kernelest}
Let $\psi$ be a Schwarz function, then for $t>0$ the operator $\psi(t\mathbf{H})$ is an integral operator with kernel $K_{t}(x,y)$ where\begin{equation}\label{ineq}|K_{t}(x,y)|\lesssim t^{-1}(1+t^{-\frac{1}{2}}|x-y|)^{-N}.\end{equation} The implicit constants in $\lesssim$ depends only on $N$ and $\psi$. In particular, these operators $K_{t}$ are bounded uniformly in $t$ on $\mathcal{W}^{\sigma,p}$ for all $\sigma\in\mathbb{R}$ and $1\leq p\leq\infty$.
\end{proposition}
\begin{proof} It was proved in \cite{Dz98}, Corollary 3.14 that, for any fixed $N$, the inequality (\ref{ineq}) holds, provided \begin{equation}\psi\in\mathcal{S}_{0}^{m}([0,+\infty))=\big\{\psi\in\mathcal{S}([0,+\infty)):\psi^{(k)}(0)=0,\,\,\,\,0\leq j\leq m\big\},\end{equation} where $m$ is large enough depending on $N$ (actually the same result was proved for any Schr\"{o}dinger operator with nonnegative polynomial potential). On the other hand, when $\psi(z)=e^{-\sigma z}$ with $\sigma>0$, we have from Mehler's formula that
\begin{equation}\label{meh}K_{t}(x,y)=\frac{e^{-2\sigma t}}{\pi(1-e^{-4\sigma t})}\exp\bigg(-\frac{1}{2}\frac{1+e^{-4\sigma t}}{1-e^{-4\sigma t}}(|x|^{2}+|y|^{2})+\frac{2e^{-2\sigma t}}{1-e^{-4\sigma t}}x\cdot y\bigg).\end{equation} Writing $2\sigma t=\delta$, we know\begin{equation}\nonumber-\frac{1}{2}\frac{1+e^{-2\delta}}{1-e^{-2\delta}}(|x|^{2}+|y|^{2})+\frac{2e^{-\delta}}{1-e^{-2\delta}}x\cdot y\leq -\frac{c}{\delta}|x-y|^{2},\end{equation} thus the kernel satisfies\begin{equation}0\leq K_{t}(x,y)\leq \frac{c_{1}}{\delta}e^{-\frac{c_{2}}{\delta}|x-y|^{2}}\lesssim t^{-1}(1+t^{-\frac{1}{2}}|x-y|)^{-N}\end{equation}
 for any $N$. Now for any fixed $m$, any function $f\in\mathcal{S}([0,+\infty))$ can be written as\begin{equation}f(z)=f_{0}(z)+\sum_{j=1}^{l}c_{j}e^{-\sigma_{j}z},\end{equation} where $f_{0}\in\mathcal{S}_{0}^{m}([0,+\infty))$ and $\sigma_{j}>0$. Combining the two results above, we have proved (\ref{ineq}). The uniform boundedness now follows from (\ref{ineq}), Schur's test, and commutativity of $\psi(t\mathbf{H})$ and $\mathbf{H}^{\frac{\sigma}{2}}$.
\end{proof}
\begin{remark}
The constants in Proposition \ref{kernelest} certainly depend on $\psi$ and the Lebesgue or Sobolev exponents, but this dependence can be safely ignored in that throughout this paper, we only use a finite number of fixed cutoff functions $\psi$, and a finite number of fixed exponents.
\end{remark}
\begin{corollary}\label{sobbb}
Suppose $1\leq p\leq\infty$, $\sigma_{1,2}\in\mathbb{R}$, $R>0$ and $g$ is a function.

(1) If $\sigma_{1}\geq\sigma_{2}$, and $\langle g,e_{k}\rangle\neq0$ only if $4k+2\gtrsim R^{2}$ (for example, when $g=\sum_{N>R}\Delta_{N}h$ for some $h$), then $\|g\|_{\mathcal{W}^{\sigma_{1},p}}\gtrsim R^{\sigma_{1}-\sigma_{2}}\|g\|_{\mathcal{W}^{\sigma_{2},p}}$.

(2) If $\sigma_{1}\leq\sigma_{2}$, and $\langle g,e_{k}\rangle\neq0$ only if $4k+2\lesssim R^{2}$ (for example, when $g=\sum_{N\leq R}\Delta_{N}h$ for some $h$), then $\|g\|_{\mathcal{W}^{\sigma_{1},p}}\gtrsim R^{\sigma_{1}-\sigma_{2}}\|g\|_{\mathcal{W}^{\sigma_{2},p}}$.

(3) If $\langle g,e_{k}\rangle\neq0$ only if $4k+2\sim R^{2}$ (for example, when $R=N$ is dyadic and $g=\Delta_{N}h$ for some $h$), then $\|g\|_{\mathcal{W}^{\sigma_{1},p}}\sim R^{\sigma_{1}-\sigma_{2}}\|g\|_{\mathcal{W}^{\sigma_{2},p}}$.

(4) All the operators $\sum_{N>R}\Delta_{N}$, $\sum_{N\leq R}\Delta_{N}$ and $\Delta_{N}$ are uniformly bounded from $\mathcal{W}^{\sigma_{1},p}$ to itself.
\end{corollary}
\begin{proof}
First (4) is obvious, since $\sum_{N<R}\Delta_{N}=\eta(t\mathbf{H})$ and $\Delta_{N}=\eta(t'\mathbf{H})-\eta(2t'\mathbf{H})$ for some $t,t'$, and $\sum_{N>R}\Delta_{N}=\mathrm{Id}-\sum_{N\leq R}\Delta_{N}$. Also it is clear that (1) and (2) implies (3). In proving these we may assume $\min\{\sigma_{1},\sigma_{2}\}=0$, since $\mathbf{H}^{\frac{\sigma}{2}}g$ satisfies the same properties as $g$.

To prove (1), choose a smooth cutoff $\psi_{1}$ that equals $1$ for $x\gtrsim 1$ and equals $0$ for very small $x$, then in (1) we have $g=\psi_{1}(R^{-2}\mathbf{H})g$. Therefore we need to prove that \begin{equation}\mathbf{H}^{-\frac{\sigma}{2}}R^{\sigma}\psi_{1}(R^{-2}\mathbf{H})=\sum_{k\geq 0}2^{-\frac{k\sigma}{2}}\psi_{2}(2^{-k}R^{-2}\mathbf{H})\end{equation} is uniformly bounded on $L^{p}$ for $\sigma>0$, where $\psi_{2}(x)=x^{-\frac{\sigma}{2}}(\psi_{1}(x)-\psi_{1}(2^{-1}x))$ is a fixed smooth compactly supported function. Using (\ref{ineq}), we can estimate the kernel $K(x,y)$ of $\mathbf{H}^{-\frac{\sigma}{2}}R^{\sigma}\psi_{1}(R^{-2}\mathbf{H})$ as\begin{equation}|K(x,y)|\lesssim \sum_{k\geq 0}2^{-\frac{k\sigma}{2}}2^{k}R^{2}\langle 2^{\frac{k}{2}}R|x-y|\rangle^{-N}= R^{2}\psi_{3}( R|x-y|),\end{equation} where\[\psi_{3}(x)=\sum_{k\geq 0}2^{(1-\frac{\sigma}{2})k}\langle 2^{\frac{k}{2}}x\rangle^{-N}\lesssim(1+ |x|^{\sigma-2})\langle x\rangle^{-N}.\] The last inequality is easily verified by considering $|x|\geq 1$ and $|x|<1$ separately. Therefore by Schur's test we have proved the uniform boundedness of the operator, thus proving (1). The proof of (2) is similar and is left as an exercise.
\end{proof}
To get Sobolev and product estimates, we next need a lemma.
\begin{lemma}\label{sepr}
For all $1<p<\infty$ and $\sigma>0$, we have\begin{equation}\|g\|_{\mathcal{W}^{\sigma,p}}\sim\|\langle\nabla\rangle^{\sigma} g\|_{L^{p}}+\|\langle x\rangle^{\sigma}g\|_{L^{p}}.\end{equation} In particular we have $\|g\|_{\mathcal{W}^{\sigma_{1},p}}\lesssim\|g\|_{\mathcal{W}^{\sigma_{2},p}}$ for $\sigma_{1}\leq\sigma_{2}$.
\end{lemma}
\begin{proof}
See \cite{DG09}. There the same result was proved for any Schr\"{o}dinger operator with nonnegative polynomial potential (note the latter inequality also follows from Corollary \ref{sobbb}).
\end{proof}
\begin{proposition}\label{sobolev}
%
%
%
%Sobolev inequality
%
%
%
We have the following estimates:\begin{equation}\label{sobolev00}\|g\|_{\mathcal{W}^{\sigma_{1},q}}\lesssim\|g\|_{\mathcal{W}^{\sigma_{2},q'}},\end{equation} if $1<q,q'<\infty$ and $\sigma_{2}-\sigma_{1}\geq2(\frac{1}{q'}-\frac{1}{q})\geq 0$.\begin{equation}\label{leib}\bigg\|\prod_{j=1}^{k}g_{j}\bigg\|_{\mathcal{W}^{\sigma,p}}\lesssim\sum_{j=1}^{k}\|g_{j}\|_{\mathcal{W}^{\sigma,q_{j}}}\prod_{i\neq j}\|g_{i}\|_{L^{q_{i}}},\end{equation}if $\sigma>0$ and $1<p,q_{j}<\infty$ with $1\leq j\leq k$ and $\sum_{j=1}^{k}\frac{1}{q_{j}}=\frac{1}{p}$.
\end{proposition}
\begin{proof}

In considering (1) we may assume $\sigma_{1}=0$, and the inequality follows immediately from Lemma \ref{sepr} and the usual Sobolev inequality.

As for (2), if the $\mathcal{W}^{\sigma,p}$ norm is replaced by the usual Sobolev $W^{\sigma,p}$ norm, then (\ref{leib}) is a well-known result in Fourier analysis (for $k=2$, but the general case easily follows from induction). Now using Lemma \ref{sepr}, we only need to show \begin{equation}\nonumber\|\langle x\rangle^{\sigma}g_{1}\cdots g_{k}\|_{L^{p}}\lesssim \|\langle x\rangle ^{\sigma}g_{1}\|_{L^{q_{1}}}\prod_{j=2}^{k}\|g_{j}\|_{L^{q_{j}}},\end{equation} which is simply H\"{o}lder's inequality.
\end{proof}

Before proving Strichartz and other estimates, we need a lemma, which gives a representation formula of $\mathcal{X}^{\sigma,b}$ functions.
\begin{lemma}\label{represent}
Suppose $\sigma,b\in\mathbb{R}$. Then for every $u$, if $\|u\|_{\mathcal{X}^{\sigma,b}}\lesssim 1$, we have\begin{equation}u(t,x)=\int_{\mathbb{R}}\phi(\lambda)e^{\mathrm{i}\lambda t}\sum_{k}a_{\lambda}(k)e^{-\mathrm{i}(4k+2)t}e_{k}(x)\,\mathrm{d}\lambda,\end{equation} where $\sum_{k}(4k+2)^{\sigma}|a_{\lambda}(k)|^{2}= 1$ for all $\lambda\in\mathbb{R}$. Furthermore, if $b>\frac{1}{2}$, then we also have $\int_{\mathbb{R}}|\phi(\lambda)|\,\mathrm{d}\lambda\lesssim 1$; if $b<\frac{1}{2}$ and $\mathcal{F}_{t}\langle u,e_{k}\rangle(\lambda)$ is supported in $\{|\lambda+4k+2|\leq K\}$ for each $k$, where $K\gtrsim 1$, then we also have $\int_{\mathbb{R}}|\phi(\lambda)|\,\mathrm{d}\lambda\lesssim K^{\frac{1}{2}-b}$.
\end{lemma}
\begin{proof} Using radial Hermite expansion and Fourier transform, we can write\setlength\arraycolsep{2pt}
\begin{eqnarray}u(t,x)&=&(2\pi)^{-\frac{1}{2}}\sum_{k}\int_{\mathbb{R}}\mathcal{F}_{t}\langle u,e_{k}\rangle(\tau)e^{\mathrm{i}t\tau}e_{k}(x)\,\mathrm{d}\tau\nonumber\\
&=&(2\pi)^{-\frac{1}{2}}\sum_{k}\int_{\mathbb{R}}\mathcal{F}_{t}\langle u,e_{k}\rangle(\lambda-4k-2)e^{-\mathrm{i}(4k-2)t}e_{k}(x)e^{\mathrm{i}t\lambda}\,\mathrm{d}\lambda,\nonumber
\end{eqnarray}so we may choose\begin{equation}a_{\lambda}(k)=(\mathcal{F}_{t}\langle u,e_{k}\rangle)(\lambda-4k-2)\cdot\bigg(\sum_{l}(4l+2)^{\sigma}|\mathcal{F}_{t}\langle u,e_{l}\rangle(\lambda-4l-2)|^{2}\bigg)^{-\frac{1}{2}}\end{equation} and
\begin{equation}\phi(\lambda)=(2\pi)^{-\frac{1}{2}}\bigg(\sum_{l}(4l+2)^{\sigma}|\mathcal{F}_{t}\langle u,e_{l}\rangle(\lambda-4l-2)|^{2}\bigg)^{\frac{1}{2}}.\end{equation} Then we clearly have $\sum_{k}(4k+2)^{\sigma}|a_{\lambda}(k)|^{2}=1$ for each $\lambda$, and from the definition of $\mathcal{X}^{\sigma,b}$ norm we see \begin{equation}\int_{\mathbb{R}}\langle\lambda\rangle^{2b}|\phi(\lambda)|^{2}\,\mathrm{d}\lambda=\frac{1}{2\pi}\|u\|_{\mathcal{X}^{\sigma,b}}^{2}\lesssim 1.\end{equation}

If $b>\frac{1}{2}$, then $\langle\lambda\rangle^{-b}\in L^{2}(\mathbb{R})$, and it follows from  Cauchy-Schwartz that $\|\phi\|_{L^{1}}\leq \|\langle\lambda\rangle^{b}\phi\|_{L^{2}}\cdot\|\langle\lambda\rangle^{-b}\|_{L^{2}}\lesssim 1$. If instead $b<\frac{1}{2}$ and $u$ satisfies the support condition, then $\phi(\lambda)=0$ if $|\lambda|>K$. Again from Cauchy-Schwartz,\begin{equation}\|\phi\|_{L^{1}}\lesssim\bigg(\int_{|\lambda|\leq K}\langle\lambda\rangle^{-2b}\,\mathrm{d}\lambda\bigg)^{\frac{1}{2}}\sim K^{\frac{1}{2}-b}.\end{equation}
\end{proof}
\begin{proposition}\label{strichartz}
%
%
%
%strichartz estimate
%
%
%
Suppose $b>\frac{1}{2}$, $\sigma_{1,2}\in\mathbb{R}$, and $1<q_{2},r_{2}<2<q,r,q_{1},r_{1}<\infty$. We have the following estimates:\begin{equation}\label{str1}\|e^{-\mathrm{i}t\mathbf{H}}g\|_{L_{t}^{r}L_{x}^{q}([-T,T]\times\mathbb{R}^{2})}\lesssim \langle T\rangle^{\frac{1}{r}}\|g\|_{L^{2}},\end{equation}
if $\frac{1}{q}+\frac{1}{r}=\frac{1}{2}$, and $g$ is defined on $\mathbb{R}^{2}$.
\begin{equation}\label{str0}\bigg\|\int_{0}^{t}e^{-\mathrm{i}(t-s)\mathbf{H}}u(s)\,\mathrm{d}s\bigg\|_{L_{t}^{r_{1}}L_{x}^{q_{1}}([-T,T]\times\mathbb{R}^{2})}\lesssim \langle T\rangle^{1+\frac{1}{r_{1}}-\frac{1}{r_{2}}}\|u\|_{L_{t}^{r_{2}}L_{x}^{q_{2}}([-T,T]\times\mathbb{R}^{2})},\end{equation}if $\frac{1}{q_{1}}+\frac{1}{r_{1}}=\frac{1}{2}$, $\frac{1}{q_{2}}+\frac{1}{r_{2}}=\frac{3}{2}$, and $u$ is defined on $[-T,T]\times\mathbb{R}^{2}$.\begin{equation}\label{str2}\|u\|_{L_{t}^{r}\mathcal{W}_{x}^{\sigma_{1},q}([-T,T]\times\mathbb{R}^{2})}\lesssim\langle T\rangle^{\frac{1}{r}}\|u\|_{\mathcal{X}^{\sigma_{2},b,T}},\end{equation}  if $\sigma_{2}-\sigma_{1}\geq 1-\frac{2}{q}-\frac{2}{r}\geq 0$, and either $u$ is defined on $[-T,T]\times\mathbb{R}^{2}$, or $u$ is defined on $\mathbb{R}\times\mathbb{R}^{2}$ and the right side is replaced by $\|u\|_{\mathcal{X}^{\sigma_{2},b}}$.\begin{equation}\label{stri2.5}\|u\|_{\mathcal{X}^{\sigma_{1},b-1,T}}\lesssim\langle T\rangle^{\frac{1}{q_{2}}-\frac{1}{2}}\|u\|_{L_{t}^{q_{2}}\mathcal{W}_{x}^{\sigma_{1},q_{2}}([-T,T]\times\mathbb{R}^{2})},\end{equation}
if $b<1$, $q_{2}>\frac{2}{2-b}$, and either $u$ is defined on $[-T,T]\times\mathbb{R}^{2}$, or $u$ is defined on $\mathbb{R}\times\mathbb{R}^{2}$, supported on $[-T,T]$, and the left side is replaced by $\|u\|_{\mathcal{X}^{\sigma_{1},b-1}}$.
\begin{equation}\label{str3}\|u\|_{\mathcal{C}([-T,T],\mathcal{H}^{\sigma_{1}}(\mathbb{R}^{2}))}\lesssim \|u\|_{\mathcal{X}^{\sigma_{1}, b, T}},\end{equation} if $u$ is defined on $[-T,T]\times\mathbb{R}^{2}$. In particular if $T\leq 1$, all the implicit constants can be taken $1$.
\end{proposition}
\begin{proof} For (\ref{str1}), since $e^{-\mathrm{i}t\mathbf{H}}$ is periodic, we may assume $T\lesssim 1$, thus $\langle T\rangle\sim 1$. In addition, by subdividing the interval $[-T,T]$, we may assume $T$ is small enough. Substituting $\sigma=\mathrm{i}$ in Mehler's fromula (\ref{meh}), we can easily see the integral kernel of $e^{-\mathrm{i}t\mathbf{H}}$ is an $L^{\infty}$ function in the space variables with norm $\lesssim |t|^{-1}$ for $|t|\lesssim T$. Now using the $TT^{\ast}$ method we reduce (\ref{str1}) to \begin{equation}\label{ttast}\bigg\|\int_{-T}^{T}e^{-\mathrm{i}(t-s)\mathbf{H}}u(s)\,\mathrm{d}s\bigg\|_{L_{t}^{r}L_{x}^{q}([-T,T]\times\mathbb{R}^{2})}\lesssim\|u\|_{L_{t}^{r'}L_{x}^{q'}([-T,T]\times\mathbb{R}^{2})}.\end{equation} Now we interpolate between $L^{2}$ conservation and the $L^{1}\to L^{\infty}$ inequality deduced from the $L^{\infty}$ bound of the integral kernel, to get $\|e^{-\mathrm{i}\delta \mathbf{H}}g\|_{L^{q}}\lesssim |\delta|^{\frac{2}{q}-1}\|u\|_{L^{q'}}$ for $|t|\lesssim T$. Using this and the usual Hardy-Littlewood-Sobolev fractional integral inequality, we immediately get (\ref{ttast}).

Now from (\ref{str1}) and duality we easily get \[\bigg\|\int_{0}^{T}e^{-\mathrm{i}(t-s)\mathbf{H}}u(s)\,\mathrm{d}s\bigg\|_{L_{t}^{r_{1}}L_{x}^{q_{1}}([-T,T]\times\mathbb{R}^{2})}\lesssim \langle T\rangle^{1+\frac{1}{r_{1}}-\frac{1}{r_{2}}}\|u\|_{L_{t}^{r_{2}}L_{x}^{q_{2}}([-T,T]\times\mathbb{R}^{2})},\]
for the exponents $q_{1},r_{1},q_{2},r_{2}$, thus from Christ-Kiselev lemma we get (\ref{str0}).

We now prove (\ref{str3}) and (\ref{str2}), under the assumption $\sigma_{2}-\sigma_{1}=1-\frac{2}{q}-\frac{2}{r}=0$. Here we may assume $\sigma_{1}=0$. By the definition of $\mathcal{X}^{0,b,T}$ we can assume that $u$ is defined for all $t\in\mathbb{R}$, and only need to prove that the left side of each equation is controlled by $\|u\|_{\mathcal{X}^{0,b}}$. We shall use $\|\cdot\|_{\mathfrak{X}}$ to denote either the norm $\langle T\rangle ^{-\frac{1}{r}}\|\cdot\|_{L_{t}^{r}L_{x}^{q}([-T,T]\times\mathbb{R}^{2})}$ or $\|\cdot\|_{\mathcal{C}([-T,T],L^{2}(\mathbb{R}^{2}))}$, and from what we just proved, we know $\|e^{-\mathrm{i}t\mathbf{H}}g\|_{\mathfrak{X}}\lesssim \|g\|_{L^{2}}$. Assume $\|u\|_{\mathcal{X}^{0,b}}\lesssim 1$, by Lemma \ref{represent} we write \begin{equation}u(t,x)=\int_{\mathbb{R}}\phi(\lambda)e^{\mathrm{i}\lambda t}\sum_{k}a_{\lambda}(k)e^{-\mathrm{i}(4k+2)t}e_{k}(x)\,\mathrm{d}\lambda\end{equation} with $\|\phi\|_{L^{1}}\lesssim 1$ and $\sum_{k}|a_{\lambda}(k)|^{2}=1$ for each $\lambda$. Then we have\begin{equation}\nonumber u=\int_{\mathbb{R}}\phi(\lambda)e^{\mathrm{i}\lambda t}e^{-\mathrm{i}t\mathbf{H}}\bigg(\sum_{k}a_{\lambda}(k)e_{k}\bigg)\,\mathrm{d}\lambda.\end{equation} From Minkowski and Cauchy-Schwartz we see \setlength\arraycolsep{2pt}
\begin{eqnarray}\|u\|_{\mathfrak{X}}&\lesssim&\|\phi\|_{L^{1}}\cdot\sup_{\lambda}\bigg\|e^{\mathrm{i}\lambda t}e^{-\mathrm{i}t\mathbf{H}}\bigg(\sum_{k}a_{\lambda}(k)e_{k}\bigg)\bigg\|_{\mathfrak{X}}\\
&\lesssim &\|\phi\|_{L^{1}}\cdot\sup_{\lambda}\bigg\|\sum_{k}a_{\lambda}(k)e_{k}\bigg\|_{L^{2}}\lesssim 1\nonumber,
\end{eqnarray}
proving (\ref{str3}) and this special case of (\ref{str2}). To prove (\ref{str2}) in general, we use Proposition \ref{sobolev} to reduce\[\|u\|_{L_{t}^{r}\mathcal{W}_{x}^{\sigma_{1},q}([-T,T]\times\mathbb{R}^{2})}\lesssim\|u\|_{L_{t}^{r}\mathcal{W}_{x}^{\sigma_{2},q'}([-T,T]\times\mathbb{R}^{2})}\lesssim\langle T\rangle^{\frac{1}{r}}\|u\|_{\mathcal{X}^{\sigma_{2},b,T}},\] where $\frac{1}{q'}+\frac{1}{r}=\frac{1}{2}$ (so that $2<q,q',r<\infty$ and $\sigma_{2}-\sigma_{1}\geq 2(\frac{1}{q'}-\frac{1}{q})\geq 0$), and with obvious modifications when $u$ is globally defined.

Finally we prove (\ref{stri2.5}). Again we may assume $\sigma_{1}=0$. For $v=u$ on $[-T,T]$ and $v=0$ elsewhere, we need to show \begin{equation}\|v\|_{\mathcal{X}^{0,b-1}}\lesssim \langle T\rangle^{\frac{1}{q_{2}}-\frac{1}{2}}\|u\|_{L_{t,x}^{q_{2}}([-T,T]\times\mathbb{R}^{2})}.\end{equation} For any $w$ with $\|w\|_{\mathcal{X}^{0,1-b}}\lesssim 1$, we have \begin{equation}\bigg|\int_{\mathbb{R}\times\mathbb{R}^{2}}v\bar{w}\,\mathrm{d}t\mathrm{d}x\bigg|=\bigg|\int_{[-T,T]\times\mathbb{R}^{2}}u\bar{w}\,\mathrm{d}t\mathrm{d}x\bigg|\lesssim\|w\|_{L_{t,x}^{q_{3}}([-T,T]\times\mathbb{R}^{2})}\cdot\|u\|_{L_{t,x}^{q_{2}}([-T,T]\times\mathbb{R}^{2})},\end{equation} where $q_{3}=\frac{q_{2}}{q_{2}-1}$. Thus by duality, we only need to prove $\|w\|_{L_{t,x}^{q_{3}}}\lesssim \langle T\rangle^{\frac{1}{2}-\frac{1}{q_{3}}}\|w\|_{\mathcal{X}^{0,1-b}}$ for all $2<q_{3}<\frac{2}{b}$. Since the imaginary power $\langle \mathrm{i}\partial_{t}-\mathbf{H}\rangle^{\mathrm{i}\tau}$ is an isometry on $L_{t,x}^{2}$, we can use Stein's complex interpolation to reduce to the cases $(b,q_{3})=(1,2)$ and $(b_{1},4)$, where $b_{1}=\frac{q_{3}-4+bq_{3}}{2q_{3}-4}<\frac{1}{2}$. The former is trivial by definition, and the latter is a special case of (\ref{str2}).
\end{proof}
\begin{lemma}\label{linear00}
Fix $\sigma,b\in\mathbb{R}$, $0<T\leq 1$ and a cutoff function $\psi$.

(1) If $-\frac{1}{2}<b'\leq b<\frac{1}{2}$, then for $u\in \mathcal{X}^{\sigma,b}$ we have \begin{equation}\label{linear0001}\|\psi(T^{-1}t)u\|_{\mathcal{X}^{\sigma,b'}}\lesssim T^{b-b'}\|u\|_{\mathcal{X}^{\sigma,b}}.\end{equation} Also for $u\in \mathcal{X}^{\sigma,b,T}$ we have \begin{equation}\label{linear0002}\|u\|_{\mathcal{X}^{\sigma,b',T}}\lesssim T^{b-b'}\|u\|_{\mathcal{X}^{\sigma,b,T}}.\end{equation}

(2) If $\frac{1}{2}<b'=b<1$, then for $u\in\mathcal{X}^{\sigma,b}$ with $u(0)=0$, (\ref{linear0001}) holds, as well as the limit \begin{equation}\label{linear0003}\lim_{T\to 0}\|\psi(T^{-1}t)u\|_{\mathcal{X}^{\sigma,b}}=0.\end{equation}
\end{lemma}
\begin{proof}
(1) If (\ref{linear0001}) is true, then for any $u\in\mathcal{X}^{\sigma,b,T}$ and any extension $v\in\mathcal{X}^{\sigma,b}$ of $u$, we have \[\|u\|_{\mathcal{X}^{\sigma,b',T}}\leq\|\psi(T^{-1}t)v\|_{\mathcal{X}^{\sigma,b'}}\lesssim T^{b-b'}\|v\|_{\mathcal{X}^{\sigma,b}},\] provided $\psi\equiv 1$ on $[-1,1]$. Taking infimum over $v$, we get (\ref{linear0002}). Now we prove (\ref{linear0001}). Define the operator $Mu(t,x):=e^{\mathrm{i}t\mathbf{H}}u(t,\cdot)(x)$. We have \[\mathrm{i}\partial_{t}(Mu)=e^{\mathrm{i}t\mathbf{H}}(\mathrm{i}\partial_{t}-\mathbf{H})u,\] and therefore we get $\|u\|_{\mathcal{X}^{\sigma,b}}=\|Mu\|_{H_{t}^{b}\mathcal{H}_{x}^{\sigma}}$. Since $M$ also commutes with multiplication of functions of time, we can reduce to $\|\psi(T^{-1}t)v\|_{H_{t}^{b'}\mathcal{H}_{x}^{\sigma}}\lesssim T^{b-b'}\|v\|_{H_{t}^{b}\mathcal{H}_{x}^{\sigma}}$. By eigenfunction expansion, we can further reduce to \begin{equation}\label{linear0009}\|\psi(T^{-1}t)g\|_{H^{b'}}\lesssim T^{b-b'}\|g\|_{H^{b}}.\end{equation} By composition we may assume $0\leq b'\leq b$ or $b'\leq b\leq 0$, by duality we may assume $0\leq b'\leq b$, by interpolation we may assume $b'\in\{0,b\}$.

First suppose $b'=b$, and we want to prove that multiplication by $\psi(T^{-1}t)$ is bounded, independent of $T>0$, on $H^{b}$. Since it is bounded on $L^{2}$, we only need to show that it is also bounded on $\dot{H}^{b}$. By rescaling we may set $T=1$. For each $g\in\dot{H}^{b}$, we split $g=g_{1}+g_{2}$, where $\hat{g}_{1}$ is supported on $\{|\xi|\leq 1\}$ and $\hat{g}_{2}$ supported on $\{|\xi|\geq 1\}$. Multiplication by $\psi$ is obviously bounded from $H^{1}$ to $\dot{H}^{1}$, and from $L^{2}$ to $L^{2}$. So it is bounded from $H^{b}$ to $\dot{H}^{b}$, thus $\|\psi g_{2}\|_{\dot{H}^{b}}\lesssim \|g_{2}\|_{H^{b}}\lesssim\|g\|_{\dot{H}^{b}}$. Since $b<\frac{1}{2}$, we also know \[\int_{|\tau|\leq 1}|\hat{g}_{1}(\tau)|\,\mathrm{d}\tau\lesssim\big\||\tau|^{b}\hat{g}_{1}(\tau)\big\|_{L^{2}([-1,1])}\cdot\big\||\tau|^{-b}\big\|_{L^{2}([-1,1])}\lesssim\|g_{1}\|_{\dot{H}^{b}}\lesssim\|g\|_{\dot{H}^{b}}.\] Thus $(\psi g_{1})^{\wedge}(\tau)=(\hat{\psi}\ast\hat{g}_{1})(\tau)$ is bounded pointwise by $\langle\tau\rangle^{-N}\|g\|_{\dot{H}^{b}}$, since $\hat{\psi}$ is Schwartz, and the result follows. 

Next suppose $b'=0$, we only need to prove the stronger result\[\|\psi(T^{-1}t)g\|_{L^{2}}\lesssim T^{b}\|g\|_{\dot{H}^{b}}.\] By rescaling we can set $T=1$. Using the same splitting $g=g_{1}+g_{2}$, we have $\|\psi g_{2}\|_{L^{2}}\lesssim\|g_{2}\|_{L^{2}}\lesssim\|g\|_{\dot{H}^{b}}$, and $|\psi g_{1}(\tau)|\lesssim\langle \tau\rangle^{-N}\|g\|_{\dot{H}^{b}}$. This proves (\ref{linear0009}) and hence (\ref{linear0001}).

(2) We want to prove (\ref{linear0001}), and again we can reduce to (\ref{linear0009}), where we also have $g(0)=0$. Using the same arguments as in (1), we can further reduce to the boundedness on $\dot{H}^{b}$ and assume $T=1$. Split $g=g_{1}+g_{2}$ so that (though we are considering $\dot{H}^{b}$ norm here, we still assume $g\in H^{b}$, so $\hat{g}\in L^{1}$)\[\hat{g_{2}}(\tau)=\chi_{|\tau|\geq 1}\cdot\hat{g}(\tau)-\frac{1}{2}\int_{|\lambda|\geq 1}\hat{g}(\lambda)\,\mathrm{d}\lambda\cdot\chi_{1\leq|\tau|\leq 2},\] then $\hat{g}_{1}$ is supported in $\{|\tau|\leq 2\}$, $\hat{g}_{2}$ is supported in $\{|\tau|\geq 1\}$, both $\hat{g}_{i}$ has integral zero (since $\hat{g}$ has integral zero), and $\|g_{i}\|_{\dot{H}^{b}}\lesssim\|g\|_{\dot{H}^{b}}$ (since $b>\frac{1}{2}$, $\|\hat{g}\|_{L^{1}(\{|\tau|\geq 1\})}\lesssim\||\tau|^{b}\hat{g}\|_{L^{2}}=\|g\|_{\dot{H}^{b}}$). For $g_{1}$ we have $\|\psi g_{2}\|_{\dot{H}^{b}}\lesssim\|g_{2}\|_{H^{b}}\lesssim \|g\|_{\dot{H}^{b}}$ as in (1); for $g_{1}$ we have\[(\psi g_{1})^{\wedge}(\tau)=\int_{-2}^{2}(\hat{\psi}(\tau-\xi)-\hat{\psi}(\tau))\hat{g_{1}}(\xi)\,\mathrm{d}\xi.\] By Cauchy-Schwartz\[|(\psi g_{1})^{\wedge}(\tau)|\lesssim\|g_{1}\|_{\dot{H}^{b}}\bigg(\int_{-2}^{2}|\xi|^{-2b}|\hat{\psi}(\tau-\xi)-\hat{\psi}(\tau)|^{2}\,\mathrm{d}\xi\bigg)^{\frac{1}{2}}\lesssim\langle\tau\rangle^{-N}\|g\|_{\dot{H}^{b}},\] and (\ref{linear0009}) follows. Finally, to prove (\ref{linear0003}), we first use the operator $M$ and approximation by a finite linear combination of eigenfunctions to reduce to $\|\psi(T^{-1}t)g\|_{H^{b}}\to 0$ ($T\to 0$). Since this is easily verified for Schwartz $g$, we only need to check any $g\in H^{b}$ with $g(0)=0$ can be approximated by Schwartz $h$ also with $h(0)=0$. But this easily follows since $H^{b}$ is embedded in $L^{\infty}$.
\end{proof}
\begin{proposition}\label{linear}
%
%
%
%linear estimate; Duhamel estimate+subdivision+limit=0
%
%
%
Suppose $\frac{1}{2}<b<1$. We have\begin{equation}\label{linher1}\bigg\|\int_{0}^{t}e^{-\mathrm{i}(t-s)\mathbf{H}}u(s)\,\mathrm{d}s\bigg\|_{\mathcal{X}^{\sigma,b,T}}\lesssim\|u\|_{\mathcal{X}^{\sigma,b-1,T}},\end{equation} for $T\leq 1$. Also for $u\in \mathcal{X}^{\sigma,b,T}$, the function $\|u\|_{\mathcal{X}^{\sigma,b,\delta}}$ is continuous for $T\geq\delta>0$, and if $u(0)=0$, it tends to $0$ as $\delta\to 0$. Moreover, if $p>\frac{1}{2}$ and \begin{equation}\label{local}\|u-e^{-\mathrm{i}(t-k\delta)\mathbf{H}}u(k\delta)\|_{\mathcal{X}^{\sigma,b,[(k-1)\delta,(k+1)\delta]}}\leq C\end{equation} for $|k|\leq K$, then \begin{equation}\label{glb}\|u-e^{-\mathrm{i}t\mathbf{H}}u(0)\|_{\mathcal{X}^{\sigma,b,K\delta}}\lesssim c_{1}K^{2}\delta^{-\frac{b}{2}}.\end{equation}
\end{proposition}
\begin{proof}
 For the operator $M$ defined in the proof of Lemma \ref{linear00} we have\begin{equation}M\bigg(\int_{0}^{t}e^{-\mathrm{i}(t-s)\mathbf{H}}u(s)\,\mathrm{d}s\bigg)=\int_{0}^{t}Mu(s)\,\mathrm{d}s,\end{equation} therefore we can again use eigenfunction expansion to reduce the problem and see that (\ref{linher1}) will follow if the operator\begin{equation}g(t)\mapsto \mathcal{I}g(t):=\eta(t)\int_{0}^{t}g(s)\,\mathrm{d}s\end{equation} is bounded from $H_{t}^{b-1}$ to $H_{t}^{b}$, where $\eta$ is a fixed smooth function supported on $[-3,3]$ that equals $1$ on $[-2,2]$. Choose a smooth compactly supported function $\psi$ that equals $1$ on $[-10,10]$, and $\phi$ supported on $[-5,5]$ that equals $1$ on $[-4,4]$. Then we have\begin{equation}\mathcal{I}g(t)=\eta(t)\int_{-\infty}^{t}\psi(t-s)\phi(s)g(s)\,\mathrm{d}s-\eta(t)\int_{-5}^{0}\phi(s)g(s)\,\mathrm{d}s.\end{equation} We know multiplication by $\eta$ is bounded on $H^{b}$, multiplication by $\phi$ is bounded on $H^{1-b}$ (to prove these, we first prove them in $L^{2}$ and $H^{1}$ explicitly, then interpolate), and convolution with $\psi\cdot\chi_{[0,\infty)}$ is bounded from $H^{b-1}$ to $H^{b}$, since its Fourier transform is controlled by $\langle\tau\rangle^{-1}$. Thus the first term is bounded. For the second term, we only need to prove $|\langle g,\phi_{0}\rangle|\lesssim \|g\|_{H^{b-1}}$, where $\phi_{0}=\phi\cdot\chi_{[0,5]}$ with $|\hat{\phi_{0}}(\tau)|\lesssim \langle\tau\rangle^{-1}$. But this follows from Plancherel, Cauchy-Schwartz, and the assumption $b>\frac{1}{2}$. This proves (\ref{linher1}).

Next we consider the function $M(\delta):=\|u\|_{\mathcal{X}^{\sigma,b,\delta}}$, which is clearly nondecreasing. Since we only consider $0<\delta\leq T$, we may assume $u$ is defined for $t\in\mathbb{R}$ and belongs to $\mathcal{X}^{\sigma,b}$. For each $\delta>0$, denote by $M_{0}$ the left limit of the function $M$ at point $\delta$, and choose a sequence $\delta_{n}\uparrow \delta$, and (by definition)
  a sequence of $v_{n}$ so that $v_{n}\equiv u$ on $[-\delta_{n},\delta_{n}]$ and $\lim_{n\to\infty}\|v_{n}\|_{\mathcal{X}^{\sigma,b}}\leq M_{0}$. These $v_{n}$ have a subsequence converging weakly to some $v$ with $\|v\|_{\mathcal{X}^{\sigma,b}}\leq M_{0}$. Using the embedding $L_{t}^{\infty}\mathcal{H}_{x}^{\sigma}\supset\mathcal{X}^{\sigma,b}$, we easily see $v\equiv u$ on $[-\delta,\delta]$. This proves left continuity. To prove right continuity at $\delta$, write $M(\delta)=M_{1}$. For any $\epsilon$, we choose $v\equiv u$ on $[-\delta,\delta]$ and $\|v\|_{\mathcal{X}^{\sigma,b}}<M_{1}+\epsilon$. Let $u-v=w$ with $w\equiv 0$ on $[-\delta,\delta]$, and define\[w_{\tau}=\big(\psi(\tau^{-1}(t-\delta))+\psi(\tau^{-1}(t+\delta))\big)w,\] for some suitable cutoff which equals $1$ on a small neighborhood of $0$. From the definition of $w_{\tau}$, we see that for small $\tau$, $v+w_{\tau}\equiv u$ on a neighborhood of $[-\delta,\delta]$. From Lemma \ref{linear00} we know $\|w_{\tau}\|_{\mathcal{X}^{\sigma,b}}\to 0$ as $\tau\to 0$, thus $\|v+w_{\tau}\|_{\mathcal{X}^{\sigma,b}}<M_{1}+2\epsilon$ if $\tau$ is small enough. This proves right continuity. Finally, if $u(0)=0$, then \[\lim_{\delta\to 0}\|u\|_{\mathcal{X}^{\sigma,b,\delta}}\leq\lim_{\tau\to 0}\|\psi(\tau^{-1}t)u\|_{\mathcal{X}^{\sigma,b}}=0,\] for the same cutoff $\psi$.
  
Finally we prove (\ref{glb}). From (\ref{local}) and the embedding $\|g\|_{L_{t}^{\infty}\mathcal{H}_{x}^{\sigma}}\lesssim \|g\|_{\mathcal{X}^{\sigma,b,\delta}}$ we see in particular $\|u(k\delta)-e^{-\mathrm{i}k\delta\mathbf{H}}u(0)\|_{\mathcal{H}^{\sigma}}\lesssim K$. Now choose $w_{k}$ so that $w_{k}\equiv u-e^{-\mathrm{i}(t-k\delta)\mathbf{H}}u(k\delta)$ on $[(k-1)\delta,(k+1)\delta]$ and $\|w_{k}\|_{\mathcal{X}^{\sigma,b}}\leq C$, and choose a partition of unity $\psi_{k}$ subordinate to the covering $\{((k-1)\delta,(k+1)\delta)\}$ of $[-K\delta,K\delta]$, so that $\psi_{k}(t)=\tilde{\psi}_{k}(\frac{t}{\delta}-k)$ and $\tilde{\psi}_{k}$ have bounded Schwartz norms (this is well-known). We then have\begin{equation}w=\sum_{k}\psi_{k}w_{k}+\sum_{k}\psi_{k}e^{-\mathrm{i}(t-k\delta)\mathbf{H}}(u(k\delta)-e^{-\mathrm{i}k\delta\mathbf{H}}u(0))\equiv v \,\,\,\textrm{on}\,\,\,[-K\delta,K\delta],\end{equation} and $\|w\|_{\sigma,b}\lesssim K^{2}\delta^{-\frac{b}{2}}$, since it is easy to check (by reducing to estimates of functions of $t$ and interpolating between $L^{2}$ and $H^{1}$) that multiplication by $\psi_{k}$ is bounded from $\mathcal{X}^{\sigma,b}$ to itself with norm $\lesssim\delta^{-\frac{b}{2}}$, and that by definition\[\|\psi_{k}e^{-\mathrm{i}(t-k\delta)\mathbf{H}}(u(k\delta)-e^{-\mathrm{i}k\delta\mathbf{H}}u(0))\|_{\mathcal{X}^{\sigma,b}}=\|u(k\delta)-e^{-\mathrm{i}k\delta\mathbf{H}}u(0)\|_{\mathcal{H}^{\sigma}}\|\psi_{k}\|_{H^{b}}\lesssim K\delta^{\frac{1}{2}-b}.\] This completes the proof.
\end{proof}

\section{Construction of Gibbs measure}\label{gib}

We will construct the Gibbs measure of (\ref{nls2}) for $1<p<\infty$ (defocusing case) and $1<p<3$ (focusing case).  From the definition (\ref{random}) of $f$, it is obvious that \begin{equation}\|f(\omega)\|_{\mathcal{H}^{\tau}}^{2}=\sum_{k=0}^{\infty}(4k+2)^{-1+\tau}|g_{k}(\omega)|^{2}.\end{equation}
This expression is a.s. finite if $\tau<0$, and is a.s. infinite if $\tau\geq 0$. Thus we have\begin{equation}f(\omega)\in \mathcal{H}^{0-}:=\bigcap_{\delta>0}\mathcal{H}^{-\delta},\end{equation} a.s. in $\mathbb{P}$. Define $\mu=\mathbb{P}\circ f^{-1}$ to be the push-forward of $\mathbb{P}$ under $f$, then we see that the typical element in the support of $\mu$ belongs to any $\mathcal{H}^{-\delta}$ for all $\delta>0$, but does not belong to $L^{2}$. We also define $\mu_{2^{k}}^{\circ}=\mathbb{P}\circ (f_{2^{k}}^{\circ})^{-1}$, and $\mu_{2^{k}}^{\perp}=\mathbb{P}\circ (f_{2^{k}}^{\perp})^{-1}$. Now we prove two lemmas concerning linear and multilinear estimates of the eigenfunctions $e_{k}(x)$ as defined in (\ref{eigen}).

\begin{lemma}\label{lbg}
For any $2\leq q\leq\infty$ and $q\neq 4$, write $\nu=4k+2$ for $k\geq 0$, then we have\begin{equation}\label{fst}\|e_{k}\|_{L^{q}(\mathbb{R}^{2})}\lesssim \nu^{-\rho(q)},\end{equation} where $\rho(q)=\min\big\{\frac{1}{2}-\frac{1}{q},\frac{1}{q}\big\}$. If $q=4$ we have \begin{equation}\label{lin}\|e_{k}\|_{L^{4}(\mathbb{R}^{2})}\lesssim \nu^{-\frac{1}{4}}\log^{\frac{1}{4}}\nu.\end{equation}
\end{lemma}

\begin{proof}Since $e_{k}(x)=\pi^{-\frac{1}{2}}\mathcal{L}_{k}^{0}(|x|^{2})$, we easily see $\|e_{k}\|_{L^{q}(\mathbb{R}^{2})}\sim\|\mathcal{L}_{k}^{0}\|_{L^{q}(\mathbb{R}^{+})}$. Then we can use (\ref{ctl}) to compute \setlength\arraycolsep{2pt}
\begin{eqnarray}
\|\mathcal{L}_{k}^{0}\|_{L^{4}(\mathbb{R}^{+})}^{4}&\lesssim&\int_{0}^{\frac{1}{\nu}}\,\mathrm{d}z+\int_{\frac{1}{\nu}}^{\frac{\nu}{2}}(z\nu)^{-1}\,\mathrm{d}z\\&+&\nu^{-1}\int_{\frac{\nu}{2}}^{\frac{3\nu}{2}}\big(\nu^{\frac{1}{3}}+|\nu-z|\big)^{-1}\,\mathrm{d}z +\int_{\frac{3\nu}{2}}^{\infty}e^{-cz}\,\mathrm{d}z\nonumber\\
&\lesssim &\nu^{-1}\log\nu.\nonumber
\end{eqnarray} This proves (\ref{lin}). As for (\ref{fst}) we have\setlength\arraycolsep{2pt}
\begin{eqnarray}
\|\mathcal{L}_{k}^{0}\|_{L^{q}(\mathbb{R}^{+})}^{q}&\lesssim&\int_{0}^{\frac{1}{\nu}}\,\mathrm{d}z+\int_{\frac{1}{\nu}}^{\frac{\nu}{2}}(z\nu)^{-\frac{q}{4}}\,\mathrm{d}z\\&+&\nu^{-\frac{q}{4}}\int_{\frac{\nu}{2}}^{\frac{3\nu}{2}}\big(\nu^{\frac{1}{3}}+|\nu-z|\big)^{-\frac{q}{4}}\,\mathrm{d}z +\int_{\frac{3\nu}{2}}^{\infty}e^{-cz}\,\mathrm{d}z\nonumber\\
&\lesssim & \nu^{-\frac{q}{4}+|1-\frac{q}{4}|}+\nu^{\frac{1-q}{3}}+\nu^{\max(1-\frac{q}{2},\frac{1-q}{3})}\nonumber\\
&\lesssim &\nu^{-q\rho(q)}.\nonumber
\end{eqnarray}
\end{proof}
\begin{lemma}\label{multil}Suppose $l\geq 4$ and $n_{1},\cdots,n_{l}\geq 0$. Let $\nu_{j}=4n_{j}+2$ for $1\leq j\leq l$, and assume $\nu_{1}\gtrsim\cdots \gtrsim \nu_{l}$. Then we have\begin{equation}\label{smesti}\bigg|\int_{\mathbb{R}^{2}}e_{n_{1}}(x)\cdots e_{n_{l}}(x)\bigg|\lesssim \nu_{1}^{-\frac{1}{2}}\nu_{3}^{-\frac{1}{4}}\log\nu_{1}.\end{equation} Moreover, if $\nu_{1}\gtrsim\nu_{2}^{1+\epsilon}$ for some $\epsilon>0$, then\begin{equation}\label{lgdecay}\bigg|\int_{\mathbb{R}^{2}}e_{n_{1}}(x)\cdots e_{n_{l}}(x)\bigg|\lesssim \nu_{1}^{-N},\end{equation} for all $N>0$.
\end{lemma}
\begin{proof} Recall that $\mathbf{H}e_{n_{j}}=\nu_{j}e_{n_{j}}$ and $\mathbf{H}$ is self-adjoint on $L^{2}(\mathbb{R}^{2})$, we can compute using Proposition \ref{sobolev} and Lemma \ref{lbg} that  \setlength\arraycolsep{2pt}
\begin{eqnarray}
\bigg|\int_{\mathbb{R}^{2}}e_{n_{1}}(x)\cdots e_{n_{l}}(x)\bigg|&\leq&\nu_{1}^{-m}\|\mathbf{H}^m(e_{n_{2}}\cdots e_{n_{l}})\cdot e_{n_{1}}\|_{L^{1}}\nonumber\\
&\lesssim &\nu_{1}^{-m}\|e_{n_{2}}\cdots e_{n_{l}}\|_{\mathcal{H}^{2m}}\nonumber\\
&\lesssim&\nu_{1}^{-m}\sum_{j=2}^{l}\|e_{n_{j}}\|_{\mathcal{W}^{2m,2(l-1)}}\prod_{2\leq i\neq j}\|e_{n_{i}}\|_{L^{2(l-1)}}\nonumber\\&\lesssim & (\nu_{1}^{-1}\nu_{2})^{m}.\nonumber\end{eqnarray}
If $\nu_{1}\gtrsim \nu_{2}^{1+\epsilon}$, we can choose $m$ large enough and prove (\ref{lgdecay}). As for (\ref{smesti}), we choose $m=1$ and estimate
\setlength\arraycolsep{2pt}
\begin{eqnarray}
\bigg|\int_{\mathbb{R}^{2}}e_{n_{1}}(x)\cdots e_{n_{l}}(x)\bigg|&\leq&\nu_{1}^{-1}\|\mathbf{H}(e_{n_{2}}\cdots e_{n_{l}})\cdot e_{n_{1}}\|_{L^{1}}\nonumber\\
&\lesssim &\nu_{1}^{-\frac{5}{4}}\log^{\frac{1}{4}}\nu_{1}\cdot\|e_{n_{2}}\cdots e_{n_{l}}\|_{\mathcal{W}^{2,\frac{4}{3}}}\nonumber\\
&\lesssim&\nu_{1}^{-\frac{5}{4}}\log^{\frac{1}{4}}\nu_{1}\cdot\nu_{2} \|e_{n_{2}}\|_{L^{4}}\|e_{n_{3}}\|_{L^{4}}\prod_{j\geq 4}\|e_{n_{i}}\|_{L^{4(l-3)}}\nonumber\\
&\lesssim & \nu_{1}^{-\frac{5}{4}}\nu_{2}^{\frac{3}{4}}\nu_{3}^{-\frac{1}{4}}\log^{\frac{3}{4}}\nu_{1}\nonumber\\
&\lesssim&\nu_{1}^{-\frac{1}{2}}\nu_{3}^{-\frac{1}{4}}\log\nu_{1}.\nonumber
\end{eqnarray}
\end{proof}
Before we are able to state and prove the probabilistic $L^{p}$ estimates for our $\mathcal{S}'$-valued random variable $f$, we need a result proved by Fernique.
\begin{lemma}[Fernique]\label{fnq} There exist absolute constants $c,C$ such that for any finite dimensional normed vector space $(V,\|\cdot\|)$, any centered Gaussian random variable $f(\omega)$ taking its value in $V$, and any positive constant $A$, if $\mathbb{P}(\|f(\omega)\|>A)<\frac{1}{10}$, then\begin{equation}\mathbb{E}\big(e^{cA^{-2}\|f(\omega)\|^{2}}\big)\leq C.\end{equation}
\end{lemma}
\begin{proof} See Fernique \cite{Fe74} or Prato-Zabczyk \cite{PZ93}, Theorem 2.6.
\end{proof}
\begin{proposition}\label{lin2}
Fix $2<q<\infty$, $1<r<\infty$, $0<\alpha<\min(\frac{2}{q},1-\frac{2}{q})$, and two positive integers $M>10N$. For any $g$, we define \begin{equation}\Pi  g=\sum_{j=N-1}^{M}\langle g,e_{j}\rangle e_{j}.\end{equation}Then, for the random variable $f$ as defined in (\ref{random}), we have the large deviation estimates \begin{equation}\label{dyd1}\mathbb{P}\big(\|\Pi  f(\omega)\|_{\mathcal{W}^{\alpha,q}}>AN^{-\delta}\big)\leq Ce^{-cA^{2}},\end{equation}
\begin{equation}\label{dyd2}\mathbb{P}\big(\|e^{-\mathrm{i}t\mathbf{H}}\Pi  f(\omega)\|_{L_{t}^{r}\mathcal{W}_{x}^{\alpha,q}([-T,T]\times\mathbb{R}^{2})}>AN^{-\delta}T^{\frac{1}{r}}\big)\leq Ce^{-cA^{2}},\end{equation} where $\delta>0$ is some small positive exponent.
\end{proposition}
\begin{proof} 
We compute for each $t\in[-\pi,\pi]$\begin{equation}\mathbb{E}\big(\|e^{-\mathrm{i}t\mathbf{H}}\Pi f(\omega)\|_{\mathcal{W}_{x}^{\alpha,q}}^{q}\big)=\int_{\mathbb{R}^{2}}\mathbb{E}\bigg|\sum_{j=N-1}^{M}(4j+2)^{\frac{\alpha-1}{2}}g_{j}(\omega)e_{j}(x)\bigg|^{q}\,\mathrm{d}x.\end{equation}
Now by Khintchine's inequality (the variant for Gaussians), we have\begin{equation}
\mathbb{E}\bigg|\sum_{j=N-1}^{M}(4j+2)^{\frac{\alpha-1}{2}}g_{j}(\omega)e_{j}(x)\bigg|^{q}\lesssim\bigg(\sum_{j=N-1}^{M}\frac{e_{j}(x)^{2}}{(4j+2)^{1-\alpha}}\bigg)^{\frac{q}{2}}.\end{equation} Then integrating in $x$, using Minkowski's inequality (since $q>2$), we get \begin{equation}\label{fixt}\mathbb{E}\big(\|e^{-\mathrm{i}t\mathbf{H}}\Pi f(\omega)\|_{\mathcal{W}_{x}^{\alpha,q}}^{q}\big)\lesssim\bigg(\sum_{j=N-1}^{M}\frac{\|e_{j}\|_{L^{q}}^{2}}{(4j+2)^{1-\alpha}}\bigg)^{\frac{q}{2}}\leq CN^{-q\delta},\end{equation} due to Lemma \ref{lbg}, and the assumption $\alpha<2\rho(q)$. Now we can take $t=0$ in (\ref{fixt}) and use Markov's inequality and Lemma \ref{fnq}, and immediately get (\ref{dyd1}).

As for (\ref{dyd2}), we need a little more work. What we need is \begin{equation}\label{dev0}\mathbb{P}\big(\|e^{-\mathrm{i}t\mathbf{H}}\Pi f(\omega)\|_{L_{t}^{r}\mathcal{W}_{x}^{\alpha,q}}>CN^{-\delta}T^{\frac{1}{r}}\big)<\frac{1}{10},\end{equation} for large $C$. If the event in (\ref{dev0}) happens, then there exists an integer $l\geq 0$ such that\begin{equation}\label{dis}\big|\big\{t\in[-T,T]:\|e^{-\mathrm{i}t\mathbf{H}}\Pi f(\omega)\|_{\mathcal{W}_{x}^{\alpha,q}}>2^{l}N^{-\delta}\big\}\big|>K2^{-2rl}T.\end{equation}
For fixed $t$, due to (\ref{fixt}) and Lemma \ref{fnq}, the probability that $\|e^{-\mathrm{i}t\mathbf{H}}\Pi f(\omega)\|_{\mathcal{W}_{x}^{\alpha,q}}>2^{l}N^{-\delta}$ is less than $c_{1}e^{-c_{2}2^{2l}}$. We then use Fubini's Theorem to conclude that the probability that (\ref{dis}) happens is less than $K^{-1}c_{1}2^{2rl}e^{-c_{2}2^{2l}}$. Then we sum over $l\geq 0$ and choose $K$ large enough so that this sum is less than $\frac{1}{10}$.
\end{proof}
\begin{corollary}\label{prob}

For the same parameters $q,r,\alpha$ as in Proposition \ref{lin2}, we have \begin{equation}\label{wh}\mathbb{P}\big(\|f(\omega)\|_{\mathcal{W}^{\alpha,q}}>A\big)\leq Ce^{-cA^{2}},\end{equation} 
\begin{equation}\label{sup}\mathbb{P}\big(\sup_{k\geq 0}\|f_{2^{k}}^{\circ}(\omega)\|_{\mathcal{W}^{\alpha,q}}>A\big)\leq Ce^{-cA^{2}},\end{equation} \begin{equation}\label{wh2}\mathbb{P}\big(\|e^{-\mathrm{i}t\mathbf{H}}f(\omega)\|_{L_{t}^{r}\mathcal{W}_{x}^{\alpha,q}([-T,T]\times\mathbb{R}^{2})}>AT^{\frac{1}{r}}\big)\leq Ce^{-cA^{2}},\end{equation}\begin{equation}\label{sup2}\mathbb{P}\big(\sup_{k\geq 0}\|e^{-\mathrm{i}t\mathbf{H}}f_{2^{k}}^{\circ}(\omega)\|_{L_{t}^{r}\mathcal{W}_{x}^{\alpha,q}([-T,T]\times\mathbb{R}^{2})}>AT^{\frac{1}{r}}\big)\leq Ce^{-cA^{2}},\end{equation}
\begin{equation}\label{conv}\lim_{k\to\infty}\|f_{2^{k}}^{\circ}(\omega)-f(\omega)\|_{\mathcal{W}^{\alpha,q}}+\|e^{-\mathrm{i}t\mathbf{H}}(f_{2^{k}}^{\circ}(\omega)-f(\omega))\|_{L_{t}^{r}\mathcal{W}_{x}^{\alpha,q}([-T,T]\times\mathbb{R}^{2})}=0\,\,\,\mathrm{a.s.}\,\mathbb{P}.\end{equation}
\end{corollary}
\begin{proof}

We know $f_{2^{k}}^{\circ}(\omega)\to f(\omega)$ and $e^{-\mathrm{i}t\mathbf{H}}f_{2^{k}}^{\circ}(\omega)\to e^{-\mathrm{i}t\mathbf{H}}f(\omega)$ in $\mathcal{S}'$. If we can prove (\ref{sup}) and (\ref{sup2}), then a.s. in $\mathbb{P}$, we have \begin{equation}\sup_{k\geq 0} \|e^{-\mathrm{i}t\mathbf{H}}f_{2^{k}}^{\circ}(\omega)\|_{L_{t}^{r}\mathcal{W}_{x}^{\alpha,q}}<\infty,\end{equation} and there must be a subsequence of $\{e^{-\mathrm{i}t\mathbf{H}}f_{2^{k}}^{\circ}(\omega)\}$ converging weakly in $L_{t}^{r}\mathcal{W}_{x}^{\alpha,q}$. This weak limit must be $e^{-\mathrm{i}t\mathbf{H}}f(\omega)$, so we know that\begin{equation}\label{fatou}\|e^{-\mathrm{i}t\mathbf{H}}f(\omega)\|_{L_{t}^{r}\mathcal{W}_{x}^{\alpha,q}}\leq \sup_{k\geq0}\|e^{-\mathrm{i}t\mathbf{H}}f_{2^{k}}^{\circ}(\omega)\|_{L_{t}^{r}\mathcal{W}_{x}^{\alpha,q}}<\infty,\end{equation} a.s. in $\mathbb{P}$. Thus (\ref{wh2}) also holds true, with the same constants as in (\ref{sup2}). Clearly (\ref{wh}) also follows from (\ref{sup}) in the same way.

To prove (\ref{sup}) and (\ref{sup2}), we use (\ref{dyd1}) and (\ref{dyd2}). For any $k$, the difference $f_{2^{k}}^{\circ}(\omega)-f_{2^{k-1}}^{\circ}(\omega)$ is of the form $\Pi f(\omega)$ as defined in Proposition \ref{lin2}, with the parameter $N\sim 2^{k}$. We then have, for some $\delta>0$
\begin{equation}\label{sep2}\mathbb{P}\big(\|e^{-\mathrm{i}t\mathbf{H}}(f_{2^{k}}^{\circ}(\omega)-f_{2^{k-1}}^{\circ}(\omega))\|_{L_{t}^{r}\mathcal{W}_{x}^{\alpha,q}}>A2^{-\frac{k\delta}{2}}T^{\frac{1}{r}}\big)\leq c_{1}e^{-c_{2}2^{k\delta}A^{2}}.\end{equation} Choose $c$ small enough, then
\begin{equation}\sup_{k\geq 0}\|e^{-\mathrm{i}t\mathbf{H}}f_{2^{k}}^{\circ}(\omega)\|_{L_{t}^{r}\mathcal{W}_{x}^{\alpha,q}}>AT^{\frac{1}{r}}\end{equation} implies\begin{equation}\exists k\geq 0,\,\,\|e^{-\mathrm{i}t\mathbf{H}}(f_{2^{k}}^{\circ}(\omega)-f_{2^{k-1}}^{\circ}(\omega))\|_{L_{t}^{r}\mathcal{W}_{x}^{\alpha,q}}>cA2^{-\frac{k\delta}{2}}T^{\frac{1}{r}}.\end{equation} Now we can combine this with (\ref{sep2}) to get\begin{equation}\mathbb{P}\big(\sup_{k\geq0}\|e^{-\mathrm{i}t\mathbf{H}}f_{2^{k}}^{\circ}(\omega)\|_{L_{t}^{r}\mathcal{W}_{x}^{\alpha,q}}>AT^{\frac{1}{r}}\big)\leq\sum_{k=0}^{\infty}c_{3}e^{-c_{4}2^{k\delta}A^{2}}\leq c_{5}e^{-c_{6}A^{2}}.\end{equation} This proves (\ref{sup2}). Clearly (\ref{sup}) also follows from (\ref{dyd1}) in the same way.

Finally we prove (\ref{conv}). From the above discussion we see\begin{equation}\mathbb{P}\big(\sup_{k\geq0}2^{\frac{k\delta}{2}}\|e^{-\mathrm{i}t\mathbf{H}}(f_{2^{k}}^{\circ}(\omega)-f_{2^{k-1}}^{\circ}(\omega))\|_{L_{t}^{r}\mathcal{W}_{x}^{\alpha,q}}<\infty\big)=1,\end{equation} thus with probability $1$, the series\begin{equation}\sum_{k=0}^{\infty}e^{-\mathrm{i}t\mathbf{H}}(f_{2^{k}}^{\circ}(\omega)-f_{2^{k-1}}^{\circ}(\omega))\end{equation} converges in $L_{t}^{r}\mathcal{W}_{x}^{\alpha,q}$. This can only converge to $e^{-\mathrm{i}t\mathbf{H}}f(\omega)$, and the same argument works for the space $\mathcal{W}^{\alpha,q}$. This completes the proof.
\end{proof}

Equation (\ref{nls2}) is a hamiltonian PDE with formally conserved mass $\|u\|_{L^{2}}^{2}$ and Hamiltonian \begin{equation}E(u)=\langle \mathbf{H}u,u\rangle\pm\frac{2}{p+1}\|u\|_{L^{p+1}}^{p+1}=\int_{\mathbb{R}^{n}}\bigg(|\nabla u|^{2}+|xu|^{2}\pm\frac{2}{p+1}|u|^{p+1}\bigg)\,\mathrm{d}x.\end{equation} Recall that $\mu=\mathbb{P}\circ f^{-1}$ is a probability measure on $\mathcal{S}'(\mathbb{R}^{2})$, the push-forward of $\mathbb{P}$ under $f$. In the defocusing case, for all $1<p<\infty$, we define the Gibbs measure of (\ref{nls2}) to be \begin{equation}\label{gbsm}\mathrm{d}\nu=\exp\bigg(-\frac{2}{p+1}\|u\|_{L^{p+1}}^{p+1}\bigg)\,\mathrm{d}\mu.\end{equation} Since the integrand in (\ref{gbsm}) is well-defined, bounded and positive, by Corollary \ref{prob}, we know $\nu$ is finite and mutually absolutely continuous with $\mu$. We also define the truncated measures\begin{equation}\mathrm{d}\nu_{2^{k}}=\exp\bigg(-\frac{2}{p+1}\|u_{2^{k}}^{\circ}\|_{L^{p+1}}^{p+1}\bigg)\,\mathrm{d}\mu.\end{equation} Since $\|u_{2^{k}}^{\circ}\|_{L^{p+1}}\to\|u\|_{L^{p+1}}$ a.e. in $\mu$, thanks to Corollary \ref{prob}, we know $\nu_{2^{k}}\to\nu$ in the strong sense that the total variance of $\nu_{2^{k}}-\nu$ tends to $0$.

In the focusing case, for $1<p<3$, we define the truncated measures $\mathrm{d}\nu_{2^{k}}=\rho_{2^{k}}\,\mathrm{d}\mu$ where \begin{equation}\label{cutoff}\rho_{2^{k}}(u)=\chi(\|u_{2^{k}}^{\circ}\|_{L^{2}}^{2}-\alpha_{2^{k}})\exp\bigg(\frac{2}{p+1}\|u_{2^{k}}^{\circ}\|_{L^{p+1}}^{p+1}\bigg).\end{equation} Here $\chi$ is some compactly supported continuous function on $\mathbb{R}$ that equals $1$ on a neighborhood of $0$, and\begin{equation}\alpha_{2^{k}}=\mathbb{E}\big(\|f_{2^{k}}^{\circ}(\omega)\|_{L^{2}}^{2}\big)=\sum_{j=0}^{2^{k}}\frac{1}{4j+2}.\end{equation} Clearly $\alpha_{2^{k}}\lesssim k$ for $k\geq 1$. We define the Gibbs measure $\nu$ as the limit of these $\nu_{2^{k}}$. More precisely, we have\begin{proposition}\label{intg}The functions $\rho_{2^{k}}$ converges to a function $\rho$ in $L^{r}(\mu)$ for all $1\leq r<\infty$. The measure $\mathrm{d}\nu=\rho\,\mathrm{d}\mu$ is finite and absolutely continuous with respect to $\mu$. We also know $\nu_{2^{k}}\to\nu$ in the strong sense that the total variance of $\nu_{2^{k}}-\nu$ tends to $0$. Finally, we can choose a countable number of $\chi_{(m)}$ so that the union of the supports of the corresponding Radon-Nikodym derivatives $\rho_{(m)}$ has full $\mu$ measure in $\mathcal{S}'(\mathbb{R}^{2})$. If we have fixed $\chi$, we will define $\nu$ to be the Gibbs measure of equation (\ref{nls2}).\end{proposition}

\begin{proof}
First we prove that $\rho_{2^{k}}$ converges a.e. in $\mu$, or equivalently, that $\rho_{2^{k}}(f(\omega))$ converges a.s. in $\mathbb{P}$. Consider \begin{equation}\|f_{2^{k}}^{\circ}(\omega)\|_{L^{2}}^{2}-\alpha_{2^{k}}=\sum_{j=0}^{k}\frac{|g_{j}(\omega)|^{2}-1}{4j+2},\end{equation}and see that it is a (partial) independent sum of random variables with zero mean and summable variance (the variance of $j$-th term is $\sim (j+1)^{-2}$), so it converges almost surely. Thus by the continuity of $\chi$, the first factor $\chi(\|f_{2^{k}}^{\circ}(\omega)\|_{L^{2}}^{2}-\alpha_{2^{k}})$ in $\rho_{2^{k}}(f(\omega))$ converges almost surely. Next, since $f_{2^{k}}^{\circ}(\omega)\to f(\omega)$ in $L^{p+1}$ for a.s. $\omega\in\Omega$, we know that the second factor also converges almost surely. Therefore, we have that a.e. in $\mu$, $\rho_{2^{k}}$ converges, say to some $\rho$.

To prove $\rho_{2^{k}}(f)\to\rho(f)$ in $L^{r}(\mathbb{P})$, we need some uniform integrability conditions. This is provided by the following large deviation estimate\begin{equation}\label{dev}\mathbb{P}\big(\|f_{2^{k}}^{\circ}(\omega)\|_{L^{2}}^{2}-\alpha_{2^{k}}\leq \beta,\,\,\,\,\|f_{2^{k}}^{\circ}(\omega)\|_{L^{p+1}}>A\big)\leq Ce^{-cA^{\delta}},\end{equation} for some $\delta>p+1$ and all large enough $A$, where $\beta$ is such that $\chi(z)=0$ for $|z|\geq\beta$. To prove (\ref{dev}) we may assume $A$ is sufficiently large, and set $k_{0}\in\mathbb{N}$ so that $2^{k_{0}}\sim e^{A^{\delta}}$ for some $\delta>0$ to be determined later.

First we prove (\ref{dev}) is true for $k\leq k_{0}+1$, with $\beta$ and $A$ on the left side replaced by $2\beta$ and $\frac{A}{2}$. In fact, by H\"{o}lder's inequality, if \begin{equation}\label{arg1}\|f_{2^{k}}^{\circ}(\omega)\|_{L^{2}}^{2}\leq \alpha_{2^{k}}+2\beta\lesssim k\lesssim A^{\delta},\,\,\,\,\,\|f_{2^{k}}^{\circ}(\omega)\|_{L^{p+1}}>\frac{A}{2},\end{equation} then\begin{equation}\|f_{2^{k}}^{\circ}(\omega)\|_{L^{q}}\gtrsim A^{\sigma},\,\,\,\,\,\sigma=\frac{(q-2)(p+1)-\delta(q-p-1)}{(p-1)q},\end{equation} under the assumption $p+1\leq q<\infty$. Since $2<q<\infty$, we know from Corollary \ref{prob} that\begin{equation}\label{arg2}\mathbb{P}\big(\|f_{2^{k}}^{\circ}(\omega)\|_{L^{q}}>A^{\sigma}\big)\leq Ce^{-cA^{2\sigma}}.\end{equation} If $1<p<3$, then for $q$ sufficiently large and $\delta$ sufficiently small, we have $2\sigma>p+1$, so (\ref{dev}) is true in this case.

Next we assume $k\geq k_{0}+2$. In this case we can prove\begin{equation}\label{eexp}\mathbb{P}\bigg(\|f_{2^{k}}^{\circ}(\omega)-f_{2^{k_{0}}}^{\circ}(\omega)\|_{L^{p+1}}>\frac{A}{2}\bigg)\leq c_{1}e^{-c_{2}e^{c_{3}A^{c_{4}}}}.\end{equation} In fact, since $f_{2^{k}}^{\circ}(\omega)-f_{2^{k_{0}}}^{\circ}(\omega)$ is of the form $\Pi f(\omega)$ as defined in Proposition \ref{lin2}, with the parameter $N\sim 2^{k_{0}}$, by Proposition \ref{lin2} we immediately get (\ref{eexp}) (notice $N\sim e^{A^{\delta}}$).

Now if $\|f_{2^{k}}^{\circ}(\omega)\|_{L^{2}}^{2}\leq \alpha_{k}+\beta$ and $\|f_{2^{k}}^{\circ}(\omega)\|_{L^{p+1}}>A$, then we have three possibilities. 

(1) If $\|f_{2^{k}}^{\circ}(\omega)-f_{2^{k_{0}}}^{\circ}(\omega)\|_{L^{p+1}}>\frac{A}{2}$, then we are already done, since this probability is controlled due to (\ref{eexp}).
 
 (2) If $\|f_{2^{k_{0}}}^{\circ}(\omega)\|_{L^{p+1}}>\frac{A}{2}$ and $\|f_{2^{k_{0}}}^{\circ}(\omega)\|_{L^{2}}^{2}\leq\alpha_{k_{0}}+2\beta$, Then we may set $k=k_{0}$ in the arguments from (\ref{arg1}) to (\ref{arg2}), and again get the desired bound.

(3) If $\|f_{2^{k}}^{\circ}(\omega)\|_{L^{2}}^{2}\leq \alpha_{2^{k}}+\beta$ as well as $\|f_{2^{k_{0}}}^{\circ}(\omega)\|_{L^{2}}^{2}> \alpha_{2^{k_{0}}}+2\beta$, then\begin{equation}\|f_{2^{k}}^{\circ}(\omega)\|_{L^{2}}^{2}-\|f_{2^{k_{0}}}^{\circ}(\omega)\|_{L^{2}}^{2}-(\alpha_{2^{k}}-\alpha_{2^{k_{0}}})\leq-\beta,\end{equation} or equivalently\begin{equation}Y=\sum_{j=2^{k_{0}}+1}^{2^{k}}\frac{1-|g_{j}|^{2}}{4j+2}\geq\beta.\end{equation} Notice that $Y$ is an independent sum with standard deviation \begin{equation} \kappa=\bigg(\sum_{j=2^{k_{0}}+1}^{2^{k}}\frac{1}{(4j+2)^{2}}\bigg)^{\frac{1}{2}}\lesssim 2^{-\frac{k_{0}}{2}}\leq c_{1}e^{-c_{2}A^{c_{3}}},\end{equation} we can compute \setlength\arraycolsep{2pt}
\begin{eqnarray}\label{fir}
\mathbb{E}\bigg(\exp\big(\frac{Y}{2\kappa}\big)\bigg)&=&\prod_{j=2^{k_{0}}+1}^{2^{k}}\mathbb{E}\bigg(\exp\big(\frac{\kappa(1-|g_{j}|^{2})}{2(4j+2)}\big)\bigg)\\ &=&\prod_{j=2^{k_{0}}+1}^{2^{k}}\bigg(e^{\frac{\kappa}{2(4j+2)}}\big(1+\frac{\kappa}{2(4j+2)}\big)^{-1}\bigg)\nonumber\\ \label{var} &\leq &\prod_{j=2^{k_{0}}+1}^{2^{k}}e^{\frac{c\theta_{j}^{2}}{4(4j+2)^{2}\kappa^{2}}}=e^{\frac{c}{4}}.\nonumber
\end{eqnarray}Here we have used the fact that $\mathbb{E}(e^{-\lambda|g|^{2}})=(1+\lambda)^{-1}$ when $\lambda>-1$, and $g$ is a normalized complex Gaussian; and that $e^{x}(1+x)^{-1}\leq e^{cx^{2}}$ for large $c$, and $0\leq x\leq\frac{1}{2}$. Therefore we have obtained\begin{equation}\mathbb{P}(Y>\beta)\leq e^{-c\kappa^{-1}}\leq c_{1}e^{-c_{2}e^{c_{3}A^{c_{4}}}}.\end{equation}
 
 This completes the proof of (\ref{dev}). The other conclusions now follow easily from this large deviation estimate, except the one regarding the support of $\rho$. We choose a sequence of cutoff functions $\chi_{(m)}$ so that $\chi_{(m)}\equiv 1$ on $[-\gamma_{m},\gamma_{m}]$ with $\gamma_{m}\uparrow\infty$.  By our previous discussions, after discarding null sets, the function $\rho_{(m)}$ will be nonzero wherever \begin{equation}\label{limitsmall}\lim_{k\to\infty}\big|\|f_{2^{k}}^{\circ}(\omega)\|_{L^{2}}^{2}-\alpha_{2^{k}}\big|\leq\gamma_{m}.\end{equation} Since this limit exists almost surely, and $\gamma_{m}\uparrow\infty$, we know almost surely, (\ref{limitsmall}) will hold for at least one $m$. So the union of support of these $\rho_{j}$ will have full $\mu$ measure.
 \end{proof}

Now in both defocusing and focusing case we have defined the Gibbs measure $\nu$ and the approximating measure $\nu_{2^{k}}$. They will be used in Section \ref{gwp} to obtain global well-posedness, and the invariance of $\nu$ will be proved in Section \ref{inv}.

\section{Multilinear Analysis in $\mathcal{X}^{\sigma,b}$ Spaces}\label{nonlinear}
First let us recall the hypercontractivity property of complex Gaussians. To make equations easier to write, we introduce the notation in which $u^{-}$ represents some element in $\{u,\bar{u}\}$ for any complex number $u$. This will be used throughout the rest of the paper. The first result about hypercontractivity was proved in Nelson \cite{Ne73}. Here we use a  formulation of this property taken from \cite{TT10}.
\begin{proposition}\label{hypercon}
Suppose $l,d\geq 1$, and a random variable $S$ has the following form\begin{equation}\label{form}S=\sum_{0\leq n_{1},\cdots,n_{l}\leq d}c_{n_{1},\cdots,n_{l}}\cdot g_{n_{1}}^{-}(\omega)\cdots g_{n_{l}}^{-}(\omega),\end{equation} where $c_{n_{1},\cdots,n_{l}}\in\mathbb{C}$, and the $(g_{n})_{0\leq n\leq d}$ are independent normalized complex Gaussians, then we have the estimate\[\big(\mathbb{E}|S|^{p}\big)^{\frac{1}{p}}\leq \sqrt{l+1}(p-1)^{\frac{l}{2}}\big(\mathbb{E}|S|^{2}\big)^{\frac{1}{2}},\] for all $p\geq 2$.
\end{proposition}
\begin{proof} This is basically a restatement of Proposition 2.4 in \cite{TT10}. There the authors required $n_{j}\geq 1$ and $n_{1}\leq\cdots\leq n_{l}$, but an easy modification will immediately settle this. The only difference is that here we may have $g_{n_{j}}$ or $\bar{g}_{n_{j}}$, but if we write $g_{n}=\frac{1}{\sqrt{2}}(\gamma_{n}+\mathrm{i}\tilde{\gamma}_{n})$ where $\gamma_{n}$ and $\tilde{\gamma}_{n}$ are mutually independent normalized real Gaussians, then $\bar{g}_{n}=\frac{1}{\sqrt{2}}(\gamma_{n}-\mathrm{i}\tilde{\gamma}_{n})$. So $S$ is again written as a linear combination of products of independent normalized real Gaussians. Then the result follows in the same way as \cite{TT10}.
\end{proof}
Next we want to adapt the result in Proposition \ref{hypercon} to our specific case to yield a large deviation bound on appropriate multilinear expressions of Gaussians. Namely, we have the following
\begin{proposition}\label{gaussianest}
Let $N_{1}\geq\cdots \geq N_{l}\geq 2$ be dyadic numbers such that $N_{1}\geq 10^{3}N_{2}$. Assume for $n\geq 0$ and $4n+2\leq10 N_{1}^{2}$, we have independent normalized complex Gaussians $\{w_{n}\}$. Also let $\varrho$ be any integer, and $\delta_{n_{1},\cdots,n_{l}}$ be arbitrary complex numbers with absolute value $\leq 1$. Define\begin{equation}\label{Xi}\Xi=\bigg\{(n_{1},\cdots,n_{l}): n_{j}\geq 0, \,\frac{1}{10}\leq\frac{4n_{j}+2}{N_{j}^{2}}\leq 10\,(1\leq j\leq l),\,\sum_{j=1}^{l}\epsilon_{j}(4n_{j}+2)=\varrho\bigg\}\end{equation} with $\epsilon_{j}=\pm 1$, then we have\begin{equation}\label{gau}\mathbb{P}\bigg(\bigg\{\bigg|\sum_{(n_{1},\cdots,n_{l})\in\Xi}\delta_{n_{1},\cdots,n_{l}}w_{n_{1}}^{-}(\omega)\cdots w_{n_{l}}^{-}(\omega)\bigg|>K\prod_{j=2}^{l}N_{j}\bigg\}\bigg)\leq c_{1}e^{-c_{2}K^{c_{3}}}.\end{equation} Here all the constants depends only on $l$.
\end{proposition}
\begin{proof}
We denote the sum on the left side of (\ref{gau}) by $S$. Using Proposition \ref{hypercon}, we can get \[\big(\mathbb{E}|S|^{p}\big)^{\frac{1}{p}}\leq \sqrt{l+1}(p-1)^{\frac{l}{2}}A,\] where we denote $A=\big(\mathbb{E}|S|^{2}\big)^{\frac{1}{2}}$. By Markov's inequality, we in particular have\[\mathbb{P}(|S|>KA)\leq (KA)^{-p}\cdot\mathbb{E}|S|^{p}\leq K^{-p}(l+1)^{\frac{p}{2}}(p-1)^{\frac{lp}{2}},\] for all $p\geq 2$. If $K\geq 2\sqrt{l+1}$, we may choose $p=1+K^{\frac{2}{l}}2^{-\frac{2}{l}}(l+1)^{-\frac{1}{l}}\geq 2$ in the above inequality to obtain \[\mathbb{P}(|S|>KA)\leq 2^{-p}\leq c_{1}e^{-c_{2}K^{c_{3}}}.\] By choosing the constants appropriately, we can guarantee that this also hold for $K<2\sqrt{l+1}$. Now what remains is to prove that $A\lesssim \prod_{j=2}^{l}N_{j}$, or equivalently \[\mathbb{E}|S|^{2}\lesssim\prod_{j=2}^{l}N_{j}^{2}.\] Now we expand the square to get\[\mathbb{E}|S|^{2}=\sum\delta_{n_{1},\cdots,n_{l}}\bar{\delta}_{m_{1},\cdots,m_{l}}\Delta_{n_{1},\cdots,n_{l},m_{1},\cdots,m_{l}},\] where the sum is taken over all $(n_{1},\cdots,n_{l},m_{1},\cdots,m_{l})\in\Xi\times\Xi$, and\[\Delta_{n_{1},\cdots,n_{l},m_{1},\cdots,m_{l}}=\mathbb{E}\bigg(\prod_{j=1}^{l}w_{n_{j}}^{-}w_{m_{j}}^{-}\bigg).\] Since each of the $\delta$'s and $\Delta$'s has absolute value $\lesssim 1$ (depending on $l$) in any possible case, we will be done once we establish the following\begin{equation}\label{comb}\#\bigg\{(n_{1},\cdots,n_{l},m_{1},\cdots,m_{l})\in\Xi\times\Xi:\Delta_{n_{1},\cdots,n_{l},m_{1},\cdots,m_{l}}\neq 0\bigg\}\lesssim \prod_{j=2}^{l}N_{j}^{2}.\end{equation}

The crucial observation is that, due to the independence assumption, if the expectation $\Delta$ is nonzero, then any integer that appears in $(n_{1},\cdots,n_{l},m_{1},\cdots,m_{l})$, must appear at least twice. Next, due to our assumption $N_{1}\geq 10^{3}N_{2}$, we know $n_{1}= m_{1}$, and any integer that appears in $(n_{2},\cdots,n_{l},m_{2},\cdots,m_{l})$ must appear at least twice. If we permute all the different integers appearing in this $(2l-2)$-tuple as $\sigma_{1}>\sigma_{2}>\cdots>\sigma_{r}$, then with $r$ and all $\sigma_{i}$ fixed, we have at most $(2l-2)^{2l-2}$ choices for the $(2l-2)$-tuple; also due to the linear relation enjoyed by both $(n_{1},\cdots,n_{l})$ and $(m_{1},\cdots,m_{l})$, the $(2l-2)$-tuple will uniquely determine $n_{1}$ and $m_{1}$. Thus we only need to show for each possible $1\leq r\leq 2l$, there are $\lesssim\prod_{j=2}^{l}N_{j}^{2}$ choices for $(\sigma_{1},\cdots,\sigma_{r})$. Now for each $1\leq i\leq r$, since each $\sigma_{j}(1\leq j\leq i)$ appear in the $(2l-2)$-tuple at least twice (and different $\sigma_{j}$ cannot appear at the same place), here must exist $1\leq j_{1}\leq i<i+1\leq j_{2}$ such that $\sigma_{j_{1}}\in\{n_{j_{2}},m_{j_{2}}\}$. This implies \[4\sigma_{i}+2\leq 4\sigma_{j_{1}}+2\lesssim N_{j_{2}}^{2}\lesssim N_{i+1}^{2},\] so for each $1\leq i\leq r$, there is at most $N_{i+1}^{2}$ choices for $\sigma_{i}$, and necessarily $1\leq r\leq l-1$. Therefore, for each $r\leq l-1$, we have at most\[\lesssim\prod_{i=1}^{r}N_{i+1}^{2}\lesssim\prod_{j=2}^{l}N_{j}^{2}\] choices for $(\sigma_{1},\cdots,\sigma_{r})$. This completes the proof.
\end{proof}
\begin{proposition}\label{longest}
Suppose $p\geq 3$ is an odd integer. We choose $\sigma$ and $b$ so that $0<\sigma<1$ is sufficiently close to $1$ depending on $p$, and $1>b>\frac{1}{2}$ is sufficiently close to $\frac{1}{2}$ depending on $\sigma$ and $p$. Let $T$ be small enough depending on $b,\sigma$ and $p$. Then we can find a set $\Omega_{T}\subset\Omega$ and a positive number $\theta$ that only depends on $\sigma,b$ and $T$, so that $\mathbb{P}(\Omega_{T})\leq c_{1}e^{-c_{2}T^{-c_{3}}}$, and that the following holds: for any $t_{0}\in\mathbb{R}$ and $\omega\in\Omega_{T}^{c}$, if for each $1\leq j\leq p$, a function $u_{j}$ on $[-T,T]\times\mathbb{R}^{2}$ is given by either \begin{equation}\label{pro}u_{j}=e^{-\mathrm{i}(t+t_{0})\mathbf{H}}f(\omega),\end{equation} or
\begin{equation}\label{det}\|u_{j}\|_{\mathcal{X}^{\sigma,b,T}}\lesssim 1,\end{equation}then we have \begin{equation}\label{mainpro}\|u_{1}^{-}\cdots u_{p}^{-}\|_{\mathcal{X}^{\sigma,b-1,T}}\lesssim T^{\theta}.\end{equation} Here all the constants will depend on $\sigma,b$ and $p$.
\end{proposition}
\begin{proof}
 In what follows, if an estimate holds for $\omega$ outside a set with measure $\epsilon$, we simply say it holds ``with exceptional probability $\epsilon$''. We will use various exponents $q_{j}$, and each of them will remain the \emph{same} throughout the proof. First we can use Lemma \ref{linear00} to estimate\[\|u_{1}^{-}\cdots u_{p}^{-}\|_{\mathcal{X}^{\sigma,b-1,T}}\lesssim T^{2b-1}\|u_{1}^{-}\cdots u_{p}^{-}\|_{\mathcal{X}^{\sigma,3b-2,T}},\] since $-\frac{1}{2}<b-1<3b-2<\frac{1}{2}$. Thus we only need to prove\begin{equation}\label{firstofall}\|u_{1}^{-}\cdots u_{p}^{-}\|_{\mathcal{X}^{\sigma,3b-2,T}}\lesssim T^{\frac{1}{2}-b},\end{equation} with exceptional probability $\leq c_{1}e^{-c_{2}T^{-c_{3}}}$. Recall the Littlewood-Paley projections (\ref{litp}), we have\begin{equation}u=\sum_{N\geq 2}u_{N},\end{equation} where for simplicity we write $u_{N}=\Delta_{N}u$. Thus we only need to estimate the terms (note $(u_{N})^{-}=(u^{-})_{N}$ since the Littlewood-Paley projectors are real)\begin{equation}\nonumber\prod_{j=1}^{p}(u_{j})_{N_{j}}^{-},\end{equation} where we have fixed a choice between $u_{j}$ and $\bar{u}_{j}$, and between (\ref{pro}) and (\ref{det}), for each $u_{j}$. Define \[A=\{1\leq j\leq p:u_{j}\,\,\mathrm{given}\,\,\mathrm{by}\,\,(\ref{pro})\},\] and \[B=\{1\leq j\leq p:u_{j}\,\,\mathrm{given}\,\,\mathrm{by}\,\,(\ref{det})\}.\] Let \begin{equation}\mathcal{A}=\bigg\{(N_{1},\cdots,N_{p}):\exists j\in B,\,\,N_{j}>10^{3}\sum_{i\neq j}N_{i}\bigg\}.\end{equation} We first consider the sum of terms with $(N_{1},\cdots, N_{p})\in \mathcal{A}$, and rewrite it as\begin{equation}\sum_{j\in B}\sum_{(N_{i})_{i\neq j}}\prod_{i\neq j}(u_{i})_{N_{i}}^{-}\cdot\bigg(\sum_{N_{j}>10^{3}\sum_{i\neq j}N_{i}}(u_{j})_{N_{j}}^{-}\bigg).\end{equation} To bound this expression we only need to consider a fixed $j_{0}\in B$, and without loss of generality, we may assume $j_{0}=p$. For each $(N_{1},\cdots,N_{p-1})$ if we write \begin{equation}u_{p}^{hi}= \sum_{N_{p}>10^{3}\sum_{i=1}^{p-1}N_{i}}(u_{p})_{N_{p}},\end{equation} then we only need to prove\begin{equation}\label{subtogether}\mathfrak{S}:=\|(u_{1})_{N_{1}}^{-}\cdots(u_{p-1})_{N_{p-1}}^{-}(u_{p}^{hi})^{-}\|_{\mathcal{X}^{\sigma,3b-2,T}}\lesssim T^{\frac{1}{2}-b}(\max_{j<p}N_{j})^{-\theta},\end{equation} for some $\theta>0$, with exceptional probability $\leq c_{1}e^{-c_{2}T^{-c_{3}}(\max_{j<p}N_{j})^{c_{4}}}$ (note that when we take the sum over all $(N_{1},\cdots, N_{p-1})$, we still get an expression $\leq c_{1}e^{-c_{2}T^{-c_{3}}}$). 

To prove (\ref{subtogether}), we use Propositions \ref{sobolev} and \ref{strichartz} to estimate (for simplicity, we shall omit the spacetime domain $[-T,T]\times\mathbb{R}^{2}$ in the following estimates, but one should keep in mind that we are working on a very short time interval)\setlength\arraycolsep{2pt}
\begin{eqnarray}
\label{beingone}\mathfrak{S}
&\lesssim&\|(u_{1})_{N_{1}}^{-}\cdots(u_{p-1})_{N_{p-1}}^{-}(u_{p}^{hi})^{-}\|_{L_{t}^{q_{1}}\mathcal{W}_{x}^{\sigma,q_{1}}}\\
&\lesssim &\|(u_{p}^{hi})^{-}\|_{L_{t}^{4}\mathcal{W}_{x}^{\sigma,4}}\prod_{j=1}^{p-1}\|(u_{j})_{N_{j}}^{-}\|_{L_{t,x}^{q_{2}}}\nonumber\\
&+&\sum_{j=1}^{p-1}\|(u_{p}^{hi})^{-}\|_{L_{t,x}^{4}}\|(u_{j})_{N_{j}}^{-}\|_{L_{t}^{q_{2}}\mathcal{W}_{x}^{\sigma,q_{2}}}\prod_{j\neq i<p}\|(u_{i})_{N_{i}}^{-}\|_{L_{t,x}^{q_{2}}}\nonumber\\
&\lesssim & \|(u_{p}^{hi})^{-}\|_{L_{t}^{4}\mathcal{W}_{x}^{\sigma,4}}\prod_{j=1}^{p-1}\|(u_{j})_{N_{j}}^{-}\|_{L_{t,x}^{q_{2}}}\nonumber\\
\label{beingtwo}&+&\sum_{j=1}^{p-1}N_{j}^{\sigma}\|(u_{p}^{hi})^{-}\|_{L_{t,x}^{4}}\prod_{i=1}^{p-1}\|(u_{i})_{N_{i}}^{-}\|_{L_{t,x}^{q_{2}}}\\ \label{beingthree}&\lesssim & \|u_{p}^{-}\|_{L_{t}^{4}\mathcal{W}_{x}^{\sigma,4}}\prod_{j=1}^{p-1}\|(u_{j})_{N_{j}}^{-}\|_{L_{t,x}^{q_{2}}}
\\
\label{beingfour}&\lesssim & \prod_{j=2}^{p}\|(u_{j})_{N_{j}}^{-}\|_{L_{t,x}^{q_{2}}},
\end{eqnarray} where in (\ref{beingtwo}) and (\ref{beingthree}) we have used Corollary \ref{sobbb} (recall the definition of $u_{p}^{hi}$). In (\ref{beingfour}) we have used Proposition \ref{strichartz} and the asumption that $p\in B$. For the parameters, we choose $q_{1}>\frac{4}{3}$ and sufficiently close to $\frac{4}{3}$ depending on $p$, and $\frac{p-1}{q_{2}}=\frac{1}{q_{1}}-\frac{1}{4}$, and check that (\ref{beingone}) indeed hold, provided $b$ is sufficiently close to $\frac{1}{2}$, depending on $q_{1}$ (see Proposition \ref{strichartz}, with $b$ there replaced by $3b-1$). 

Now we proceed to analyze the expression (\ref{beingfour}). Choose $1\leq j\leq p-1$ so that $N_{j}=\max_{i<p}N_{i}$. If $j\in B$, then from Corollary \ref{sobbb} and Proposition \ref{strichartz} we have\begin{equation}\label{determin}\|(u_{j})_{N_{j}}^{-}\|_{L_{t,x}^{q_{2}}}\lesssim N_{j}^{-\epsilon}\|u_{j}\|_{L_{t}^{q_{2}}\mathcal{W}_{x}^{\epsilon,q_{2}}}\lesssim N_{j}^{-\epsilon}\|u_{j}\|_{\mathcal{X}^{\sigma,b,T}}\lesssim N_{j}^{-\epsilon},\end{equation} provided $\sigma-\epsilon>1-\frac{4}{q_{2}}$ (note $q_{2}>4$ from our choice of exponents above). This can be achieved if $\epsilon$ is small enough depending on $q_{2}$, and $\sigma$ is sufficiently close to $1$ depending on $q_{2}$ and $\epsilon$. If instead $j\in A$, then from Corollary \ref{sobbb} we have \begin{equation}\label{probabli}\|(u_{j})_{N_{j}}^{-}\|_{L_{t,x}^{q_{2}}}\lesssim N_{j}^{-\epsilon}\|u_{j}\|_{L_{t}^{q_{2}}\mathcal{W}_{x}^{\epsilon,q_{2}}}=N_{j}^{-\epsilon}\|e^{-\mathrm{i}(t+t_{0})\mathbf{H}}f(\omega)\|_{L_{t}^{q_{2}}\mathcal{W}_{x}^{\epsilon,q_{2}}}.\end{equation} The norm in the last expression equals the $L_{t}^{q_{2}}\mathcal{W}_{x}^{\epsilon,q_{2}}$ norm of $e^{-\mathrm{i}t\mathbf{H}}f(\omega)$ on the interval $[t_{0}-T,t_{0}+T]$. Since $T<1$, we may expand this interval to an interval with length $2\pi$. Since $e^{-\mathrm{i}t\mathbf{H}}f(\omega)$ has period $2\pi$ in $T$, we may replace the enlarged norm by the norm on $[-\pi,\pi]$. Then we could use Corollary \ref{prob} to bound\begin{equation}\label{alltime}N_{j}^{-\epsilon}\|e^{-\mathrm{i}(t+t_{0})\mathbf{H}}f(\omega)\|_{L_{t}^{q_{2}}\mathcal{W}_{x}^{\epsilon,q_{2}}}\lesssim T^{\frac{1}{10p}(\frac{1}{2}-b)}N_{j}^{-\frac{\epsilon}{2}}\end{equation} for all $t_{0}$, with exceptional probability $\leq c_{1}e^{-c_{2}T^{-c_{3}}N_{j}^{c_{4}}}$, provided $0<\epsilon<\frac{2}{q_{2}}$. Therefore in each case we have\begin{equation}\|(u_{j})_{N_{j}}^{-}\|_{L_{t,x}^{q_{2}}}\lesssim T^{\frac{1}{10p}(\frac{1}{2}-b)}N_{j}^{-\theta},\end{equation} with exceptional probability $\leq c_{1}e^{-c_{2}T^{-c_{3}}N_{j}^{c_{4}}}$, for some $\theta>0$.

Then we treat the terms with $i\neq j$. If $i\in B$, we can use Proposition \ref{strichartz} to bound $\|(u_{i})_{N_{i}}^{-}\|_{L_{t,x}^{q_{2}}}\lesssim 1$; if $i\in A$, we can use Corollary \ref{prob} to bound $\|(u_{i})_{N_{i}}^{-}\|_{L_{t,x}^{q_{2}}}\lesssim T^{\frac{1}{10p}(\frac{1}{2}-b)}N_{j}^{\frac{\theta}{10p}}$ for all $t_{0}$, with exceptional probability $\leq c_{1}e^{-c_{2}T^{-c_{3}}N_{j}^{c_{4}}}$. Putting these together, we have shown\begin{equation}\label{strong}(\ref{beingfour})\lesssim T^{\frac{1}{2}-b}(\max_{j<p}N_{j})^{-\theta}\end{equation} for some $\theta>0$, with exceptional probability $\leq c_{1}e^{-c_{2}T^{-c_{3}}(\max_{j< p}N_{j})^{c_{4}}}$. This takes care of the sum of terms with $(N_{1},\cdots, N_{p})\in \mathcal{A}$.

For $(N_{1},\cdots, N_{p})\not\in \mathcal{A}$, we are going to prove\begin{equation}\label{separate}J=\|v_{1}^{-}\cdots v_{p}^{-}\|_{\mathcal{X}^{\sigma, 3b-2,T}}\lesssim T^{\frac{1}{2}-b}(\max_{j\geq 1}N_{j})^{-\theta},\end{equation} where $v_{j}=(u_{j})_{N_{j}}$, with exceptional probability $\leq  c_{1}e^{-c_{2}T^{-c_{3}}(\max_{j\geq 1}N_{j})^{c_{4}}}$. This, together with the analysis above, clearly implies (\ref{mainpro}). Now without loss of generality, assume $N_{1}=\max_{j\geq 1}N_{j}$. If $1\in B$, then we have $N_{1}\sim \max_{j\geq 2}N_{j}$. By switching the role of $1$ and $p$ in the argument above and replacing $u_{1}^{hi}$ by $v_{1}$ (note $v_{1}$ also satisfy the estimates about $u_{1}^{hi}$ that we would use), we can prove (\ref{strong}) with the role of $1$ and $p$ switched. Since $N_{1}\sim\max_{j\geq 2}N_{j}$, this also proves (\ref{separate}).

Now we assume that $N_{1}=\max_{j\geq 1}N_{j}$ and $1\in A$. If $N_{1}\lesssim N_{j_{0}}^{\frac{1+\sigma}{3\sigma-1}}$ (note this exponent is $>1$) for some $j_{0}\geq 2$, then we may assume $j_{0}=2$. Now use the same arguments as in (\ref{beingone}) (but with different exponents), we have \setlength\arraycolsep{2pt}
\begin{eqnarray}
\label{phone}J
&\lesssim&\|v_{1}^{-}v_{2}^{-}\cdots v_{p}^{-}\|_{L_{t}^{q_{1}}\mathcal{W}_{x}^{\sigma,q_{1}}}\\
&\lesssim &(\|v_{1}^{-}\|_{L_{t}^{4}\mathcal{W}_{x}^{\sigma,4}}\|v_{2}^{-}\|_{L_{t,x}^{4}}+\|v_{1}^{-}\|_{L_{t,x}^{2}}\|v_{2}^{-}\|_{L_{t}^{4}\mathcal{W}_{x}^{\sigma,4}})\prod_{j=3}^{p}\|v_{j}^{-}\|_{L_{t,x}^{q_{4}}}\nonumber\\
&+&\sum_{j=3}^{p}\|v_{1}^{-}\|_{L_{t,x}^{4}}\|v_{2}^{-}\|_{L_{t,x}^{4}}\|v_{i}^{-}\|_{L_{t}^{q_{4}}\mathcal{W}_{x}^{\sigma,q_{4}}}\prod_{3\leq i\neq j}\|v_{j}^{-}\|_{L_{t,x}^{q_{4}}}\nonumber\\
\label{phtwo}&\lesssim &\bigg(\sum_{j=1}^{p}N_{j}^{\sigma}\bigg)\|v_{1}^{-}\|_{L_{t,x}^{4}}\|v_{2}^{-}\|_{L_{t,x}^{4}}\prod_{j=3}^{p}\|v_{j}^{-}\|_{L_{t,x}^{q_{4}}}\\
\label{phthree} &\lesssim&N_{1}^{\frac{1+\sigma}{4}}N_{2}^{\frac{1+\sigma}{4}}\|v_{1}^{-}\|_{L_{t,x}^{4}}\|v_{2}^{-}\|_{L_{t,x}^{4}}\prod_{j=3}^{p}\|v_{j}^{-}\|_{L_{t,x}^{q_{4}}}\\
\label{phfour}&\lesssim& \|v_{1}^{-}\|_{L_{t}^{4}\mathcal{W}_{x}^{\frac{1+\sigma}{4},4}}\|v_{2}^{-}\|_{L_{t}^{4}\mathcal{W}_{x}^{\frac{1+\sigma}{4},4}}\prod_{j=3}^{p}\|v_{j}^{-}\|_{L_{t,x}^{q_{4}}},
\end{eqnarray} where $\frac{p-2}{q_{4}}=\frac{1}{q_{1}}-\frac{1}{2}$ and $q_{4}>4$. Here in (\ref{phtwo}) and (\ref{phfour}) we have used Corollary \ref{sobbb} and the fact that $v_{j}=(u_{j})_{N_{j}}$, while in (\ref{phthree}) we have used $N_{j}\lesssim N_{1}\lesssim N_{2}^{\frac{1+\sigma}{3\sigma-1}}$ for all $j$.

Now we analyze the expression (\ref{phfour}). If $2\in B$, then by Corollary \ref{sobbb} and Proposition \ref{strichartz} we have (note $N_{1}\lesssim N_{2}^{2}$ when $\sigma>\frac{3}{5}$)\begin{equation}\|v_{2}^{-}\|_{L_{t}^{4}\mathcal{W}_{x}^{\frac{1+\sigma}{4},4}}\lesssim N_{1}^{-\frac{1}{24}}\|v_{2}^{-}\|_{L_{t}^{6}\mathcal{W}_{x}^{\frac{2\sigma}{3},4}}\lesssim N_{1}^{-\frac{1}{24}}\|u_{2}\|_{\mathcal{X}^{\sigma,b,T}}\lesssim N_{1}^{-\frac{1}{24}},\end{equation} provided $\frac{2\sigma}{3}>\frac{1+\sigma}{4}+\frac{1}{12}$ and $\sigma>\frac{2\sigma}{3}+\frac{1}{6}$, which is true for $\sigma>\frac{4}{5}$. If $2\in A$ (which is the case for $1$), we can use the arguments from (\ref{probabli}) to (\ref{alltime}) to get\begin{equation}\|v_{2}^{-}\|_{L_{t}^{4}\mathcal{W}_{x}^{\frac{1+\sigma}{4},4}}\lesssim N_{1}^{-\frac{1-\sigma}{16}}\|u_{2}^{-}\|_{L_{t}^{4}\mathcal{W}_{x}^{\frac{3+\sigma}{8},4}}\lesssim T^{\frac{1}{10p}(\frac{1}{2}-b)}N_{1}^{-\frac{1-\sigma}{32}}\end{equation} for all $t_{0}$, with exceptional probability $\leq c_{1}e^{-c_{2}T^{-c_{3}}N_{1}^{c_{4}}}$, thanks to Corollary \ref{prob}, and the hypothesis $\sigma<1$ (hence $\frac{3+\sigma}{8}<\frac{1}{2}$).

Then we treat the terms with $j\geq 3$. If $j\in B$, we can use Proposition \ref{strichartz} to bound $\|v_{j}^{-}\|_{L_{t,x}^{q_{4}}}\lesssim 1$; if $j\in A$, we can use Corollary \ref{prob} to bound $\|v_{j}^{-}\|_{L_{t,x}^{q_{4}}}\lesssim T^{\frac{1}{10p}(\frac{1}{2}-b)}N_{1}^{\frac{1-\sigma}{100p}}$ for all $t_{0}$, with exceptional probability $\leq c_{1}e^{-c_{2}T^{-c_{3}}N_{1}^{c_{4}}}$. Putting these together, we have proved\begin{equation}(\ref{phfour})\lesssim T^{\frac{1}{2}-b}N_{1}^{-\theta},\end{equation} with exceptional probability $\leq c_{1}e^{-c_{2}T^{-c_{3}}N_{1}^{c_{4}}}$ for some $\theta>0$. Thus we have proved (\ref{separate}) in this case.

In the final case, we assume that $N_{1}>(10p)^{3}(\max_{j\geq 2}N_{j})^{\frac{1+\sigma}{3\sigma-1}}$, which in particular implies $N_{1}>10^{3}\sum_{j\geq 1}N_{j}$, and that $1\in A$. For each $j\in B$, by definition we can extend $u_{j}$ to be a function on $\mathbb{R}\times\mathbb{R}^{2}$ (still denoted by $u_{j}$) with $\mathcal{X}^{\sigma,b}$ norm $\lesssim 1$. The relation $v_{j}=(u_{j})_{N_{j}}$ also extends to $t\in\mathbb{R}$, giving an extension of $v_{j}$ also. Choose $\zeta_{0}$ smooth, supported on $[-2,2]$ and equals $1$ on $[-1,1]$ and define $\zeta(t)=\zeta_{0}(T^{-1}t)$. We are to prove \begin{equation}\label{separate3}\|\zeta\cdot v_{1}^{-}\cdots v_{p}^{-}\|_{\mathcal{X}^{\sigma,3b-2}}\lesssim T^{\frac{1}{2}-b}N_{1}^{-\theta}\end{equation}for the extended $v_{j}$, with exceptional probability $\leq c_{1}e^{-c_{2}T^{-c_{3}}N_{1}^{c_{4}}}$. For a function $w$ on $\mathbb{R}\times\mathbb{R}^{2}$ radial in $x$, we split $w=w_{ne}+w_{fa}$, with\begin{equation}\mathcal{F}_{t}\langle w_{ne},e_{k}\rangle(\tau)=\chi_{\{|\tau+4k+2|\leq N_{1}^{\gamma}\}}\cdot \mathcal{F}_{t}\langle w,e_{k}\rangle(\tau),\end{equation} and $w_{fa}$ by replacing the $\leq$ by $>$. We now split the product in (\ref{separate3}) into $fa$ and $ne$ parts and estimate them separately.

We first estimate the $fa$ part of product as (due to the presence of $\zeta$, we can work on time interval $[-2T,2T]$ in the time-Lebesgue norms below, thus gaining powers in $T$)\setlength\arraycolsep{2pt}
\begin{eqnarray}
\label{sb2one}\|(\zeta\cdot v_{1}^{-}\cdots v_{p}^{-})_{fa}\|_{\mathcal{X}^{\sigma,3b-2}}&\lesssim&N_{1}^{-\frac{\gamma}{36}}\|\zeta\cdot v_{1}^{-}\cdots v_{p}^{-}\|_{\mathcal{X}^{\sigma,-\frac{4}{9}}}\\
\label{sb2two}&\lesssim &N_{1}^{-\frac{\gamma}{36}}\|\zeta\cdot v_{1}^{-}\cdots v_{p}^{-}\|_{L_{t}^{\frac{3}{2}}\mathcal{W}_{x}^{\sigma,\frac{3}{2}}}\\
\label{sb2three}&\lesssim &N_{1}^{\sigma-\frac{\gamma}{36}}\prod_{i}\|v_{i}^{-}\|_{L_{t,x}^{\frac{3p}{2}}}.
\end{eqnarray}
Here in (\ref{sb2one}) we have used the definition of the $fa$-projection and that $b$ is close to $\frac{1}{2}$ (in particular, $b<\frac{1}{2}+\frac{1}{108}$); in (\ref{sb2two}) we have used Proposition \ref{strichartz}; in (\ref{sb2three}) we have combined Corollary \ref{sobbb} and Proposition \ref{sobolev}. Now for each $i$, if $i\in B$ then (provided $\sigma$ is close to $1$ depending on $p$) $\|v_{i}^{-}\|_{L_{t,x}^{\frac{3p}{2}}}\lesssim \|v_{i}\|_{\mathcal{X}^{\sigma,b}}\lesssim 1$. If $i\in A$ (such as $i=1$) we have $\|v_{i}^{-}\|_{L_{t,x}^\frac{3p}{2}}\lesssim T^{\frac{1}{10p}(\frac{1}{2}-b)}N_{1}^{\frac{1}{p}}$ for all $t_{0}$, with exceptional probability $\leq c_{1}e^{-c_{2}T^{-c_{3}}N_{1}^{c_{4}}}$. Therefore, we have $(\ref{sb2three})\lesssim T^{\frac{1}{2}-b}N_{1}^{-\theta}$ with exceptional probability $\leq c_{1}e^{-c_{2}T^{-c_{3}}N_{1}^{c_{4}}}$, provided $\gamma>108$.

Now we estimate the $ne$ part of the product. Choose $v_{0}$ so that $\|v_{0}\|_{\mathcal{X}^{0,2-3b}}\lesssim 1$. Since we are taking the $ne$ part, we may assume $v_{0}=v_{0,ne}$. The aim is to estimate $|\mathfrak{J}|$ (recall $H$ is self-adjoint), where\begin{equation}\mathfrak{J}=\int_{\mathbb{R}\times\mathbb{R}^{2}}v_{1}^{-}\cdots v_{p}^{-}\cdot (\zeta \mathbf{H}^{\frac{\sigma}{2}}\bar{v}_{0}).\end{equation} We use Lemma \ref{represent} to write down\begin{equation}\label{rep}v_{j}(x,t)=\int_{\mathbb{R}}\phi_{j}(\lambda_{j})e^{\mathrm{i}\lambda_{j} t}\sum_{k}a_{\lambda_{j}}^{j}(k)e^{-\mathrm{i}(4k+2)t}e_{k}(x)\,\mathrm{d}\lambda_{j}\end{equation} for $j\in B\cup\{0\}$, where the parameters satisfy\begin{equation}\label{sqsum}\sum_{k}|a_{\lambda_{0}}^{0}(k)|^{2}\lesssim 1\end{equation} for each $\lambda_{0}$. Since $v_{0}=v_{0,ne}$, we also have $\|\phi_{0}\|_{L^{1}}\lesssim N_{1}^{3\gamma(b-\frac{1}{2})}$. For $j\in B$, since $v_{j}=(u_{j})_{N_{j}}$, we know $a_{\lambda_{j}}^{j}(n_{j})=0$ unless $\frac{1}{10}\leq\frac{4n_{j}+2}{N_{j}^{2}}\leq 10$, and hence\begin{equation}\label{squarectrl}\sum_{4n_{j}+2\sim N_{j}^{2}}|a_{\lambda_{j}}^{j}(k)|^{2}\lesssim N_{j}^{-2\sigma}.\end{equation} Also since $b>\frac{1}{2}$, we have $\|\phi_{j}\|_{L^{1}}\lesssim 1$.

 For the sake of convenience, in the following proof, we shall use $v^{\sim}(n,\tau)$ to denote $\mathcal{F}_{t}\langle v,e_{n}\rangle(\tau)$. Thus from (\ref{rep}) we have \begin{equation}\label{gamma1}v_{j}^{\sim}(n_{j},\tau_{j})=(2\pi)^{\frac{1}{2}} a_{\tau_{j}+4n_{j}+2}^{j}(n_{j})\phi_{j}(\tau_{j}+4n_{j}+2)\end{equation} for $j\in B$. If $j\in A$ we have \begin{equation}\label{gamma2}v_{j}^{\sim}(n_{j},\tau_{j})=(2\pi)^{\frac{1}{2}}e^{-\mathrm{i}(4n_{j}+2)t_{0}}\frac{\theta_{j}(n_{j})g_{n_{j}}(\omega)}{\sqrt{4n_{j}+2}}\delta(\tau_{j}+4n_{j}+2),\end{equation} where \[\theta_{j}(n_{j})=\eta\big(\frac{2(4n_{j}+2)}{N_{j}^{2}}\big)-\eta\big(\frac{4(4n_{j}+2)}{N_{j}^{2}}\big).\]Clearly $|\theta_{j}|\leq 2$, and $\theta_{j}\neq 0$ only when $\frac{1}{10}\leq\frac{4n_{j}+2}{N_{j}^{2}}\leq 10$ (note we have fixed $N_{j}$). Finally, for $j=0$ we have (we may assume $\zeta$ is real)  \setlength\arraycolsep{2pt}\begin{eqnarray}\label{gamma3}(\zeta \mathbf{H}^{\frac{\sigma}{2}}v_{0})^{\sim}(n_{0},\tau_{0})&=&(4n_{0}+2)^{\frac{\sigma}{2}}\cdot \int_{\mathbb{R}} a_{\varrho_{0}+4n_{0}+2}^{0}(n_{0})\\
 &\times&\phi_{0}(\varrho_{0}+4n_{0}+2)\hat{\zeta}(\tau_{0}-\varrho_{0})\,\mathrm{d}\varrho_{0}.\nonumber\end{eqnarray}
 
 We write $\gamma_{j}=v_{j}^{\sim}$ for $j\geq 1$, and $\gamma_{0}=(\zeta \mathbf{H}^{\frac{\sigma}{2}}v_{0})^{\sim}$. From the rules of Fourier transform and orthogonality of $e_{k}$, we have
 \begin{equation}\mathfrak{J}=(2\pi)^{-\frac{p-2}{2}}\sum_{n_{1},\cdots,n_{p},n_{0}}\kappa_{n_{1},\cdots,n_{p}}^{n_{0}}\int_{\mathbb{D}}\prod_{j=0}^{p}(\gamma_{j}(n_{j},\tau_{j}))^{-}\,\mathrm{d}\tau_{1}\cdots\mathrm{d}\tau_{p},\end{equation} where \begin{equation}\kappa_{n_{1},\cdots,n_{p}}^{n_{0}}=\int_{\mathbb{R}^{2}}e_{n_{1}}(x)\cdots e_{n_{p}}(x)e_{n_{0}}(x)\,\mathrm{d}x,\end{equation} and\begin{equation}\mathbb{D}=\bigg\{(\tau_{1},\cdots,\tau_{p},\tau_{0}):\tau_{0}=\sum_{j=1}^{p}\epsilon_{j}\tau_{j}\bigg\},\end{equation} with $\epsilon_{j}=\pm 1$ depending on the choice of $v_{j}$ or $\bar{v}_{j}$. We notice that $\epsilon_{j}=1$ if and only if the corresponding $\gamma_{j}^{-}$ takes $\gamma_{j}$. Now plug in (\ref{gamma1}), (\ref{gamma2}), and (\ref{gamma3}), and use the change of variables $\lambda_{j}=\tau_{j}+4n_{j}+2$ for $j\in B$, $\lambda_{0}=\varrho_{0}+4n_{0}+2$, we get
 \setlength\arraycolsep{2pt}
\begin{eqnarray}\mathfrak{J}&=&2\pi\sum_{n_{1},\cdots,n_{p},n_{0}}\kappa_{n_{1},\cdots,n_{p}}^{n_{0}}\int\prod_{j\in B\cup\{0\}}\mathrm{d}\lambda_{j}\\
&\times &\prod_{j\in B}\phi_{j}(\lambda_{j})a_{\lambda_{j}}^{j}(n_{j})^{-}\prod_{j\in A}\frac{\theta_{j}(n_{j})g_{n_{j}}^{-}(\omega)}{\sqrt{4n_{j}+2}}\cdot a_{\lambda_{0}}^{0}(n_{0})^{-}\phi_{0}(\lambda_{0})\nonumber\\
&\times&\hat{\zeta}\bigg(\sum_{j\in B}\epsilon_{j}\lambda_{j}-\lambda_{0}-\sum_{j=1}^{p}\epsilon_{j}(4n_{j}+2)+(4n_{0}+2)\bigg)^{-}\nonumber\\
&\times&(4n_{0}+2)^{\frac{\sigma}{2}}\exp\bigg(-\mathrm{i}t_{0}\sum_{j\in A}(4n_{j}+2)\epsilon_{j}\bigg).\nonumber
\end{eqnarray} Here the terms corresponding to $j\in A$ are delta functions and have already been encorporated in the final expression. Let $\varrho=(4n_{0}+2)-\sum_{j=1}^{p}\epsilon_{j}(4n_{j}+2)$, we can further reduce the expression to  \setlength\arraycolsep{2pt}\begin{eqnarray}\mathfrak{J}&=&(2\pi)^{p+\frac{1}{2}}\sum_{\varrho\in\mathbb{Z}}\int \prod_{j\in B\cup\{0\}}\phi_{j}(\lambda_{j})\,\mathrm{d}\lambda_{j}\cdot\hat{\zeta}\bigg(\sum_{j\in B}\epsilon_{j}\lambda_{j}-\lambda_{0}+\varrho\bigg)^{-}\\&\times&\sum_{\mathbb{S}_{\varrho}}\kappa_{n_{1},\cdots,n_{p}}^{n_{0}}(4n_{0}+2)^{\frac{\sigma}{2}}\prod_{j\in B\cup\{0\}}a_{\lambda_{j}}^{j}(n_{j})^{-}\prod_{j\in A}\frac{\theta_{j}(n_{j})g_{n_{j}}^{-}(\omega)}{\sqrt{4n_{j}+2}}\nonumber\\
&\times &\exp\bigg(-\mathrm{i}t_{0}\sum_{j\in A}(4n_{j}+2)\epsilon_{j}\bigg),\nonumber
\end{eqnarray} where we write \begin{equation}\mathbb{S}_{\varrho}=\bigg\{(n_{0},\cdots, n_{p}):\frac{1}{10}\leq\frac{4n_{j}+2}{N_{j}^{2}}\leq 10\,\,(j\geq 1),\,\,(4n_{0}+2)-\sum_{j=1}^{p}\epsilon_{j}(4n_{j}+2)=\varrho\bigg\}.\end{equation} Notice that $\hat{\zeta}=T\hat{\zeta}_{0}(T\cdot)$, and that $\hat{\zeta}_{0}$ is a Schwartz function, we have\begin{equation}\sum_{\varrho\in\mathbb{Z}}|\hat{\zeta}(\lambda+\varrho)|\lesssim \sum_{\varrho\in\mathbb{Z}} T\langle T(\lambda+\varrho)\rangle^{-2}\lesssim 1\end{equation} for all $\lambda\in[0,1]$, and by periodicity, for all $\lambda\in\mathbb{R}$. Therefore\begin{equation}\sum_{\varrho\in\mathbb{Z}}\int\prod_{j\in B\cup\{0\}}|\phi_{j}(\lambda_{j})|\,\mathrm{d}\lambda_{j}\cdot\hat{\zeta}\bigg|\bigg(\sum_{j\in B}\epsilon_{j}\lambda_{j}-\lambda_{0}+\varrho\bigg)\bigg|\lesssim N_{1}^{3\gamma(b-\frac{1}{2})}.\end{equation} Since we choose $b$ close enough to $\frac{1}{2}$ depending on $\sigma$ and $p$, and $\gamma$ does not have any dependence on $b$ whatsoever (we may simply take $\gamma=200$), (\ref{separate3}) will follow if \setlength\arraycolsep{2pt}\begin{eqnarray}\label{cs}\bigg|\sum_{\mathbb{S}_{\varrho}}\kappa_{n_{1},\cdots,n_{p}}^{n_{0}}(4n_{0}+2)^{\frac{\sigma}{2}}&\times&\prod_{j\in B\cup\{0\}}a_{\lambda_{j}}^{j}(n_{j})^{-}\prod_{j\in A}\frac{\theta_{j}(n_{j})g_{n_{j}}^{-}(\omega)}{\sqrt{4n_{j}+2}}\\
&\times&\exp\bigg(-\mathrm{i}t_{0}\sum_{j\in A}(4n_{j}+2)\epsilon_{j}\bigg)\bigg|\lesssim T^{\frac{1}{2}-b}N_{1}^{-\delta}\nonumber,\end{eqnarray} for \emph{all} possible choices of $t_{0}\in\mathbb{R}$, $\varrho\in\mathbb{Z}$, $\lambda_{j}\in\mathbb{R}(j\in B\cup\{0\})$, $\{a_{\lambda_{j}}^{j}(k)\}$ satisfying (\ref{sqsum}) and (\ref{squarectrl}), with $\delta>0$ depending on $\sigma$ and $p$, but \emph{not} on $b$.

Next, by Cauchy-Schwartz in the sum with respect to $n_{0}$, we can further estimate the LHS of (\ref{cs}) by \setlength\arraycolsep{2pt}\begin{eqnarray}\label{cs2}\bigg(\sum_{n_{0}}(4n_{0}+2)^{\sigma}&\times&\bigg|\sum_{\mathbb{S}_{\varrho,n_{0}}}\kappa_{n_{1},\cdots,n_{p}}^{n_{0}}\prod_{j\in B}b_{j}(n_{j})^{-}\prod_{j\in A}\frac{\theta_{j}(n_{j})g_{n_{j}}^{-}(\omega)}{\sqrt{4n_{j}+2}}\\
&\times&\exp\bigg(-\mathrm{i}t_{0}\sum_{j\in A}(4n_{j}+2)\epsilon_{j}\bigg)\bigg|^{2}\bigg)^{\frac{1}{2}}\nonumber,\end{eqnarray} where $\mathbb{S}_{\varrho,n_{0}}=\{(n_{1},\cdots,n_{p}):(n_{0},\cdots,n_{p})\in\mathbb{S}_{\varrho}\}$, and $b_{j}(k)=a_{\lambda_{j}}^{j}(k)$.

Concerning the inner sum of (\ref{cs2}) we have (recall that $\frac{1}{10}\leq\frac{4n_{j}+2}{N_{j}^{2}}\leq 10$ for each $1\leq j\leq p$)  \setlength\arraycolsep{2pt}
\begin{eqnarray}
\label{sb5zero}&&\bigg|\sum_{\mathbb{S}_{\varrho,n_{0}}}\kappa_{n_{1},\cdots,n_{p}}^{n_{0}}\prod_{j\in B}b_{j}(n_{j})^{-}\prod_{j\in A}\frac{\theta_{j}(n_{j})g_{n_{j}}^{-}(\omega)}{\sqrt{4n_{j}+2}}\cdot e^{-\mathrm{i}t_{0}\sum_{j\in A}(4n_{j}+2)\epsilon_{j}}\bigg|\\&\lesssim &\label{sb5one}\sum_{(n_{j})_{j\in B}}\bigg|\sum_{\Theta}\tau_{n_{1},\cdots,n_{p}}^{n_{0}}\prod_{j\in A}g_{n_{j}}^{-}(\omega)\bigg|\prod_{j\in B}|b_{j}(n_{j})|\\
&\lesssim &\sup_{(n_{j})_{j\in B}}\bigg|\sum_{\Theta}\tau_{n_{1},\cdots,n_{p}}^{n_{0}}\prod_{j\in A}g_{n_{j}}^{-}(\omega)\bigg| \prod_{j\in B}\sum_{4n_{j}+2\sim N_{j}^{2}}|b_{j}(n_{j})|\nonumber\\ 
\label{sb5two}&\lesssim &\sup_{(n_{j})_{j\in B}}\bigg|\sum_{\Theta}\tau_{n_{1},\cdots,n_{p}}^{n_{0}}\prod_{j\in A}g_{n_{j}}^{-}(\omega)\bigg|\prod_{j\in B}(N_{j}^{2}N_{j}^{-2\sigma})^{\frac{1}{2}}\\
&\lesssim & \sup_{(n_{j})_{j\in B}}\bigg|\sum_{\Theta}\tau_{n_{1},\cdots,n_{p}}^{n_{0}}\prod_{j\in A}g_{n_{j}}^{-}(\omega)\bigg|\prod_{j\in B}N_{j}^{1-\sigma}\nonumber
,\end{eqnarray} where in (\ref{sb5one}) we write $\Theta=\{(n_{j})_{j\in A}:(n_{1},\cdots,n_{p})\in\mathbb{S}_{\varrho,n_{0}}\}$ for fixed $(n_{j})_{j\in B}$, and $\tau_{n_{1},\cdots,n_{p}}^{n_{0}}=\kappa_{n_{1},\cdots,n_{p}}^{n_{0}}\prod_{j\in A}\theta_{j}(n_{j})(4n_{j}+2)^{-\frac{1}{2}}$. One should notice that for all $(n_{j})_{j\in A}\in\Theta$, by definition the expression $e^{-\mathrm{i}t_{0}\sum_{j\in A}(4n_{j}+2)\epsilon_{j}}$ is a fixed constant with absolute value $1$ which can be extracted. In (\ref{sb5two}) we have used Cauchy-Schwartz and (\ref{squarectrl}).

Let us fix $\varrho$ and $n_{0}$, and $(n_{j})_{j\in B}$. We also assume $|4n_{0}+2-\varrho|\lesssim N_{1}^{2}$ (otherwise $\mathbb{S}_{\varrho,n_{0}}$ would be empty). Since the set $\Theta$ has the form of $\Xi$ in (\ref{Xi}) and $N_{1}>10^{3}\sum_{j\in A-\{1\}}N_{j}$, we can use Proposition \ref{gaussianest} to get\begin{equation}\bigg|\sum_{\Theta}\tau_{n_{1},\cdots,n_{p}}^{n_{0}}\prod_{j\in A}g_{n_{j}}^{-}(\omega)\bigg|\leq K\prod_{j\in A-\{1\}}N_{j}\cdot \sup_{\Theta}|\tau_{n_{1},\cdots,n_{p}}^{n_{0}}|,\end{equation} with exceptional probability $\leq c_{1}e^{-c_{2}K^{c_{3}}}$. We choose $K=T^{\frac{1}{2}-b}N_{1}^{\frac{1-\sigma}{200}}(4n_{0}+2)^{\frac{1-\sigma}{400}}$, then the corresponding exceptional probability is $\leq c_{1}e^{-c_{2}T^{-c_{3}}N_{1}^{c_{4}}(4n_{0}+2)^{c_{5}}}$. If we add them up with respect to \emph{all} possible choices of $\varrho$ and $(n_{j})_{j\in B\cup\{0\}}$, we still get an expression $\leq c_{1}e^{-c_{2}T^{-c_{3}}N_{1}^{c_{4}}}$ (there are $\lesssim N_{1}^{2}$ choices for each $n_{j}(j\in B)$, and for fixed $n_{0}$, there are $\lesssim N_{1}^{2}$ choices of $\varrho$). Therefore with exceptional probability $\leq c_{1}e^{-c_{2}T^{-c_{3}}N_{1}^{c_{4}}}$ we have
\setlength\arraycolsep{2pt}
\begin{eqnarray}
(\ref{cs2})&\lesssim& T^{\frac{1}{2}-b}N_{1}^{\frac{1-\sigma}{200}}\prod_{j\in B}N_{j}^{1-\sigma}\prod_{j\in A-\{1\}}N_{j}\\
&\times&\bigg(\sum_{n_{0}}(4n_{0}+2)^{\sigma+\frac{1-\sigma}{200}}\sup_{\mathbb{S}_{\mu,n_{0}}}|\tau_{n_{1},\cdots,n_{p}}^{n_{0}}|^{2}\bigg)^{\frac{1}{2}}\nonumber\\
\label{sb6}&\lesssim &  T^{\frac{1}{2}-b}N_{1}^{\frac{1-\sigma}{200}-1}\prod_{j\in B}N_{j}^{1-\sigma} \\
&\times&\bigg(\sum_{n_{0}}(4n_{0}+2)^{\sigma+\frac{1-\sigma}{200}}\sup_{\mathbb{S}_{\mu,n_{0}}}|\kappa_{n_{1},\cdots,n_{p}}^{n_{0}}|^{2}\bigg)^{\frac{1}{2}}.\nonumber
\end{eqnarray}

To complete the proof of Proposition \ref{longest}, let us estimate $\kappa_{n_{1},\cdots,n_{p}}^{n_{0}}$. Let $\nu_{(0)}\geq\cdots\geq\nu_{(p)}$ be the nonincreasing permutation of $\nu_{j}=4n_{j}+2(0\leq j\leq p)$. If $\nu_{0}\geq N_{1}^{2(1+\frac{1-\sigma}{200})}$, from Lemma \ref{multil} we have $|\kappa_{n_{1},\cdots,n_{p}}^{n_{0}}|\lesssim \nu_{0}^{-100}$. If $\nu_{0}<N_{1}^{2(1+\frac{1-\sigma}{200})}$, since $N_{1}\gtrsim (\max_{j\geq 2}N_{j})^{\frac{\sigma+1}{3\sigma-1}}$, we see that if $\nu_{0}\leq\max_{j\geq 2}\nu_{j}$, then $\nu_{1}\gtrsim\max_{j\neq 1}\nu_{j}^{\frac{\sigma+1}{3\sigma-1}}$ and $|\kappa_{n_{1},\cdots,n_{p}}^{n_{0}}|\lesssim N_{1}^{-100}$; if $\nu_{0}>\max_{j\geq 2}\nu_{j}$, then $\nu_{(2)}\geq \max_{j\geq 2}\nu_{j}$ and from Lemma \ref{multil} we have\begin{equation}|\kappa_{n_{1},\cdots,n_{p}}^{n_{0}}|\lesssim \nu_{(0)}^{-\frac{1}{2}}\nu_{(2)}^{-\frac{1}{4}}\log\nu_{(0)}\lesssim N_{1}^{-1}(\max_{j\geq 2}N_{j})^{-\frac{1}{2}}\log N_{1}.\end{equation} Therefore we have \setlength\arraycolsep{2pt}
\begin{eqnarray}
(\ref{sb6})&\lesssim& T^{\frac{1}{2}-b}N_{1}^{\frac{1-\sigma}{200}-1}\prod_{j\in B}N_{j}^{1-\sigma}\nonumber\\
&\times&\bigg(\sum_{\nu_{0}< N_{1}^{2(1+\frac{1-\sigma}{200})}}(N_{1})^{2(\sigma+\frac{1-\sigma}{200})(1+\frac{1-\sigma}{200})}N_{1}^{-2}(\max_{j\geq 2}N_{j})^{-1}\log^{2}N_{1}\nonumber\\
&+&\sum_{\nu_{0}\geq N_{1}^{2(1+\frac{1-\sigma}{200})}}(4n_{0}+2)^{-198}\bigg)^{\frac{1}{2}}\nonumber\\
&\lesssim &  T^{\frac{1}{2}-b}N_{1}^{-\theta_{0}}\log N_{1}\cdot(\max_{j\geq 2}N_{j})^{-\frac{1}{2}}\prod_{j\in B}N_{j}^{1-\sigma}\nonumber \\&\lesssim &T^{\frac{1}{2}-b}N_{1}^{-\frac{\theta_{0}}{2}}(\max_{j\geq 2}N_{j})^{-\frac{1}{2}}\prod_{j\in B}N_{j}^{1-\sigma},\nonumber
\end{eqnarray}where \begin{equation}\theta_{0}=1-\frac{1-\sigma}{100}-\sigma-\frac{(1+\sigma)(1-\sigma)}{200}-\frac{(1-\sigma)^{2}}{40000}>\frac{1-\sigma}{2}>0.\end{equation} Finally, since $1\in A$, we have\begin{equation}(\max_{j\geq 2}N_{j})^{-\frac{1}{2}}\prod_{j\in B}N_{j}^{1-\sigma}\lesssim(\max_{j\geq 2}N_{j})^{-\frac{1}{2}+(p-1)(1-\sigma)}\lesssim 1,\end{equation} provided $\sigma>1-\frac{1}{2(p-1)}$.

Having considered all the different cases, we have now finished the proof of Proposition \ref{longest}.
\end{proof}

From now on we will fix $\sigma$ and $b$ as stated in Proposition $\ref{longest}$. As an easy corollary of this proposition, we now have

\begin{corollary}\label{picard}There exist some $\theta>0$ and $T_{0}>0$, such that the following holds:
for all $0<T<T_{0}$, there exists a set $\Omega_{T}\subset\Omega$ such that $\mathbb{P}(\Omega_{T})\leq c_{1}e^{-c_{2}T^{-c_{3}}}$ and for all $\omega\not\in\Omega_{T}$, the mapping\begin{equation}\label{contra}u\mapsto e^{-\mathrm{i}t\mathbf{H}}f(\omega)\mp\mathrm{i}\int_{0}^{t}e^{-\mathrm{i}(t-s)\mathbf{H}}(|u(s)|^{p-1}u(s))\,\mathrm{d}s\end{equation} is a contraction mapping from the affine ball\begin{equation}e^{-\mathrm{i}t\mathbf{H}}f(\omega)+\big\{v:\|v\|_{\mathcal{X}^{\sigma,b,T}}\leq T^{\theta}\big\}\end{equation} to itself.
\end{corollary}
\begin{proof} Suppose $u=e^{-\mathrm{i}t\mathbf{H}}f(\omega)+v$, where $\|v\|_{\mathcal{X}^{\sigma,b,T}}\leq T^{\theta}\leq 1$. From Proposition \ref{linear} we have
\setlength\arraycolsep{2pt}
\begin{eqnarray}
\mathfrak{M} &:=&\bigg\|\mp\mathrm{i}\int_{0}^{t}e^{-\mathrm{i}(t-s)\mathbf{H}}(|u(s)|^{p-1}u(s))\,\mathrm{d}s\bigg\|_{\mathcal{X}^{\sigma,b,T}}\nonumber\\
&\lesssim&\big\||u|^{p-1}u\big\|_{\mathcal{X}^{\sigma,b-1,T}}\nonumber\\
&=&\big\|(e^{-\mathrm{i}t\mathbf{H}}f(\omega)+v)^{\frac{p+1}{2}}\cdot(\overline{e^{-\mathrm{i}t\mathbf{H}}f(\omega)}+\bar{v})^{\frac{p-1}{2}}\big\|_{\mathcal{X}^{\sigma,b-1,T}}.\nonumber
\end{eqnarray}
If we expand the product, then each term has the form as in Proposition \ref{longest} (namely, $u_{1}^{-}\cdots u_{p}^{-}$ with each $u_{j}$ either equal to $e^{-\mathrm{i}t\mathbf{H}}f(\omega)$ or has $\mathcal{X}^{\sigma,b,T}$ norm $\lesssim 1$), thus we have $\mathfrak{M}\lesssim T^{\theta_{0}}$ for some $\theta_{0}$ depending only on $\sigma,b$ and $p$; thus if we choose $\theta<\theta_{0}$ and $T_{0}$ small enough, then the mapping does map the affine ball to itself.

In addition,  if $u_{i}=e^{-\mathrm{i}t\mathbf{H}}f+v_{i}$ with $\|v_{i}\|_{\mathcal{X}^{\sigma,b,T}}\leq T^{\theta}$ for $i\in\{1,2\}$, then \setlength\arraycolsep{2pt}
\begin{eqnarray}
\mathfrak{D}&:=&\bigg\|\mp\mathrm{i}\int_{0}^{t}e^{-\mathrm{i}(t-s)\mathbf{H}}(|u_{1}(s)|^{p-1}u_{1}(s)-|u_{2}(x)|^{p-1}u_{2}(s))\,\mathrm{d}s\bigg\|_{\mathcal{X}^{\sigma,b,T}}\nonumber\\
&\lesssim&\big\||u_{1}|^{p-1}u_{1}-|u_{2}|^{p-1}u_{2}\big\|_{\mathcal{X}^{\sigma,b-1,T}}\nonumber\\
&\lesssim&\sum_{\mathbb{F}}\|(u_{1}-u_{2})^{-}\prod_{k=1}^{p-1}u_{j_{k}}^{-}\|_{\mathcal{X}^{\sigma,b-1,T}},\nonumber
\end{eqnarray} where $\mathbb{F}$ is some finite set, and each $j_{k}\in\{1,2\}$. Since $u_{1}-u_{2}=v_{1}-v_{2}\in \mathcal{X}^{\sigma,b,T}$, and each $u_{j}$ is the sum of two terms, one being $e^{-\mathrm{i}t\mathbf{H}}f(\omega)$, the other having $\mathcal{X}^{\sigma,b,T}$ norm $\lesssim 1$, we can use Proposition \ref{longest} to estimate $\mathfrak{D}\lesssim T^{\theta_{0}}\|v_{1}-v_{2}\|_{\mathcal{X}^{\sigma,b,T}}$ for all $\omega\not\in\Omega_{T}$. Thus the result follows if we choose $T$ small enough.
\end{proof}

\section{Local well-posedness results}\label{lwp}

In proving local in time results, we will not care about the $\pm$ sign in (\ref{nls22}). First we define the truncated Cauchy problem\begin{equation}\label{trc}
\left\{
\begin{array}{ll}
\mathrm{i}\partial_{t}u+(\Delta-|x|^{2}) u=(\pm|u|^{p-1}u)_{2^{k}}^{\circ}\\
u(0)=f_{2^{k}}^{\circ}(\omega)
\end{array}
\right.
\end{equation} for each $k\geq 1$. When $k=\infty$, we understand that $v_{2^{\infty}}^{\circ}=v$, so this is just the original equation (\ref{nls22}). If $k<\infty$, we solve (\ref{trc}) in the finite dimensional space $V_{2^{k}}$. We will consider two cases depending on whether $p\geq 3$ odd or $1<p<3$.

\subsection{The algebraic case}\label{defo}
Here we assume $p\geq 3$ is an odd integer, so we could use the estimates is Section \ref{nonlinear}.
\begin{proposition}\label{defolwp}
Suppose $T>0$ is sufficiently small. There exists a set $\Omega_{T}$ (possibly different from the one in Proposition \ref{longest}), such that $\mathbb{P}(\Omega_{T})\leq c_{1}e^{-c_{2}T^{-c_{3}}}$, and when $\omega\not\in\Omega_{T}$, for each $1\leq k\leq \infty$, (\ref{trc}) has a unique solution \begin{equation}\label{space2}u\in e^{-\mathrm{i}t\mathbf{H}}f_{2^{k}}^{\circ}(\omega)+\mathcal{X}^{\sigma,b,T}\end{equation} on $[-T,T]$, satisfying\begin{equation}\label{ctrl2}\|u-e^{-\mathrm{i}t\mathbf{H}}f_{2^{k}}^{\circ}(\omega)\|_{\mathcal{X}^{\sigma,b,T}}\leq T^{\theta}.\end{equation} 
\end{proposition}
\begin{proof}
When $k=\infty$, the existence and uniqueness directly follows from Corollary \ref{picard} via Picard iteration. Now we assume $1\leq k<\infty$, then the equation (\ref{trc}) is just an ODE, so the solution is unique, and exists until its norm approaches infinity. Thus we only need to obtain the control on each of these solutions, uniformly in $k$. To this end we need the following modification of Proposition \ref{longest}.
\begin{lemma}\label{longest2}For each $T$ sufficiently small, we can find a set (still denoted by $\Omega_{T}$), so that $\mathbb{P}(\Omega_{T})\leq c_{1}e^{-c_{2}T^{-c_{3}}}$, and in Proposition \ref{longest}, if one replaces some $u_{j}$ by any $(u_{j})_{2^{k_{j}}}^{\circ}$ or $(u_{j})_{2^{k_{j}}}^{\perp}$, the result still holds true. Moreover, if there is at least one $(u_{j})_{k_{j}}^{\perp}$, then the left side of (\ref{mainpro}) tends to zero (uniformly in all choices of $u_{j}$) as this $k_{j}\to\infty$.
\end{lemma}
\begin{proof} We use the notations as in Proposition \ref{longest}. Note the projections $u_{2^{k}}^{\circ}$ and $u_{2^{k}}^{\perp}$ are uniformly bounded on $X^{\sigma,b,T}$, we may assume the modification is only for $j\in A$. Since $f_{2^{k}}^{\circ}(\omega)=f(\omega)-f_{2^{k}}^{\perp}(\omega)$ and the result is true when all terms are still $u_{j}$, we may assume each term is either $u_{j}$ or $(u_{j})_{2^{k_{j}}}^{\perp}$, with at least one $(u_{j})_{2^{k_{j}}}^{\perp}$.

For each $(k_{j})$, we follow exactly the proof of Proposition \ref{longest}. Suppose $L=\max_{j}2^{k_{j}}$, then in the dyadic decomposition we only need to consider the terms $\max_{j\in A}N_{j}\gtrsim L$ (for example, if $(N_{1},\cdots,N_{p})\in\mathcal{A}$ with the largest being $N_{1}$, then $\max_{j\geq 2}N_{j}\gtrsim L$; otherwise we have $\max_{j}N_{j}\gtrsim L$). On the other hand, all the probabilistic Lebesgue/Sobolev estimates of $f(\omega)$ we used in Proposition \ref{longest} comes from Corollary \ref{prob}, thus they also hold for $f_{2^{k}}^{\perp}(\omega)=f(\omega)-f_{2^{k}}^{\circ}(\omega)$ uniformly in $k$. As for the multilinear estimates of Gaussians (Proposition \ref{gaussianest}), they indeed hold for fixed $k_{j}$, because fixing $k_{j}$ (and replacing $f(\omega)$ by $f(\omega)_{2^{k_{j}}}^{\circ}$) corresponds to adding constraints $n_{j}\leq 2^{k_{j}}$ in the set $\Xi$ in (\ref{Xi}), which does not affect the estimates in (\ref{comb}) (which is based on upper bounds of the cardinals of some sets). Therefore for fixed $k_{j}$, the estimates about each individual terms (including the ``grouped'' terms in $\mathcal{A}$) in the proof of Proposition \ref{longest} still hold, with constants \emph{independent} of $k_{j}$. Therefore, we have\[\|\mathrm{Modified}(u_{1}^{-}\cdots u_{p}^{-})\|_{X^{\sigma,b,T}}\lesssim\sum_{\max_{j}N_{j}\gtrsim L}T^{\theta}(\max_{j}N_{j})^{-\theta}\lesssim T^{\theta}L^{-\frac{\theta}{2}},\] with exceptional probability not exceeding \[\sum_{\max_{j}N_{j}\gtrsim L}c_{1}e^{-c_{2}T^{-c_{3}}(\max_{j}N_{j})^{c_{4}}}\leq c_{5}e^{-c_{6}T^{-c_{7}}L^{c_{8}}},\]which implies \[\|\mathrm{Modified}(u_{1}^{-}\cdots u_{p}^{-})\|_{X^{\sigma,b,T}}\lesssim T^{\theta}(\max_{j}2^{k_{j}})^{-\frac{\theta}{2}},\] for \emph{all} possible choices of $k_{j}$, with exceptional probability not exceeding\[\sum_{(k_{j})}c_{5}e^{-c_{6}T^{-c_{7}}(\max_{j}2^{k_{j}})^{c_{8}}}\lesssim c_{9}e^{-c_{10}T^{-c_{11}}}.\] If we choose this final exceptional set as our $\Omega_{T}$, we easily see that all requirements are satisfied.
\end{proof}
\begin{remark}\label{rem}
In Proposition \ref{longest} and Lemma \ref{longest2}, the estimates still hold when the $\mathcal{X}^{\sigma,b,T}$ norm is replaced by $\mathcal{X}^{\sigma,b,I}$ (but with $T^{\theta}$ on the right side of (\ref{mainpro}) unchanged), for \emph{any} interval $I\subset [-T,T]$, and $\omega$ outside a \emph{single} $\Omega_{T}$. One can check the proof that all estimates do not become worse with $[-T,T]$ replaced by $I$. In particular we can get a contraction mapping as in Corollary \ref{picard} for interval $[-T,0]$ or $[-T_{1},T_{1}]$ for $T_{1}\leq T$.
\end{remark}
Using Lemma \ref{longest2}, we can now proceed with the proof of Proposition \ref{defolwp}. Suppose for some $k$, $u=e^{-\mathrm{i}t\mathbf{H}}f_{2^{k}}^{\circ}(\omega)+v$ is a maximal solution to (\ref{trc}) (strictly speaking the $T$ below should be another $T'$ denoting the lifespan of $u$, but we will ignore this, in view of Remark \ref{rem}). Then outside the $\Omega_{T}$ constructed in Lemma \ref{longest2} we have \setlength\arraycolsep{2pt}
\begin{eqnarray}
\|v\|_{\mathcal{X}^{\sigma,b,T}} &=&\bigg\|\mp\mathrm{i}\int_{0}^{t}e^{-\mathrm{i}(t-s)\mathbf{H}}\big(|u(s)|^{p-1}u(s)\big)_{2^{k}}^{\circ}\,\mathrm{d}s\bigg\|_{\mathcal{X}^{\sigma,b,T}}\nonumber\\
&\lesssim&\big\||u|^{p-1}u\big\|_{\mathcal{X}^{\sigma,b-1,T}}\nonumber\\
&=&\big\|(e^{-\mathrm{i}t\mathbf{H}}f_{2^{k}}^{\circ}(\omega)+v)^{\frac{p+1}{2}}\cdot(\overline{e^{-\mathrm{i}t\mathbf{H}}f_{2^{k}}^{\circ}(\omega)}+\bar{v})^{\frac{p-1}{2}}\big\|_{\mathcal{X}^{\sigma,b-1,T}}.\nonumber
\end{eqnarray} Each term in the expansion of the final product has the form as in Lemma \ref{longest2} (namely $\prod_{j}(u_{j}^{-})_{2^{k_{j}}}^{\circ}$ with $1\leq k_{j}\leq\infty$, and each $u_{j}$ either equal to $e^{-\mathrm{i}t\mathbf{H}}f(\omega)$ or has $\mathcal{X}^{\sigma,b,T}$ norm $\lesssim \|v\|_{\mathcal{X}^{\sigma,b,T}}$). Therefore for some $\theta>0$ we get \[\|v\|_{\mathcal{X}^{\sigma,b,T}}\lesssim T^{\theta}(1+\|v\|_{\mathcal{X}^{\sigma,b,T}})^{p},\] since $v\in \mathcal{X}^{\sigma,b,T}$ and $v(0)=0$, we know $\|v\|_{\mathcal{X}^{\sigma,b,t}}\to 0$ as $t\to 0$. The local norm is continuous in $t$, thus we can use a bootstrap argument to get $\|v\|_{\mathcal{X}^{\sigma,b,T}}\leq T^{\frac{\theta}{2}}$. Note this also works for the original equation, showing that (\ref{ctrl2}) holds for the solution of (\ref{nls22}) with any $k$. The uniqueness of (\ref{nls22}) now follows from Corollary \ref{picard}.
\end{proof}

\subsection{The subcubic case}\label{foc}
Here we assume $1<p<3$, and we do not need any multilinear estimate to solve the local problem.
\begin{proposition}\label{foclwp}
Suppose $T>0$ is sufficiently small. There exists a set $\Omega_{T}$ (possibly different from the one in Proposition \ref{longest}), such that $\mathbb{P}(\Omega_{T})\leq c_{1}e^{-c_{2}T^{-c_{3}}}$, and when $\omega\not\in\Omega_{T}$, for each $1\leq k\leq \infty$, (\ref{trc}) has a unique solution \begin{equation}\label{space3}u\in e^{-\mathrm{i}t\mathbf{H}}f_{2^{k}}^{\circ}(\omega)+\mathcal{X}^{0,b,T}\end{equation} on $[-T,T]$, satisfying\begin{equation}\label{ctrl3}\|u-e^{-\mathrm{i}t\mathbf{H}}f_{2^{k}}^{\circ}(\omega)\|_{\mathcal{X}^{0,b,T}}\leq T^{\theta}.\end{equation}
\end{proposition}
\begin{proof}
The proof here is almost the same as Proposition \ref{defolwp}. In fact, once we can obtain\begin{equation}\label{fracchain1}\big\||e^{-\mathrm{i}t\mathbf{H}}f_{2^{k}}^{\circ}(\omega)+v|^{p-1}(e^{-\mathrm{i}t\mathbf{H}}f_{2^{k}}^{\circ}(\omega)+v)\big\|_{\mathcal{X}^{0,b-1,T}}\lesssim T^{\theta}(1+\|v\|_{\mathcal{X}^{0,b,T}}^{p})\end{equation} and 
\begin{equation}\label{fracchain2}\big\||u|^{p-1}u-|u'|^{p-1}u'\big\|_{\mathcal{X}^{0,b-1,T}}\lesssim T^{\theta}\|v-v'\|_{\mathcal{X}^{0,b,T}}\cdot(1+\|v\|_{\mathcal{X}^{0,b,T}}+\|v'\|_{\mathcal{X}^{0,b,T}})^{p-1}\end{equation} for all $1\leq k\leq\infty$ and $\omega\not\in\Omega_{T}$, where $u=e^{-\mathrm{i}t\mathbf{H}}f_{2^{k}}^{\circ}(\omega)+v$ and $u'=e^{-\mathrm{i}t\mathbf{H}}f_{2^{k}}^{\circ}(\omega)+v'$, we can use Proposition \ref{linear} and argue as in the proof of Corollary \ref{picard} to show that for $\omega\not\in\Omega_{T}$, \[u\mapsto e^{-\mathrm{i}t\mathbf{H}}f(\omega)\mp\mathrm{i}\int_{0}^{t}e^{-\mathrm{i}(t-s)\mathbf{H}}(|u(s)|^{p-1}u(s))\,\mathrm{d}s\] is a contraction mapping from \[e^{-\mathrm{i}t\mathbf{H}}f(\omega)+\{v:\|v\|_{\mathcal{X}^{0,b,T}}\leq T^{\theta}\}\] to itself, for some $\theta>0$. Also we will have the same estimates on solutions to (\ref{trc}) as in Proposition \ref{defolwp} which is enough for the proof.

To prove (\ref{fracchain2}), we simply compute (again we omit the time domain $[-T,T]$ here)\setlength\arraycolsep{2pt}
\begin{eqnarray}&&\big\||u|^{p-1}u-|u'|^{p-1}u'\big\|_{\mathcal{X}^{0,b-1,T}}\nonumber\\\label{zhutou01}&\lesssim&T^{2b-1}\big\|(u-u')(|u|+|u'|)^{p-1}\big\|_{L_{t,x}^{q}}\\
\label{zhutou02}&\lesssim&T^{2b-1}\|v-v'\|_{L_{t}^{r_{1}}L_{x}^{q_{1}}}\cdot(\|u\|_{L_{t}^{r_{2}}L_{x}^{q_{2}}}+\|u'\|_{L_{t}^{r_{2}}L_{x}^{q_{2}}})^{p-1}\\
\label{zhutou03}&\lesssim & T^{2b-1}\|v-v'\|_{\mathcal{X}^{0,b,T}}(\|v\|_{\mathcal{X}^{0,b,T}}+\|v'\|_{\mathcal{X}^{0,b,T}}+\|e^{-\mathrm{i}t\mathbf{H}}f_{2^{k}}^{\circ}(\omega)\|_{L_{t}^{r_{2}}L_{x}^{q_{2}}})^{p-1}\\
\label{zhutou04}&\lesssim & T^{b-\frac{1}{2}}\|v-v'\|_{\mathcal{X}^{0,b,T}}\cdot(\|v\|_{\mathcal{X}^{0,b,T}}+\|v'\|_{\mathcal{X}^{0,b,T}}+1)^{p-1}
\end{eqnarray}
outside $\Omega_{T}$, where $\mathbb{P}(\Omega_{T})\leq c_{1}e^{-c_{2}T^{-c_{3}}}$. Here in (\ref{zhutou01}) we have used Proposition \ref{strichartz} and Lemma \ref{linear00}, and require $\frac{1}{2}<b<\frac{2}{3}$, $2>q>\frac{2}{3-3b}$. In (\ref{zhutou02}) we have used H\"{o}lder and require $1<r_{1},q_{1},r_{2},q_{2}<\infty$ and $\frac{1}{r_{1}}+\frac{p-1}{r_{2}}=\frac{1}{q_{1}}+\frac{p-1}{q_{2}}=\frac{1}{q}$. In (\ref{zhutou03}) we have used (H\"{o}lder in time and) Proposition \ref{strichartz} and require \[\frac{1}{q_{1}}+\frac{1}{r_{1}}\geq\frac{1}{2},\,\,\frac{1}{q_{2}}+\frac{1}{r_{2}}\geq\frac{1}{2};\,\,\,\,\,\,\,\,\,2<q_{1},q_{2}<\infty.\]In (\ref{zhutou04}) we have used Corollary \ref{prob} to bound\[\|e^{-\mathrm{i}t\mathbf{H}}f_{2^{k}}^{\circ}(\omega)\|_{L_{t}^{r_{2}}L_{x}^{q_{2}}}\lesssim T^{-\frac{2b-1}{100p}},\] with exceptional probability $\leq c_{1}e^{-c_{2}T^{-c_{3}}}$. Therefore, we may choose $q<\frac{4}{p}$ sufficiently close to $\frac{4}{p}$. In particular $q>\frac{4}{3}$, so we may choose $\frac{1}{2}<b<1-\frac{2}{3q}$. Then a simple computation shows that we can choose $q_{1},q_{2},r_{1},r_{2}$ appropriately so that the scaling equations hold, and $1<r_{1,2}<\infty$, $2<q_{1,2}<\infty$, and $\frac{1}{q_{i}}+\frac{1}{r_{i}}>\frac{1}{2}$ for $i=1,2$. This completes the proof of (\ref{fracchain2}).

The estimate (\ref{fracchain1}) follows from the same choice of exponents and similar arguments. The only difference is that we will have a term $\|e^{-\mathrm{i}t\mathbf{H}}f_{2^{k}}^{\circ}(\omega)\|_{L_{t}^{r_{1}}L_{x}^{q_{1}}}$, which is fine as long as $2<q_{1}<\infty$.
\end{proof}
\subsection{Approximating by ODEs}

Here we will prove that almost surely, uniform global bounds on the solutions to the truncated equations (\ref{trc}) for infinitely many $k<\infty$ implies the global existence and uniqueness for the original equation (\ref{nls22}).
\begin{proposition}\label{appppp} Let $[-T,T]$ be a time interval, where we assume $T$ is large. Suppose for $\omega$ belonging to some set $E$, there exists a subsequence $\{k_{j}\}_{j\geq 0}\uparrow\infty$ (possibly depending on $\omega$) such that each of the equations (\ref{trc}) with $k=k_{j}$ has a unique solution $u_{j}$ on $[-T,T]$ and that\begin{equation}\label{uniformity}\sup_{j}\|u_{j}-e^{-\mathrm{i}t\mathbf{H}}f_{2^{k_{j}}}^{\circ}(\omega)\|_{\mathcal{X}^{\sigma,b,T}}<\infty.\end{equation}Then a.s. $\omega\in E$, the equation (\ref{nls22}) possesses a unique solution $u$ on $[-T,T]$ so that $u\in e^{-\mathrm{i}t\mathbf{H}}f(\omega)+\mathcal{X}^{\sigma,b,T}$. Moreover for this subsequence we have\begin{equation}\label{application}\lim_{j\to\infty}\|u_{j}-e^{-\mathrm{i}t\mathbf{H}}f_{2^{k_{j}}}^{\circ}(\omega)-(u-e^{-\mathrm{i}t\mathbf{H}}f(\omega))\|_{\mathcal{X}^{\sigma,b,T}}=0.\end{equation}
\end{proposition}
\begin{proof} For $\omega\in E$, with small exceptional probability (tending to $0$ as $A\to\infty$), we may choose a sequence $u_{j}$ solving (\ref{trc}) with $k=k_{j}\uparrow\infty$, and \begin{equation}\label{limit}\|u_{j}-e^{-\mathrm{i}t\mathbf{H}}f_{2^{k_{j}}}^{\circ}(\omega)\|_{\mathcal{X}^{\sigma,b,T}}\leq A,\end{equation} for all $j$. Then we choose an integer $M$ large enough depending on $T$ and $A$. We are going to prove for each $1\leq m\leq M$ that, (\ref{nls22}) has a unique solution $u\in e^{-\mathrm{i}t\mathbf{H}}f(\omega)+\mathcal{X}^{\sigma,b,\frac{mT}{M}}$ on the interval $[-\frac{mT}{M},\frac{mT}{M}]$, and \begin{equation}\label{approximate}\lim_{j\to\infty}\|u_{j}-e^{-\mathrm{i}t\mathbf{H}}f_{2^{k_{j}}}^{\circ}(\omega)-(u-e^{-\mathrm{i}t\mathbf{H}}f(\omega))\|_{\mathcal{X}^{\sigma,b,\frac{mT}{M}}}\to 0,\end{equation} for $\omega$ outside the fixed set $\Omega_{\frac{T}{M}}$ which is constructed in the proof of Lemma \ref{longest2}. Since $\mathbb{P}(\Omega_{\frac{T}{M}})\to 0$ as $M\to\infty$, this clearly contains the conclusion we need.

Now we proceed by induction on $m$. First assume $p\geq 3$ is odd. Suppose the conclusion holds for $m-1$ (including $m=1$), we will prove it for $m$. Write $\delta=M^{-1}T$ and $t_{0}=(m-1)\delta$, we know the solution $u$ exists and is unique on $[-t_{0},t_{0}]$, and we want to extend it to $[-(t_{0}+\delta),t_{0}+\delta]$. Without loss if generality we consider the half-line $t>0$.

From (\ref{limit}) and (\ref{approximate}) we have \begin{equation}\label{zhutou10}\lim_{j\to\infty}\|u_{j}(t_{0})-u(t_{0})+e^{-\mathrm{i}t_{0}\mathbf{H}}f_{2^{k_{j}}}^{\perp}(\omega)\|_{\mathcal{H}^{\sigma}}=0,\end{equation} and $\|u(t_{0})-e^{-\mathrm{i}t_{0}\mathbf{H}}f(\omega)\|_{\mathcal{H}^{\sigma}}\leq A$. We would like to solve the equation (\ref{nls2}) with initial data $u(t_{0})$ on $[-\delta,\delta]$, and argue as in Corollary \ref{picard}. Here the linear term is not $e^{-\mathrm{i}t\mathbf{H}}f(\omega)$, but \[e^{-\mathrm{i}t\mathbf{H}}u(t_{0})=e^{-\mathrm{i}(t+t_{0})\mathbf{H}}f(\omega)+v,\] where $v$ is the linear evolution of some function with $\mathcal{H}^{\sigma}$ norm $\lesssim A$, thus $\|v\|_{\mathcal{X}^{\sigma,b,\delta}}\lesssim A$ (this is easily proved by introducing a cutoff and using $\delta\leq 1$). Since $\omega\not\in\Omega_{\frac{T}{M}}$, we can use the full strength of proposition \ref{longest} and Lemma \ref{longest2}. In particular we could proceed as in the proof of Corollary \ref{picard} and obtain 
\begin{equation}
\mathfrak{M}:=\bigg\|\mp\mathrm{i}\int_{0}^{t}e^{-\mathrm{i}(t-s)\mathbf{H}}(|w_{1}(s)|^{p-1}w_{1}(s))\,\mathrm{d}s\bigg\|_{\mathcal{X}^{\sigma,b,\delta}}\nonumber\lesssim \delta^{\theta_{0}}A^{p}\leq \delta^{\theta},\nonumber
\end{equation}
and
\setlength\arraycolsep{2pt}
\begin{eqnarray}
\mathfrak{D}&=&\bigg\|\mp\mathrm{i}\int_{0}^{t}e^{-\mathrm{i}(t-s)\mathbf{H}}(|w_{1}(s)|^{p-1}w_{1}(s)-|w_{2}(x)|^{p-1}w_{2}(s))\,\mathrm{d}s\bigg\|_{\mathcal{X}^{\sigma,b,\delta}}\nonumber\\
&\lesssim&\delta^{\theta_{0}}A^{p-1}\|h_{1}-h_{2}\|_{\mathcal{X}^{\sigma,b,\delta}}<\frac{1}{2}\|h_{1}-h_{2}\|_{\mathcal{X}^{\sigma,b,\delta}},\nonumber
\end{eqnarray} for all $w_{i}=e^{-\mathrm{i}t\mathbf{H}}u(t_{0})+h_{i}$ with $\|h_{i}\|_{\mathcal{X}^{\sigma,b,\delta}}\leq \delta^{\theta}$, provided $M$ is large enough ($\delta$ is small enough) depending on $T$ and $A$. Then we could use Picard iteration and the same bootstrap argument to prove that the original solution $u$ can be uniquely extended to $[t_{0},t_{0}+\delta]$ (and by symmetry, to the other side).

It remains to prove (\ref{approximate}) for $m$. First we know \[\lim_{j\to\infty}\big\|e^{-\mathrm{i}(t-t_{0})\mathbf{H}}u_{j}(t_{0})-e^{-\mathrm{i}(t-t_{0})\mathbf{H}}u(t_{0})+e^{-\mathrm{i}t\mathbf{H}}f_{2^{k_{j}}}^{\perp}(\omega))\big\|_{\mathcal{X}^{\sigma,b,[t_{0}-\delta,t_{0}+\delta]}}=0,\]
which is a consequence of (\ref{zhutou10}). In view of the induction hypothesis, we only need to prove \footnote[3]{Here we have used the following fact: given two intervals $[x,y]$ and $[z,w]$ with $x<z<y<w$, for some constant $C$ we have $\|u\|_{\mathcal{X}^{\sigma,b,[x,w]}}\leq C(\|u\|_{\mathcal{X}^{\sigma,b,[x,y]}}+\|u\|_{\mathcal{X}^{\sigma,b,[z,w]}})$. This is easily proved by using a partition of unity.}\[\lim_{j\to\infty}\big\|u_{j}-e^{-\mathrm{i}(t-t_{0})\mathbf{H}}u_{j}(t_{0})-(u-e^{-\mathrm{i}(t-t_{0})\mathbf{H}}u(t_{0}))\big\|_{\mathcal{X}^{\sigma,b,[t_{0}-\delta,t_{0}+\delta]}}=0,\] which, after a translation of time, is equivalent to\begin{equation}\label{equaivalentness}\lim_{j\to\infty}\big\|w_{j}-e^{-\mathrm{i}t\mathbf{H}}w_{j}(0)-(w-e^{-\mathrm{i}t\mathbf{H}}w(0))\big\|_{\mathcal{X}^{\sigma,b,\delta}}=0,\end{equation} where $w_{j}$ is a solution of the truncated equation with $k=k_{j}$, and $w_{j}(0)=u_{j}(t_{0})$; and $w$ is a solution of the original equation with $w(0)=u(t_{0})$. Write $w_{j}-w=h=h_{li}+h_{no}$, where \begin{equation}h_{li}=e^{-\mathrm{i}t\mathbf{H}}(u_{j}(t_{0})-u(t_{0}))=-e^{-\mathrm{i}(t+t_{0})\mathbf{H}}f_{2^{k_{j}}}^{\perp}(\omega)+e^{-\mathrm{i}t\mathbf{H}}\lambda_{j},\end{equation} with $\|\lambda_{j}\|_{\mathcal{H}^{\sigma}}\to 0$, and \setlength\arraycolsep{2pt}
\begin{eqnarray}h_{no}&=&\mp\mathrm{i}\int_{0}^{t}e^{-\mathrm{i}(t-s)\mathbf{H}}\big((|w_{j}|^{p-1}w_{j}-|w|^{p-1}w)_{2^{k_{j}}}^{\circ}-(|w|^{p-1}w)_{2^{k_{j}}}^{\perp}(s)\big)\,\mathrm{d}s\nonumber\\
&=&\label{equivalentness2}\mp\mathrm{i}\int_{0}^{t}e^{-\mathrm{i}(t-s)\mathbf{H}}(|w_{j}|^{p-1}w_{j}-|w|^{p-1}w)_{2^{k_{j}}}^{\circ}(s)\,\mathrm{d}s-(w-e^{-\mathrm{i}t\mathbf{H}}w(0))_{2^{k_{j}}}^{\perp}.\end{eqnarray} Now we need to prove $\|h_{no}\|_{\mathcal{X}^{\sigma,b,\delta}}\to 0$.  Since $w-e^{-\mathrm{i}t\mathbf{H}}w(0)\in \mathcal{X}^{\sigma,b,\delta}$, the second term in (\ref{equivalentness2}) tends to zero in $\mathcal{X}^{\sigma,b,\delta}$ as $j\to\infty$. For the first term, we estimate the norm without the final projection. The expression in the paranthesis can be written as a linear combination of terms like $z_{1}^{-}\cdots z_{p}^{-}$, where $z_{1}$ is either $h_{no}$, or $e^{-\mathrm{i}(t+t_{0})\mathbf{H}}f_{2^{k_{j}}}^{\perp}(\omega)$, or $e^{-\mathrm{i}t\mathbf{H}}\lambda_{j}$ which has $\mathcal{X}^{\sigma,b,\delta}$ norm $\to 0$. For $i\geq 2$, each $z_{i}$ is one of the following:

(1) $e^{-\mathrm{i}(t+t_{0})\mathbf{H}}f_{2^{k_{j}}}^{\circ}(\omega)$. This is within the applicability of Lemma \ref{longest2} since $\omega\not\in\Omega_{\frac{T}{M}}$.

(2) $w_{j}-e^{-\mathrm{i}(t+t_{0})\mathbf{H}}f_{2^{k_{j}}}^{\circ}(\omega)$. This has $\mathcal{X}^{\sigma,b,\delta}$ norm $\lesssim A$ since $w_{j}(t)=u_{j}(t+t_{0})$, due to (\ref{limit}).

(3) one of the components of $w_{j}-w$. These include $h_{no}$ and $e^{-\mathrm{i}(t+t_{0})\mathbf{H}}f_{2^{k_{j}}}^{\perp}(\omega)$, as well as another term with $\mathcal{X}^{\sigma,b,\delta}$ norm $\lesssim A$. Since $\omega\not\in\Omega_{\frac{T}{M}}$, these terms are controllable using Lemma \ref{longest2}.

If $z_{1}=e^{-\mathrm{i}(t+t_{0})\mathbf{H}}f_{2^{k_{j}}}^{\perp}(\omega)$, then from Proposition \ref{linear} and Lemma \ref{longest2}, the corresponding term tends to $0$ as $j\to\infty$ (since $h_{no}$ is bounded in $\mathcal{X}^{\sigma,b,\delta}$ independent of $j$; see below). If $z_{1}$ is the term with $\mathcal{X}^{\sigma,b,\delta}$ tending to $0$, the same conclusion holds. If $z_{1}=h_{no}$, then the norm of the corresponding term is bounded by $\delta^{\theta}\|h_{no}\|_{\mathcal{X}^{\sigma,b,\delta}}(\|h_{no}\|_{\mathcal{X}^{\sigma,b,\delta}}+A)^{p-1}$. Therefore we have\[\|h_{no}\|_{\mathcal{X}^{\sigma,b,\delta}}\lesssim \delta^{\theta}\|h_{no}\|_{\mathcal{X}^{\sigma,b,\delta}}(\|h_{no}\|_{\mathcal{X}^{\sigma,b,\delta}}+A)^{p-1}+o(1),\] as $j\to\infty$. By (\ref{limit}) and the Picard argument above, we know $\|h_{no}\|_{\mathcal{X}^{\sigma,b,\delta}}\lesssim A$ independent of $j$. Therefore, if we choose $\delta$ small enough ($M$ large enough), we must have $\|h_{no}\|_{\mathcal{X}^{\sigma,b,\delta}}=o(1)$.

The proof when $1<p<3$ is basically the same, using linear estimates (Corollary \ref{prob}) instead of Proposition \ref{longest}. We will also need a variant of Lemma \ref{longest2}, but the proof of this is not hard and is essentially contained in Proposition \ref{lin2} and Corollary \ref{prob}.
\end{proof}

\section{Global well-posedness}\label{gwp}

In what follows, we fix a sufficiently large $T$ , and a positive integer $M$ so that $M\gtrsim T^{2}$.

First let us consider the truncated equation (\ref{trc}), which is  an ODE on the finite dimensional space $V_{2^{k}}$. If we identify $V_{2^{k}}$ with $\mathbb{R}^{2^{k+1}+2}$ by the coordinates \begin{equation}g=\sum_{j=0}^{2^{k}}(a_{j}+\mathrm{i}b_{j})e_{j},\end{equation} then it is easy to check that (\ref{trc}) becomes\begin{equation}\partial_{t}a_{j}=\frac{\partial E}{\partial b_{j}},\,\,\,\,\partial_{t}b_{j}=-\frac{\partial E}{\partial a_{j}},\end{equation} with Hamiltonian\begin{equation}E_{0}(a_{j},b_{j})=\sum_{j=0}^{2^{k}}(2j+1)(a_{j}^{2}+b_{j}^{2})\pm\frac{1}{p+1}\bigg\|\sum_{j=0}^{2^{k}}(a_{j}+\mathrm{i}b_{j})e_{j}\bigg\|_{L^{p+1}}^{p+1}.\end{equation}If we denote the solution flow of this equation by $\Phi_{2^{k},t}$, then the following is true by the theory of Hamiltonian ODEs: the map $(t,x)\mapsto\Phi_{2^{k},t}(x)$ is defined on an open subset of $\mathbb{R}\times V_{2^{k}}$. For each $t\in\mathbb{R}$, $\Phi_{2^{k},t}$ is a homeomorphism between two open sets $\mathcal{O}_{t}$ and $\mathcal{R}_{t}$ of $V_{2^{k}}$. If $p\geq 3$ is odd, it is a diffeomorphism and preserves the quantities  \begin{equation}\label{preservedquan}\|g\|_{L^{2}}^{2}=\sum_{j=0}^{2^{k}}(a_{j}^{2}+b_{j}^{2})\,\,\,\,\,\,\,\mathrm{and}\,\,\,\,\,\,\,E=2E_{0}\end{equation}and the Lebesgue measure. If $1<p<3$, it (and its inverse) can be approximated, uniformly on each compact subset of $\mathcal{O}_{t}$ and $\mathcal{R}_{t}$, by a sequence of pairs of diffeomorphisms, which preserve the quantities (\ref{preservedquan}) and the Lebesgue measure. Therefore $\Phi_{2^{k},t}$ itself also preserves (\ref{preservedquan}) and the Lebesgue measure.

From above we know that $\Phi_{2^{k},t}$ (viewed as a map between $\mathcal{O}_{t}$ and $\mathcal{R}_{t}$) preserves the measure \begin{equation}\nu_{2^{k}}^{\circ}=\pi^{-1-2^{k}}\zeta\cdot e^{-E}\prod_{j=0}^{2^{k}}\mathrm{d}a_{j}\mathrm{d}b_{j}\end{equation}on $V_{2^{k}}$, where $\zeta=1$ in the defocusing case, and $\zeta=\chi(\|g\|_{L^{2}}^{2}-\alpha_{2^{k}})$ in the focusing case as in (\ref{cutoff}). By the definition of $\mu$ and $\nu_{2^{k}}$ (see Section \ref{gib}) we have\begin{equation}\nu_{2^{k}}=(\rho_{2^{k}}\cdot\mu_{2^{k}}^{\circ})\otimes\mu_{2^{k}}^{\perp}=\nu_{2^{k}}^{\circ}\otimes\mu_{2^{k}}^{\perp},\end{equation} in both cases, where we understand that $\mu_{2^{k}}^{\circ}$ and $\mu_{2^{k}}^{\perp}$ are measures on $V_{2^{k}}$ and $V_{2^{k}}^{\perp}$ respectively, and identify $V$ with\footnote[4]{Here $V$ is some space on which $\mu$ is supported. The exact choice of $V$ is unimportant; for example, we may choose $V=\mathcal{S}'(\mathbb{R}^{2})$, or $V=\cap_{\delta>0}\mathcal{H}^{-\delta}(\mathbb{R}^{2})$.} $V_{2^{k}}\times V_{2^{k}}^{\perp}$. From this we immediately see, for each Borel set $J$ of $V_{2^{k}}$, that\begin{equation}\label{recur0}\nu_{2^{k}}\big(\{g:g_{2^{k}}^{\circ}\in J\cap\mathcal{R}_{t}\}\big)=\nu_{2^{k}}\big(\{g:g_{2^{k}}^{\circ}\in(\Phi_{2^{k},t})^{-1}(J)\}\big).\end{equation} 

Now we fix the choice \[J=J_{M}=\big\{g_{2^{k}}^{\circ}:g\in f\big(\Omega_{\frac{T}{M}}^{c}\big)\big\}^{c},\] where $\Omega_{\frac{T}{M}}$ is constructed in the proof of Lemma \ref{longest2}. Consider the maximal $m_{0}\leq M+1$ so that the equation (\ref{trc}) has a solution $u$ on $[-\frac{m_{0}T}{M},\frac{m_{0}T}{M}]$ which satisfies \begin{equation}\label{partialintegral}\big\|u-e^{-\mathrm{i}(t-\frac{mT}{M})\mathbf{H}}u\big(\frac{mT}{M}\big)\big\|_{\mathcal{X}^{\sigma,b,[\frac{(m-1)T}{M},\frac{(m+1)T}{M}]}}\leq 1,\end{equation} for $|m|\leq m_{0}-1$. If $m_{0}=M+1$, from Proposition \ref{linear} we know that $u$ is defined on $[-T,T]$ and\begin{equation}\label{partialsum}\|u-e^{-\mathrm{i}t\mathbf{H}}f_{2^{k}}^{\circ}(\omega)\|_{\mathcal{X}^{\sigma,b,T}}\lesssim M^{3}.\end{equation} If $m_{0}\leq M$, then for some choice of $\pm$ sign, we have that $f_{2^{k}}^{\circ}(\omega)\in\mathcal{O}_{\pm\frac{ m_{0}T}{M}}$ and $\Phi_{2^{k},\pm\frac{ m_{0}T}{M}}(f_{2^{k}}^{\circ}(\omega))\in J_{M}$. In fact, if this fails, we can use Proposition \ref{defolwp} to extend the solution to $[-\frac{(m_{0}+1)T}{M},\frac{(m_{0}+1)T}{M}]$ with (\ref{partialintegral}) remaining true, thus contradicting the definition of $m_{0}$. Now we use (\ref{recur0}) and sum over $m_{0}\leq M$ to get
\begin{equation}(\nu_{2^{k}}\circ f)\big(\big\{\omega:(\ref{partialsum})\,\,\mathrm{fails}\big\}\big)\lesssim M\cdot\nu_{2^{k}}\big(\big\{g:g_{2^{k}}^{\circ}\in J_{M}\big\}\big).\end{equation}

In the defocusing case we have $\nu_{2^{k}}\leq \mu$. Using Fubini's theorem we get\begin{equation}\label{recur3}\mu\big(\big\{g:g_{2^{k}}^{\circ}\in J_{M}^{c}\big\}\big)\geq\mu\big(f\big(\Omega_{\frac{T}{M}}^{c}\big)\big)\geq 1-c_{1}e^{-c_{2}T^{-c_{3}}M^{c_{3}}},\end{equation} hence \begin{equation}\label{recur4}(\nu_{2^{k}}\circ f)\big(\big\{\omega:(\ref{partialsum})\,\,\mathrm{fails}\big\}\big)\lesssim c_{1}Me^{-c_{2}T^{-c_{3}}M^{c_{3}}}.\end{equation}

In the focusing case we have \begin{equation}\frac{\mathrm{d}\nu_{2^{k}}}{\mathrm{d}\mu}(g)=\rho_{2^{k}}(g)=\chi(\|g_{2^{k}}^{\circ}\|_{L^{2}}^{2}-\alpha_{2^{k}})\exp\bigg(\frac{2}{p+1}\|g_{2^{k}}^{\circ}\|_{L^{p+1}}^{p+1}\bigg).\end{equation} This function, by Proposition \ref{intg}, has bounded $L^{2}(\mu)$ norm, so by Cauchy-Schwartz we get\begin{equation}\nu_{2^{k}}\big(\big\{g:g_{2^{k}}^{\circ}\in J_{M}\big\}\big)\lesssim\bigg(\mu\big(\big\{g:g_{2^{k}}^{\circ}\in J_{M}\big\}\big)\bigg)^{\frac{1}{2}}\leq c_{1}e^{-c_{2}T^{-c_{3}}M^{c_{3}}},\end{equation} which again implies (\ref{recur4}). We summarize our results in the following proposition.

\begin{proposition}\label{temp}
For fixed $T$ and $k$, there exists a subset\footnote[5]{This should not be confused with the $\Omega_{T}$ notation defined above, since our $\Omega_{k}$ is for $k\geq 1$ here!} $\Omega_{k}\subset\Omega$ so that $(\nu_{2^{k}}\circ f)(\Omega_{k}^{c})=0$, and for $\omega\in\Omega_{k}$, (\ref{trc}) has a unique solution $u_{k}$ on $[-T,T]$, and that\begin{equation}\label{integrand}\sup_{k}\int_{\Omega_{k}}\exp\big(\|u_{k}-e^{-\mathrm{i}t\mathbf{H}}f_{2^{k}}^{\circ}(\omega)\|_{\mathcal{X}^{\sigma,b,T}}^{\theta}\big)\,\mathrm{d}(\nu_{2^{k}}\circ f)(\omega)<\infty,\end{equation} for some $\theta>0$.
\end{proposition}
\begin{proof} We choose\[\Omega_{k}=\cap_{M\gtrsim T^{2}}Z_{M}:=\bigcap_{M\gtrsim T^{2}}\big\{\omega:(\ref{partialsum})\,\,\mathrm{fails}\,\,\mathrm{for}\,\,M\big\}.\] From the above discussion we easily see $\nu_{2^{k}}(f(\Omega_{k}^{c}))\leq\lim_{M\to\infty}\nu_{2^{k}}(f(Z_{M}))=0$. Also for $\omega\not\in Z_{M}$ the solution $u_{k}$ to (\ref{trc}) exists and is unique, and satisfies\[\|u_{k}-e^{-\mathrm{i}t\mathbf{H}}f_{2^{k}}^{\circ}(\omega)\|_{\mathcal{X}^{\sigma,b,T}}\lesssim M^{3}.\]In other words we have \[(\nu_{2^{k}}\circ f)\big(\omega\in\Omega_{k}:\|u_{k}-e^{-\mathrm{i}t\mathbf{H}}f_{2^{k}}^{\circ}(\omega)\|_{\mathcal{X}^{\sigma,b,T}}>A\big)\leq\nu_{2^{k}}(f(Z_{M}))\leq c_{1}e^{-c_{2}A^{c_{3}}},\] for all $A>T^{100}$, where $M\sim A^{\frac{1}{3}}$ is an integer. Since $\nu_{2^{k}}\circ f$ is uniformly integrable, the part with small $A$ is also under control. The claim then follows.
\end{proof}

With Propositions \ref{appppp} and \ref{temp}, we are now ready to prove the global well-posedness part of Theorem \ref{main}. Denote the integrand in (\ref{integrand}) by $\eta_{k}(\omega)$, understanding $\eta_{k}(\omega)=0$ when $\omega\not\in\Omega_{k}$. Since $\nu_{2^{k}}\to\nu$ in the strong sense and $(\nu_{2^{k}}\circ f)(\Omega_{k}^{c})=0$, we have $(\nu\circ f)(\Omega_{k}^{c})\to 0$, and we fix  a subsequence $\{k_{l}\}$ so that $\sum_{l}(\nu\circ f)(\Omega_{k_{l}}^{c})<\infty$ and hence $(\nu\circ f)(\limsup_{l\to\infty}\Omega_{k_{l}}^{c})=0$. From Proposition \ref{temp}, we get
\begin{equation}\sup_{l}\int_{\Omega}\rho_{2^{k_{l}}}(f(\omega))\eta_{k_{l}}(\omega)\,\mathrm{d}\mathbb{P}(\omega)<\infty.\end{equation} From the proof of Proposition \ref{intg}, we see $\rho_{2^{k_{l}}}\circ f\to\rho\circ f$ almost surely, so by Fatou's lemma we get\begin{equation}\label{holds}\liminf_{l\to\infty}\eta_{k_{l}}(\omega)<\infty,\end{equation} a.s. in $\mathbb{P}$, on the set where $\rho(f(\omega))\neq 0$. By the definition of $\eta_{k}$, if (\ref{holds}) holds, then either $\omega\in\Omega_{k_{l}}$ for infinitely many $l$, or there exists a subsequence $\{k_{l_{j}}\}_{j\geq 0}\uparrow\infty$, such that (\ref{trc}) has a unique solution $u_{k_{l_{j}}}$ for $k=k_{j}$ on $[-T,T]$, and\[\sup_{j}\|u_{k_{l_{j}}}-e^{-\mathrm{i}t\mathbf{H}}f_{2^{k_{l_{j}}}}^{\circ}(\omega)\|_{\mathcal{X}^{\sigma,b,T}}<\infty.\] In the former case we get a null set (actually a set with null $\nu\circ f$ measure, but $\nu\circ f$ is mutually absolutely continuous with $\mathbb{P}$ on the set where $\rho(f(\omega))\neq 0$), while in the latter case we can use Proposition \ref{appppp} to deduce that, except for another null set, (\ref{nls22}) also has a unique solution $u$ on $[-T,T]$ so that $u\in e^{-\mathrm{i}t\mathbf{H}}f(\omega)+\mathcal{X}^{\sigma,b,T}$.

Therefore, for each $T>0$, except for a null set, the equation (\ref{nls22}) has a unique solution $u\in e^{-\mathrm{i}t\mathbf{H}}f(\omega)+\mathcal{X}^{\sigma,b,T}$, for $\omega$ in the support of $\rho\circ f$. In the defocusing case, this support itself has full probability in $\Omega$; in the focusing case, it follows from Proposition \ref{intg} that we can choose a countable number of cutoff $\chi$, so that the (countable) union of the support of the corresponding $\rho\circ f$ has full probability. In any case we have found a subset of $\Omega$ having full probability, such that when $\omega$ does belong to this set, (\ref{nls22}) has a unique solution $u\in e^{-\mathrm{i}t\mathbf{H}}f(\omega)+\mathcal{X}^{\sigma,b,T}$. We then take another countable union to get that, almost surely in $\mathbb{P}$, equation (\ref{nls22}) has a unique solution $u$ on $\mathbb{R}\times\mathbb{R}^{2}$ such that\setlength\arraycolsep{2pt}
\begin{eqnarray}u\in e^{-\mathrm{i}t\mathbf{H}}f(\omega)+\mathcal{X}^{\sigma,b,T}&\subset &e^{-\mathrm{i}t\mathbf{H}}f(\omega)+\mathcal{C}([-T,T],\mathcal{H}^{\sigma}(\mathbb{R}^{2}))\nonumber\\
&\subset&\mathcal{C}([-T,T],\cap_{\delta>0}\mathcal{H}^{-\delta}(\mathbb{R}^{2}))\nonumber,\end{eqnarray} for all $T>0$. This completes the proof.
\begin{remark}
In fact, from the above argument we can extract a polynomial bound on the solution; namely we can prove that for each large $A$, with exceptional probability $\leq c_{1}e^{-c_{2}A^{c_{3}}}$ we have\[\|u-e^{-\mathrm{i}t\mathbf{H}}f(\omega)\|_{\mathcal{X}^{\sigma,b,T}}\leq A\langle T\rangle ^{C},\] for all $T>0$, with some constant $C$. We omit the details.
\end{remark}

\section{Transforming into NLS without harmonic potential}\label{lens0}
As we have mentioned before, the idea of introducing the lens transform and reducing (\ref{nls1}) to (\ref{nls2}) is inspired by the arguments in \cite{BTT10}. First we define the lens transform (\cite{Ta09}, Section 2; \cite{BTT10}, Section 10)\begin{equation}\label{lens}\mathcal{L}u(t,x)=\frac{1}{\cos(2t)}u\big(\frac{\tan(2t)}{2},\frac{x}{\cos(2t)}\big)e^{-\frac{\mathrm{i}|x|^{2}\tan(2t)}{2}},\end{equation} where $u$ is defined on $\mathbb{R}\times\mathbb{R}^{2}$, and $\mathcal{L}u$ is defined on $\big(-\frac{\pi}{4},\frac{\pi}{4}\big)\times\mathbb{R}^{2}$. By a simple computation we deduce\begin{equation}\label{conjugate0}(\mathrm{i}\partial_{t}-\mathbf{H})(\mathcal{L}u)(t,x)=(\cos(2t))^{-2}\mathcal{L}((\mathrm{i}\partial_{t}+\Delta)u)(t,x).\end{equation} For the inverse transform \begin{equation}\label{inverse}\mathcal{L}^{-1}u(t,x)=(1+4t^{2})^{-\frac{1}{2}}u\big(\frac{\tan^{-1}(2t)}{2},(1+4t^{2})^{-\frac{1}{2}}x\big)e^{\mathrm{i}\frac{|x|^{2}t}{1+4t^{2}}},\end{equation} we have \begin{equation}\label{conjugate}(\mathrm{i}\partial_{t}+\Delta)(\mathcal{L}^{-1}u)(t,x)=\frac{1}{1+4t^{2}}\mathcal{L}^{-1}((\mathrm{i}\partial_{t}-\mathbf{H})u)(t,x).\end{equation} 

Next we prove that the transform $\mathcal{L}^{-1}$ maps the space $\mathcal{X}^{\sigma,b,\delta}$ to $X^{\sigma,b,T}$, where $0\leq\sigma,b\leq 1$, $0<\delta<\frac{\pi}{4}$, and $T=\frac{1}{2}\tan(2\delta)$. First by using a cutoff, we are reduced to proving that $u\mapsto \mathcal{L}^{-1}(\chi\cdot u)$ is bounded from $\mathcal{X}^{\sigma,b}$ to $X^{\sigma,b}$, where $\chi=\chi(t)$ is any smooth function having compact support in $|t|<\frac{\pi}{4}$. First we fix $\sigma$. By interpolation, we can assume $b\in\{0,1\}$. If we can prove the result in the case $b=0$, then using the identity\[\|u\|_{\mathcal{X}^{\sigma,1}}^{2}=\|u\|_{\mathcal{X}^{\sigma,0}}^{2}+\|(\mathrm{i}\partial_{t}-\mathbf{H})\|_{\mathcal{X}^{\sigma,0}}^{2}\] (which remains true with $\mathcal{X}$ replaced by $X$ and $-\mathbf{H}$ replaced by $\Delta$) and (\ref{conjugate}), we see\begin{equation}\label{add1}\|\mathcal{L}^{-1}(\chi\cdot u)\|_{X^{\sigma,1}}\lesssim\|u\|_{\mathcal{X}^{\sigma,0}}+\|(\mathrm{i}\partial_{t}-\mathbf{H})(\chi u)\|_{\mathcal{X}^{\sigma,0}},\end{equation} because $v=(\mathrm{i}\partial_{t}-\mathbf{H})(\chi u)$ has compact support in $|t|<\frac{\pi}{4}$, and hence equals $\chi_{1}v$ for some other $\chi_{1}$. Since the last term in (\ref{add1}) is clearly controlled by $\|u\|_{\mathcal{X}^{\sigma,1}}$, we can conclude the proof for $b=1$. Therefore we may only consider $b=0$. Here it is easily seen that we only need to prove that multiplication by $e^{\mathrm{i}\lambda|x|^{2}}$ is uniformly bounded from $\mathcal{H}^{\sigma}$ to $H^{\sigma}$ for $0\leq\sigma\leq 1$ and $|\lambda|\leq 1$. By another interpolation we may further reduce to $\sigma\in\{0,1\}$. The $\sigma=0$ case is obvious; the $\sigma=1$ case follows from the observation\[\nabla(e^{\mathrm{i}\lambda |x|^{2}}\cdot f)=e^{\mathrm{i}\lambda |x|^{2}}(\nabla+2\mathrm{i}\lambda x)\cdot f.\] Thus we have the desired bound for all $0\leq\sigma,b\leq 1$.

Using (\ref{conjugate0}) or (\ref{conjugate}) we can compute that $u$ is a solution for the Cauchy problem (\ref{main000}) on $\mathbb{R}$, if and only if $v=\mathcal{L}u$ is a solution for the Cauchy problem \begin{equation}\label{nls44}
\left\{
\begin{array}{ll}
\mathrm{i}\partial_{t}v+(\Delta-|x|^{2})v=(\cos(2t))^{p-3}\cdot|v|^{p-1}v\\
v(0)=f(\omega)
\end{array}
\right.
\end{equation} on $|t|<\frac{\pi}{4}$. Moreover, if $v-e^{-\mathrm{i}t\mathbf{H}}f(\omega)\in\mathcal{X}^{\sigma,b,\delta}$ with $\delta<\frac{\pi}{4}$, then from the above discussion we see that \[u-\mathcal{L}^{-1}(e^{-\mathrm{i}t\mathbf{H}}f(\omega))\in X^{\sigma,b,T}\] with $T=\frac{1}{2}\tan(2\delta)\to\infty$ as $\delta\to\frac{\pi}{4}$. From (\ref{conjugate}) we see $\mathcal{L}^{-1}(e^{-\mathrm{i}t\mathbf{H}}f(\omega))$ has initial value $f(\omega)$ and annihilates $\mathrm{i}\partial_{t}+\Delta$, thus it must be $e^{\mathrm{i}t\Delta}f(\omega)$. This proves (\ref{moresmooth}). Also from (\ref{inverse}), the constants in the $\mathcal{H}_{x}^{\sigma}\to H_{x}^{\sigma}$ boundedness remains under control even near the boundary points $\pm\frac{\pi}{4}$. Thus (\ref{scatter}) will follow if (again, note that $\mathcal{L}$ conjugates the propagtors $e^{\mathrm{i}t\Delta}$ and $e^{-\mathrm{i}t\mathbf{H}}$)\begin{equation}\label{goodresult}\lim_{t\to\pm\frac{\pi}{4}}\big(v(t)-e^{-\mathrm{i}t\mathbf{H}}f(\omega)\big)\,\,\,\,\,\,\,\mathrm{exists}\,\,\mathrm{in}\,\,\mathcal{H}^{\sigma}.\end{equation}

What we will prove is that a.s. in $\mathbb{P}$, (\ref{nls44}) has a unique (strong) solution $v$ for $|t|\leq\frac{\pi}{4}$ so that $v-e^{-\mathrm{i}t\mathbf{H}}f(\omega)\in\mathcal{X}^{\sigma,b,\frac{\pi}{4}}$. As is demonstrated above, this implies (\ref{goodresult}) and hence Theorem \ref{main2}.

The proof is basically the same as (\ref{nls22}). Noticing $m(t)=(\cos(2t))^{p-3}$ has all its derivatives bounded on $\mathbb{R}$, we see that multiplication by $m(t)$ is bounded from any $\mathcal{X}^{\sigma,b}$ (and hence any $\mathcal{X}^{\sigma,b,T}$) to itself. Therefore, the proof from Proposition \ref{longest} to Lemma \ref{longest2} goes without any difficulty, as if this additional factor were not present. In the proof of Proposition \ref{appppp}, when we extend the solution to a larger interval, we must solve another Cauchy problem, which is no longer (\ref{nls44}), since this equation is not autonomous. This, however, is not a problem; since we just replace $m(t)$ by some $m(t-t_{0})$ which obeys the same derivative estimates as $m(t)$, we can use the \emph{same} exceptional set as in Proposition \ref{longest}, Lemma \ref{longest2} and Proposition \ref{appppp}, and the other discussions remain unchanged.

The only difficulty we face is the lack of a (formally) invariant measure. This is compensated, however, by a monotonicity property, which was first observed in \cite{BTT10}.
\begin{lemma}\label{mono} Consider the truncated Cauchy problem \begin{equation}\label{nls55}
\left\{
\begin{array}{ll}
\mathrm{i}\partial_{t}v+(\Delta-|x|^{2})v=(\cos(2t))^{p-3}\cdot(|v|^{p-1}v)_{2^{k}}^{\circ}\\
v(0)=f_{2^{k}}^{\circ}(\omega)
\end{array}
\right.
\end{equation} then for its solution $v$, the quantity\[\mathcal{E}(t,v(t))=\langle \mathbf{H}v,v\rangle+\frac{2(\cos(2t))^{p-3}}{p+1}\|v\|_{L^{p+1}}^{p+1}\] is monotonically nonincreasing in $|t|$, for $|t|\leq\frac{\pi}{4}$.
\end{lemma}
\begin{proof}
We directly compute\[\frac{\mathrm{d}\mathcal{E}}{\mathrm{d}t}=-\frac{2(p-3)(\cos(2t))^{p-4}\sin(2t)}{p+1}\|v(t)\|_{L^{p+1}}^{p+1},\] which is nonpositive for $0\leq t\leq\frac{\pi}{4}$, and nonnegative for $-\frac{\pi}{4}\leq t\leq 0$.
\end{proof}
We argue as in Section \ref{gwp}, but we fix $T=\frac{\pi}{4}$ here. If we could prove \begin{equation}\label{recur1}\mu\big(\{g:g_{2^{k}}^{\circ}\in J\}\big)\geq\nu_{2^{k}}\big(\{g:\Phi_{2^{k},t}(g_{2^{k}}^{\circ})\in J\}\big)\end{equation} for $-\frac{\pi}{4}\leq t\leq \frac{\pi}{4}$, where, of course, $\Phi_{2^{k},t}$ is now the solution flow of (\ref{nls44}), then combining this inequality with (\ref{recur3}) we can get (\ref{recur4}). Starting from this point, we can follow the argument in Section \ref{gwp} word by word to get a.s. global well-posedness of (\ref{nls44}) on $[-\frac{\pi}{4},\frac{\pi}{4}]$.

The proof of (\ref{recur1}) is also simple.  By Lemma \ref{mono} \setlength\arraycolsep{2pt}
\begin{eqnarray}
\nu_{2^{k}}\big(\{g:\Phi_{2^{k},t}(g_{2^{k}}^{\circ})\in J\}\big)&=&\pi^{-1-2^{k}}\int_{J_{1}}e^{-E(g)}\prod_{j=0}^{2^{k}}\mathrm{d}a_{j}\mathrm{d}b_{j}\nonumber\\&\leq &\pi^{-1-2^{k}}\int_{J_{1}}e^{-\mathcal{E}(t,g(t))}\prod_{j=0}^{2^{k}}\mathrm{d}a_{j}\mathrm{d}b_{j}\nonumber\\ &\leq &\int_{J}\,\mathrm{d}\mu_{2^{k}}^{\circ}=\mu\big(\{g:g_{2^{k}}^{\circ}\in J\}\big),\nonumber
\end{eqnarray} where $J_{1}=\{h\in V_{2^{k}}:\Phi_{2^{k},t}(h)\in J\}$. Here we have used the invariance of the Lebesgue measure under $\Phi_{2^{k},t}$, which can be directly verified (see \cite{BTT10}, Lemma 8.3). Therefore we have completed the proof of Theorem \ref{main2}.
\section{Invariance of Gibbs measure}\label{inv}
Now we return to the final assertion of Theorem \ref{main}, and prove the invariance of the Gibbs measure $\nu$ under the solution flow of (\ref{nls22}). More precisely, we have\begin{proposition}
Denote the solution flow of (\ref{nls22}) by $\Phi_{t}$. There exists a subset $\Sigma\subset\mathcal{S}'(\mathbb{R}^{2})$ such that it has full $\mu$ measure, and $\Phi_{t}$ becomes a one-parameter group from $\Sigma$ to $\Sigma$ preserving the measure $\nu$ (in the focusing case, for each choice of cutoff function $\chi$).
\end{proposition}
\begin{proof} We only consider the defocusing case. In the focusing case we need to take another countable intersection corresponding to the cutoff $\chi$ chosen, but otherwise the proof is completely analogous. Clearly the set $\Omega_{T}$ in Proposition \ref{longest} and Lemma \ref{longest2} can be chosen so that $e^{-\mathrm{i}t\mathbf{H}}f(\Omega_{T}^{c})=f(\Omega_{T}^{c})$.

We define $\Sigma=\Sigma_{1}\cap\Sigma_{2}$, where $\Sigma_{1}$ is the set of all $g\in\mathcal{S}'(\mathbb{R}^{2})$ so that (\ref{nls2}) (with initial data $u(0)=g$) has a unique solution\footnote[6]{It is a bit vague to say $u$ is a ``solution'' when $g$ is only a distribution; but since we are considering $\Sigma_{2}$ also, we can assume here $e^{-\mathrm{i}t\mathbf{H}}g\in L_{t,x}^{q}$ on any finite time interval, for appropriate $q$, and then the definition of $\Sigma_{1}$ becomes rigorous.} $u$ on $\mathbb{R}$ that belongs to $e^{-\mathrm{i}t\mathbf{H}}f(\omega)+\mathcal{X}^{\sigma,b,T}$ for all $T>0$. This has full $\mu$ measure due to the global well-posedness part of Theorem \ref{main}. Also $\Sigma_{2}$ is defined to be $\Sigma_{2}=f(\liminf_{i\to\infty}\Omega_{\gamma2^{-i}}^{c})+\mathcal{H}^{\sigma}$, and this also has full $\mu$ measure for small enough $\gamma$ due to our control on $\mathbb{P}(\Omega_{T})$. Clearly $\Sigma$ has full $\mu$ measure, and $\Phi_{t}$ is uniquely defined on $\Sigma$. If we can prove $\Phi_{t}(\Sigma)\subset\Sigma$, then they obviously form a (measurable) one-parameter group. Clearly $\Phi_{t}(\Sigma)\subset\Sigma_{2}$ since for a solution $u$ we have $u(t)\in e^{-\mathrm{i}t\mathbf{H}}u(0)+\mathcal{H}^{\sigma}$. To prove $\Phi_{t}(\Sigma)\subset\Sigma_{1}$, we only need to prove that if $u$ is a solution of (\ref{nls22}) with $u(0)\in\Sigma_{2}$, then it is automatically unique. Since all $u(t)\in\Sigma_{2}$, by bootstrap arguments we only need to prove short time uniqueness. Write $u(0)=f(\omega)+h$ with $\|h\|_{\mathcal{H}^{\sigma}}=A$ and $\omega\not\in \Omega_{c2^{-i}}$ for all large enough $i$. Repeating the extension argument in Proposition \ref{appppp}, we see for $i$ large enough depending on $A$, $\omega\not\in\Omega_{c2^{-i}}$ and the solution is unique for $|t|\leq c2^{-i}$. This proves the existence of $\Sigma$.

Now we only need to prove that for each measurable set $E\subset\Sigma$ and $t\in\mathbb{R}$, we have \begin{equation}\label{final}\nu(\Phi_{t}(E))\geq \nu(E).\end{equation} By a limiting argument we may assume\[E\subset\Pi=\Sigma_{1}\cap\bigg(\{h:\|h\|_{\mathcal{H}^{\sigma}}\leq A\}+\bigcap_{i\geq i_{0}}f(\Omega_{c2^{-i}}^{c})\bigg).\]We may also assume that $|t|$ is small enough.

Let $T$ be small enough depending on $i_{0}$ and $A$, we only need to prove (\ref{final}) for $|t|\leq T$. Define $\Psi(g)=u-e^{-\mathrm{i}t\mathbf{H}}g$, where $u$ is the solution to (\ref{nls2}) with initial data $g$, and consider the mapping \[\Psi_{1}:\,\,\Pi\to\mathcal{X}^{\sigma,b,T}\times\mathbb{C}^{\infty}:\,\,\,\,g\mapsto(\Psi(g),(\langle g, e_{k}\rangle)_{k\geq 0}),\]where in $\mathbb{C}^{\infty}$ we use the standard metric. This mapping is clearly injective (thus it induces a metric on $\Pi$) and, as will be explained in Remark \ref{explain}, its image is a Borel set of the product space (denoted by $Y$). By a theorem in measure theory (see \cite{Ha50}), the finite Borel measure $\nu\circ \Psi_{1}^{-1}$ on the complete separable metric space $Y$ is regular. For each measurable set $E\subset\Pi$ we can find a compact set $K\subset \Psi_{1}(E)$ so that $(\nu\circ \Psi_{1}^{-1})(\Psi_{1}(E)-K)<\epsilon$, thus $\Psi_{1}^{-1}(K)\subset E$ is compact in the induced metric and $\nu(E-\Psi_{1}^{-1}(K))<\epsilon$. Therefore, we only need to prove (\ref{final}) for compact sets $E\subset\Pi$. By the invariance of $\nu_{2^{k}}^{\circ}$ under the solution flow $\Phi_{2^{k},t}$ we have\[\nu_{2^{k}}\big(\big\{g:g_{2^{k}}^{\circ}=\Phi_{2^{k},t}(h_{2^{k}}^{\circ}),\,\,\,h\in E\big\}\big)\geq \nu_{2^{k}}(E).\] Let $k\to\infty$, noticing that the total variance of $\nu_{2^{k}}- \nu$ tends to zero, we only need to prove that\[\limsup_{k\to\infty}\big\{g:g_{2^{k}}^{\circ}=\Phi_{2^{k},t}(h_{2^{k}}^{\circ}),\,\,\,h\in E\big\}\subset\Phi_{t}(E).\]
Now suppose that for a subsequence $k_{j}\uparrow\infty$, we have $g_{2^{k_{j}}}^{\circ}=\Phi_{2^{k_{j}},t}((h_{k_{j}})_{2^{k_{j}}}^{\circ})$, and by compactness, assume $h_{k_{j}}\to h$ with respect to the induced metric. We are going to prove $g=\Phi_{t}(h)$.

First of all, we have \begin{equation}\label{final2}\lim_{k\to\infty}\big\|\Phi_{2^{k},t}(h_{2^{k}}^{\circ})-\Phi_{t}(h)+e^{-\mathrm{i}t\mathbf{H}}h_{2^{k}}^{\perp}\big\|_{\mathcal{H}^{\sigma}}=0,\end{equation} uniformly for $|t|\leq T$ and $h\in E$. In fact, if $T$ is small enough, we may assume $h=h_{1}+h_{2}$, where $h_{1}\in f(\Omega_{T'}^{c})$, $2T\leq T'\leq 4T$, and $\|h_{2}\|_{\mathcal{H}^{\sigma}}\leq A$. Since $T'$ is small enough depending on $A$, we can almost repeat\footnote[7]{Actually we do not have the \emph{a priori} bound on the nonlinear part of truncated equations, but since $h_{1}\in\Omega_{T'}$ with $T'$ small depending on $A$, it is not hard to get this from scratch.} the proof of Proposition \ref{appppp} to get that the $\mathcal{X}^{\sigma,b,T'}$ norm tends to $0$. Since the $\mathcal{X}^{\sigma,b,T'}$ norm is not less than the spacial $\mathcal{H}^{\sigma}$ norm at time $t$, (\ref{final2}) follows.

From (\ref{final2}) we get\[\lim_{j\to\infty}\big\|g_{2^{k_{j}}}^{\circ}-(\Phi_{t}(h_{k_{j}}))_{2^{k_{j}}}^{\circ}\big\|_{\mathcal{H}^{\sigma}}=0,\] and we only need to prove \begin{equation}\label{final3}\lim_{k\to\infty}\big\|(\Phi_{t}(h_{k_{j}}))_{2^{k_{j}}}^{\circ}-\Phi_{t}(h)_{2^{k_{j}}}\big\|_{\mathcal{H}^{\sigma}}=0.\end{equation} But since $h_{k_{j}}\to h$ with respect to the induced metric, we only need to prove that $\|(h_{k_{j}})_{2^{k_{j}}}^{\circ}-h_{2^{k_{j}}}^{\circ}\|_{\mathcal{H}^{\sigma}}\to 0$.

For $i\geq j$ we have\[\big(\Phi_{2^{k_{i}},t}((h_{k_{i}})_{2^{k_{i}}}^{\circ})\big)_{2^{k_{j}}}^{\circ}=g_{2^{k_{j}}}^{\circ}=\Phi_{2^{k_{j}},t}((h_{k_{j}})_{2^{k_{j}}}^{\circ}),\] and by using (\ref{final2}) once more we see that\[(\Phi_{t}(h_{k_{i}}))_{2^{k_{j}}}^{\circ}=(\Phi_{t}(h_{k_{j}}))_{2^{k_{j}}}^{\circ}+o(1),\] as $i\geq j\to\infty$. Again using that $h_{k_{j}}\to h$, we deduce\begin{equation}\label{final999}\lim_{i\geq j\to\infty}\|(h_{k_{i}}-h_{k_{j}})_{2^{k_{j}}}^{\circ}\|_{\mathcal{H}^{\sigma}}= 0.\end{equation} In particular, we see that $\lim_{i\to\infty}(h_{k_{i}})_{2^{k_{j}}}^{\circ}$ exists in $\mathcal{H}^{\sigma}$ for each $j$. By the definition of the metric, this limit must be $h_{2^{k_{j}}}^{\circ}$. Therefore we get \begin{equation}\label{final1000}\lim_{i\to\infty}\|(h-h_{k_{i}})_{2^{k_{j}}}^{\circ}\|_{\mathcal{H}^{\sigma}}=0.\end{equation} Combining (\ref{final999}) with (\ref{final1000}), we finally see that $\lim_{j\to\infty}\|(h_{k_{j}})_{2^{k_{j}}}^{\circ}-h_{2^{k_{j}}}^{\circ}\|_{\mathcal{H}^{\sigma}}=0$. This completes the proof of Theorem \ref{main}.
\end{proof}

\begin{remark}\label{explain}
To show that $\Psi_{1}(\Pi)$ is a Borel set in the product metric space, we only need to show that $\Psi$ is injective, $\Psi(\Pi)$ is a Borel set in $\mathcal{X}^{\sigma,b,T}$, and the map $\Psi^{-1}:\Psi(\Pi)\to\Pi$ is Borelian. To this end we notice\begin{equation}\label{inj}\Psi(g)=-\mathrm{i}\int_{0}^{t}e^{-\mathrm{i}(t-s)\mathbf{H}}\big(|u(s)|^{p-1}u(s)\big)\,\mathrm{d}s,\end{equation} where $u=u(g)$ is the solution map of (\ref{nls1}), and\footnote[8]{Here we also require $g\in \mathcal{H}^{-\epsilon}$ for appropriate $\epsilon$, so that $u\in\mathcal{C}(\mathbb{R},\mathcal{H}^{-\epsilon})$ in which $u(0)$ makes sense.} $g=u(0)$. Then we can decompose $\Psi$ as \[\Psi:g\mapsto u(g)\mapsto |u(g)|^{p-1}u(g)\mapsto \Psi(g),\] and see that at each step the mapping is injective, and the image of any Borel set is again Borelian (for example, the set $u(\Pi)$ can be characterized as the set of all $u$ so that $u-e^{-\mathrm{i}t\mathbf{H}}u(0)\in\mathcal{X}^{\sigma,b}$, $u$ satisfies equation (\ref{nls1}), and that $u(0)\in\Pi$, so it is Borelian). Hence the claim.
\end{remark}

\appendix

\section{Typical regularity on the support of $\mu$}\label{count}
In this appendix we shall prove that, if $\sigma\geq\frac{1}{2}$, then almost surely, $\mathbf{H}^{\frac{\sigma}{2}}f(\omega)$ is not a (locally integrable) function. More precisely, almost surely in $\mathbb{P}$, we have\begin{equation}\label{count01}\psi\cdot \mathbf{H}^{\frac{\sigma}{2}}f(\omega)\not\in L^{1}(\mathbb{R}^{2}),\end{equation} for \emph{all} smooth compactly supported $\psi$ that is not identically zero. 

To prove this, first notice that we can find a countable number of $\psi_{j}$ such that each is compactly supported and equals $1$ on some annular region $a<|x|<b$, and for any other $\psi$ there exists $\eta\in L^{\infty}$ and $j$ so that $\psi_{j}=\psi\cdot\eta$. So we only need to consider a fixed $\psi_{j}$ (which we write $\psi$ below) and assume it equals $1$ for $a<|x|<b$. Here we use an asymptotic formula of $\mathcal{L}_{k}^{0}$ proved in \cite{Er60}:\begin{equation}\label{count02}\mathcal{L}_{k}^{0}(z)=\frac{1}{\sqrt{2\pi}}(\nu z)^{-\frac{1}{4}}\cos\theta+O(\nu^{-\frac{3}{4}}),\end{equation} where $a^{2}< z<b^{2}$ and $\nu=4k+2$ is large, and \[\theta=\frac{\nu(\phi+\sin\phi)-\pi}{4},\,\,\,\,\,\,\,\,\phi=\cos^{-1}\frac{\nu-2z}{\nu}.\] From (\ref{count02}) we easily deduce that\[\mathcal{L}_{k}^{0}(z)=\frac{1}{\sqrt{2\pi}}(\nu z)^{-\frac{1}{4}}\cos(\sqrt{\nu z}-\frac{\pi}{4})+O(\nu^{-\frac{3}{4}}),\] and hence for each $k$\begin{equation}\label{count035}\|e_{k}\psi\|_{L^{1}}\gtrsim\int_{a^{2}< z<b^{2}}|\mathcal{L}_{k}^{0}(z)|\,\mathrm{d}z\gtrsim \nu^{-\frac{1}{4}}.\end{equation} Now we define the Gaussian random variable \[h_{M,N}(\omega)=\sum_{k=0}^{M}(4k+2)^{\frac{\sigma-1}{2}}g_{k}(\omega)\eta\big(\frac{\mathbf{H}}{N^{2}}\big)(e_{k}\psi),\] whose range lies in a finite dimensional space, and use Lemma \ref{fnq} to get the lower bound $\mathbb{P}(\|h_{M,N}(\omega)\|_{L^{1}}\geq cE_{M,N})\geq\frac{1}{10}$ with some absolute constant $c$, where\setlength\arraycolsep{2pt}
\begin{eqnarray}
E_{M,N}=\mathbb{E}(\|h_{M,N}(\omega)\|_{L^{1}})&=&\int_{\mathbb{R}^{2}}\mathbb{E}\bigg(\bigg|\sum_{k= 0}^{M}(4k+2)^{\frac{\sigma-1}{2}}g_{k}(\omega)e_{k,N}(x)\bigg|\bigg)\,\mathrm{d}x\nonumber\\
&\sim&\int_{\mathbb{R}^{2}}\bigg(\sum_{k=0}^{M}(4k+2)^{\sigma-1}|e_{k,N}(x)|^{2}\bigg)^{\frac{1}{2}}\,\mathrm{d}x\nonumber\\
&\gtrsim &\bigg(\sum_{k=0}^{M}(4k+2)^{\sigma-1}\|e_{k,N}\|_{L^{1}}^{2}\bigg)^{\frac{1}{2}},\nonumber
\end{eqnarray} and $e_{k,N}=\eta(N^{-2}\mathbf{H})(e_{k}\psi)$. Now for fixed $N$, we let $M\to\infty$ to get \[h_{M,N}\to\eta(N^{-2}\mathbf{H})(\mathbf{H}^{\frac{\sigma}{2}}f(\omega)\cdot\psi)\] in $\mathcal{S}$ almost surely, since for fixed $n$ (say $n\leq 3N$), the inner product $\langle e_{n},e_{k}\psi\rangle$ is rapidly decreasing in $k$ (using integration by parts). In particular we have a.s. $L^{1}$ convergence and hence (by taking upper limit of a sequence of sets)\[\mathbb{P}(\|\eta(N^{-2}\mathbf{H})(\mathbf{H}^{\frac{\sigma}{2}}f(\omega)\cdot\psi)\|_{L^{1}}\geq cE_{N})\geq\frac{1}{10},\] where\[E_{N}=\liminf_{M\to\infty}E_{M,N}\gtrsim \bigg(\sum_{k=0}^{\infty}(4k+2)^{\sigma-1}\|e_{k,N}\|_{L^{1}}^{2}\bigg)^{\frac{1}{2}}.\] By the uniform boundedness of $\eta(N^{-2}\mathbf{H})$, we know $\eta(N^{-2}\mathbf{H})g\to g$ in $L^{1}$ for any $g\in L^{1}$, so we have\[\liminf_{N\to\infty}E_{N}\gtrsim \bigg(\sum_{k=0}^{\infty}(4k+2)^{\sigma-1}\|e_{k}\psi\|_{L^{1}}^{2}\bigg)^{\frac{1}{2}}\gtrsim \bigg(\sum_{k=0}^{\infty}(4k+2)^{\sigma-\frac{3}{2}}\bigg)^{\frac{1}{2}}=\infty,\] due to (\ref{count035}). Now we take another upper limit, and see that with probability $\geq\frac{1}{10}$, we have\begin{equation}\label{count04}\limsup_{N\to\infty}\|\eta(N^{-2}\mathbf{H})(\mathbf{\mathbf{H}}^{\frac{\sigma}{2}}f(\omega)\cdot\psi)\|_{L^{1}}=\infty.\end{equation} Now (\ref{count04}) implies (\ref{count01}), again because of the uniform boundedness of $\eta(N^{-2}\mathbf{H})$ on $L^{1}$. Therefore we have proved that (\ref{count01}) holds with positive probability. Since it is clearly a tail event (because $e_{k}\cdot\psi$ themselves are Schwartz functions), it must hold with probability one.

\end{document}